\documentclass{amsart}
\title{Forcing with overlapping supercompact extenders}
\date{\today}
\author{Sittinon Jirattikansakul}
\usepackage{graphicx}
\usepackage[export]{adjustbox}
\graphicspath{ {./images/} }
\usepackage{amssymb}
\usepackage{amsmath}
\usepackage{float}
\usepackage{caption}
\usepackage{hyperref}
\usepackage[dvipsnames]{xcolor} 
\usepackage{tikz-cd}
\usepackage{enumitem}
\DeclareMathOperator{\dom}{dom}

\DeclareMathOperator{\rge}{rge}

\DeclareMathOperator{\cf}{cf}
\DeclareMathOperator{\crit}{crit}
\DeclareMathOperator{\OB}{OB}

\DeclareMathOperator{\mc}{mc}

\DeclareMathOperator{\supp}{supp}

\DeclareMathOperator{\Ult}{Ult}
\theoremstyle{plain}
\newtheorem{thm}{Theorem}[section]
\newtheorem{lemma}[thm]{Lemma}
\newtheorem{prop}[thm]{Proposition}
\newtheorem{claim}{Claim}[thm]

\newtheorem{coll}[thm]{Corollary}

\theoremstyle{definition}
\newtheorem{defn}[thm]{Definition}
\newtheorem{rmk}[thm]{Remark}
\usepackage[dvipsnames]{xcolor}
\newenvironment{claimproof}[1]{\par\noindent{Proof}\space#1}{\hfill $\blacksquare$}
\tikzset{close/.style={near start,outer sep=-1pt}}

\begin{document}

\begin{abstract}

We build a supercompact version of the forcing defined in \cite{gitik2019}.
For each singular cardinal in the ground model with any fixed cofinality, which is a limit of supercompact cardinals, it is possible to force so that the size of the powerset of the singular cardinal is arbitrarily large, while preserving the singular cardinal.
An important feature of this forcing is that it is possible to define the forcing so that the successor of the singular cardinal is collapsed, but all the cardinals above it are preserved.

\end{abstract}

\maketitle

\section{introduction}

The {\em Singular Cardinal Hypothesis} (SCH) states that if $\kappa$ is a singular cardinal, then $2^\kappa$ has the smallest possible value under all cardinal arithmetic provable in ZFC.
In particular, if $\kappa$ is singular strong limit, then $2^\kappa=\kappa^+$.
Gitik \cite{gitik2019} built a forcing from overlapping extenders witnessing strongness in which SCH fails for a singular cardinal of arbitrary fixed cofinality which is singular in a ground model.
Let $\kappa$ be such a cardinal and $\lambda>\kappa$ be regular.
It is possible to blow up the value of $2^\kappa$ to be $\lambda$.
The large cardinal assumption used to build the forcing in \cite{gitik2019} is the existence of an increasing sequence of strong cardinals $\langle \kappa_\alpha: \alpha<\cf(\kappa) \rangle$ whose supremum is $\kappa$, where each cardinal $\kappa_\alpha$ carries an extender witnessing strongness of $\kappa_\alpha$, and for each $\alpha<\beta$, $E_\alpha \in \Ult(V,E_\beta)$.
This hypothesis is proven to be weaker than the existence of a Woodin cardinal \cite{jir2021}
With some interleaved supercompact cardinals, the forcing also gives some interesting combinatorial results on successors of singular cardinals, e.g. a stationary reflection principle and a failure of the approachability property (see \cite{gitik2021nonap}, \cite{benhayutunger2021}).

In this paper, we build a forcing which is a supercompact version of \cite{gitik2019}.
We prepare an increasing sequence of supercompact cardinals $\langle \kappa_\alpha: \alpha<\eta \rangle$ such that if $\lambda> \sup_{\alpha<\eta} \kappa_\alpha$ is regular, each $\kappa_\alpha$ carries a long extender $E_\alpha$ witnessing $\kappa_\alpha$ being $\lambda$-supercompact, and $j_\alpha:V \to \Ult(V,E_\alpha)$ is such that $j_\alpha(\lambda) \geq \lambda^{++}$ (we can prepare the value $j_\alpha(\lambda)$ to be arbitrarily high).
Then we define a $\lambda^{++}$-c.c. forcing such that in an extension, $2^{\sup_\alpha{\kappa_\alpha}} \geq \sup_\alpha |j_\alpha(\lambda)|$. Below $\sup_\alpha \kappa_\alpha$, most of the cardinals will be preserved.
The only cardinals which are collapsed in the extension belong to one of the intervals of the form $(\sup_{\gamma<\beta} \kappa_\gamma,\lambda_\beta^+)$ with $\beta<\eta$ limit, for some $\lambda_\beta<\kappa_\beta$, or the interval $(\sup_{\alpha<\eta}\kappa_\alpha,\lambda^+)$.

We make some notations for our convenience.
For a sequence $\vec{x}=\langle x_\alpha: \alpha<\eta \rangle$, let $\vec{x} \restriction \beta=\langle x_\alpha: \alpha<\beta \rangle$ and  $\vec{x} \setminus \beta=\langle x_\alpha: \alpha \geq \beta \rangle$.
If $X$ is a set of sequence, let $X \restriction \beta= \{\vec{x} \restriction \beta: \vec{x} \in X\}$.
Define $X \setminus \beta$ is a corresponding fashion
For functions $f,g$ and $h$, define $f \circ g \circ h$ as a function with domain $\{ x \in \dom(h): h(x) \in \dom(g) \text{ and } g(h(x)) \in \dom(f)\}$, and for $x \in \dom(f \circ g \circ h)$, $f \circ g \circ h=f(g(h(x)))$.
If $X$ is a set of functions and $g$ and $h$ are functions, define $g \circ X \circ h$ as $\{g \circ f \circ h: f \in X\}$.
If $X$ is a set of functions and $d$ is a set, define $X \restriction d=\{f \restriction d: f \in X\}$, where $f \restriction d= f \restriction d \cap \dom(f)$.
If $f$ and $g$ are functions and $\dom(g) \subseteq \dom(f)$, define $f \oplus g$, {\em $f$ overwritten by $g$}, as a function $h$ with $\dom(h)=\dom(f)$, $h(x)=g(x)$ if $x \in \dom(g)$, otherwise $g(x)=f(x)$.
We point out the locations where we give non-standard notations which we use throughout the paper in order to facilitate the readers: the paragraph after Lemma \ref{woodinscc2} and the last paragraph in Section \ref{analysisofextenders}.
We also recall some important values timely.
The organization of this paper is as follows, where we note that from Section \ref{analysisofextenders} - \ref{scaleanalysis} we assume GCH.

\begin{itemize}

\item In Section \ref{supercompactextenders} we make an analysis of a characterization of extenders which capture supercompactness.

\item In Section \ref{woodinizedsupercompact} the definition of a Woodinized supercompact cardinal is introduced. 
Then we use a variation of the characterization of Woodinizedsupercompactness to build a sequence of extenders witnessing supercompactness. 
It will be served as an initial setting to build the forcing.

\item In Section \ref{analysisofextenders} we analyze the extenders which are built from Section \ref{woodinizedsupercompact}.
We define {\em domains} and {\em objects} which were first introduced by Gitik and Merimovich \cite{merimovich}.
Those concepts are important features in our forcing.

\item In each of Section \ref{forcing1extender}, \ref{forcing2extenders}, \ref{forcingomegaextender}, and \ref{forcinganyextender}, we define forcings from sequences of extenders whose sequences are of different lengths.
This is due to the fact that the proofs of the Prikry property and the strong Prikry property of a forcing rely on inductions on the lengths of the sequences of extenders.
In particular, to prove that the properties hold for a forcing defined from a sequence of exntenders whose sequence has a certain length,  the Prikry property and the strong Prikry property of the forcings defined from sequences of extenders of shorter lengths are required. 
We prove the Prikry property in Section \ref{forcing1extender}, \ref{forcing2extenders}, and \ref{forcingomegaextender}, and the sketch of the proof of the strong Prikry property is provided in \ref{forcingomegaextender} only.
The forcing in Section \ref{forcinganyextender} has the most general form.

\item In Section \ref{cardinalpreserve} we determine the cardinals which are preserved and collapsed.

\item In Section \ref{blowinguppowersets} we show that our forcing breaks the Singular Cardinal Hypothesis on the singular cardinals which are suprema of the supercompact cardinals.

\item In Section \ref{scaleanalysis} we make some analysis of the scales which are derived naturally from our forcing.

\end{itemize}

We finally draw a conclusion in Section \ref{conclusion}.
 We assume that the readers are familiar with forcings and extenders.

\section{supercompact extenders}
\label{supercompactextenders}

In this section we determine the lengths of an extender which captures supercompactness of an embedding.
For the detailed account on extenders, see Chapter 26 of \cite{higherinf}.

\begin{defn} \label{<scc}

Let $\kappa<\lambda$ be cardinals.
$\kappa$ is ${<}\lambda$-supercompact if there is an elementary embedding $j:V \to M$ such that $\crit(j)=\kappa$, $j(\kappa)>\lambda$, and ${}^{< \lambda} M \subseteq M$.

\end{defn}

Some set theorists define $\kappa$ to be ${<}	\lambda$-supercompact as $\kappa$ being $\gamma$-supercompact for $\gamma<\lambda$.
This is strictly weaker than what we define in Definition \ref{<scc}.
Others may allow the value $j(\kappa)$ in Definition \ref{<scc} to be at least $\lambda$, although such $\kappa$ in the definition with $j(\kappa)=\lambda$  characterizes almost hugeness of $\kappa$.

We will make an analysis of the extender derived from an elementary embedding $j$ witnessing some supercompactness.

Let $j:V \to M$ be an embedding witnessing $\kappa$ being ${<}\lambda$-supercompact. 
Fix $\xi \geq \lambda$. We give an overview of the structure of the $(\kappa,\xi)$-extender $E$ derived from $j$.
For each $a \in [\xi]^{<\omega}$, let $\lambda_a$ be the least ordinal such that $\max(a) < j(\lambda_a)$. Define 

\begin{center}
$A \in E_a$ if $a \in j(A)$.
\end{center}
$E_a$ is a $\kappa$-complete ultrafilter on $[\lambda_a]^{|a|}$ (note that $E_{\{\kappa\}}$ is normal).
Let $j_a:V \to M_a=\Ult(V,E_a)$.
The map $k_a([f]_{E_a})=j(f)(a)$ is an elementary embedding from $M_a$ to $M$.

If $a \subseteq b \in [\xi]^{<\omega}$, we enumerate $b$ in  increasing order as $\alpha_0< \dots<\alpha_{|b|-1}$. Assume the increasing enumeration of $a$ is $\alpha_{n_0}< \alpha_{n_1}\dots < \alpha_{n_{|a|-1}}$. There is a natural projection map $\pi_{b,a}:[\lambda_b]^{|b|} \to [\lambda_a]^{|a|}$ defined by 
  
  \begin{center}
  
  $\pi_{b,a}(\{\beta_0, \beta_1, \dots , \beta_{|b|-1}\})=\{\beta_{n_0},\beta_{n_1},\dots,\beta_{n_{|a|-1}}\}$ where $\beta_0< \beta_1< \dots <\beta_{|a|-1}$.
  
   \end{center} 
   
   This induces an elementary embedding $i_{a,b}:M_a \to M_b$ by the map 
   
   \begin{center}
   
   $[f]_{E_a} \mapsto [f \circ \pi_{b,a}]_{E_b}$.
   
   \end{center} 
   
   The family
   
   \begin{center}
   
    $\langle M_a ,i_{a,b}: a \subseteq b \in [\xi]^{<\omega} \rangle$
    
    \end{center}
    
     forms a directed system.
     Let $M_E$ be its direct limit. 
     We can form, for all $a \in [\xi]^{<\omega}$, elementary embeddings
     
     \begin{center}
     
     $j_{a,E} : M_a \to M_E$
     
     \end{center} 
     
     and 
     
     \begin{center}
     
     $j_E : V \to M_E$
     
     \end{center}
     
      such that 
      
      \begin{center}
      
      $j_E=j_{a,E} \circ j_a$ for all $a$.
      
      \end{center}
      
  Note that $M_E$ is isomorphic to an elementary submodel of $M$ with the factor map $k:M_E \to M$ defined as follows:
  For $x \in M_E$, $x=j_{a,E}([f]_{E_a})$ for some $a$,
  define $k(x)=j(f)(a)$.
  Hence, $M_E$ is well-founded, we identify $M_E$ as its transitive collapse, and assume that $k$ is the inverse of the transitive collapse from $\rge(k)$ onto $M_E$.
  
  \begin{center}

\[\begin{tikzcd}
	V &&&& M \\
	\\
	\\
	{M_a} &&&& M_E \\
	&& {M_b}
	\arrow[close,"j" {pos=0.5}, from=1-1, to=1-5]
	\arrow[close,"{j_{a,E}}"{pos=0.5}, from=4-1, to=4-5]
	\arrow[close,"k"'{pos=0.5}, from=4-5, to=1-5]
	\arrow[close,"{j_E}"{pos=0.3}, from=1-1, to=4-5]
	\arrow[close,"{i_{a,b}}"'{pos=0.5}, from=4-1, to=5-3]
	\arrow[close,"{k_b}"{pos=0.65}, from=5-3, to=1-5]
	\arrow[close,"{j_b}"{pos=0.35}, from=1-1, to=5-3]
	\arrow[close,"{k_a}"{pos=0.7}, from=4-1, to=1-5]
	\arrow[close,"{j_a}"'{pos=0.5}, from=1-1, to=4-1]
	\arrow[close,"{j_{b,E}}"'{pos=0.5}, from=5-3, to=4-5]
\end{tikzcd}\]

\end{center}

We prove a series of basic properties regarding the map $j_E$.

\begin{prop}\label{critpoint}
$\crit(k) \geq \xi$. If $\xi=j(\beta)$ for some $\beta$, then $\crit(k)>\xi$.
In particular, if $\xi>\lambda$, then $j_E(\kappa)>\lambda$.
\end{prop}

\begin{proof}

For $\gamma < \xi$, $\gamma=j(id)(\gamma) \in \rge(k)$, so $\gamma \in \rge(k)$.
Since $k$ is the inverse of the transitive collapse, $\xi \subseteq M_E$ and $k \restriction \xi = id$.
If $\xi=j(\beta)$ for some $\beta<\xi$, then since $j_E(\beta)=j(\beta)$, $\xi=k \circ j_E(\beta)=k \circ j(\beta)=k(\xi)$.
If $\xi=j(\xi)$, then $\xi=j(\xi)=k \circ j_E(\xi) \geq k(\xi) \geq \xi$. 

\end{proof}

The following proposition is not used in this paper, but is worth being mentioned.

\begin{prop}

If there is a function $s:\kappa \to \kappa$ such that $j(s)(\kappa)=\lambda$, then $\crit(k)>\lambda$.

\end{prop}

\begin{proof}

Since $\xi \geq \lambda$, $\crit(k) \geq \lambda$.
If $\lambda=j(s)(\kappa)=k(j_E(s)(\kappa))$, $\lambda \in \rge(k)$, and so $\crit(k)>\lambda$.

\end{proof}

\begin{prop}\label{rep1}

$M_E=\{j_E(f)(a): a \in [\xi]^{<\omega} \text{ and } f:[\lambda_a]^{|a|} \to V\}$ and $k(j_E(f)(a))=j(f)(a)$.

\end{prop}

\begin{proof}

For $a \in [\xi]^{<\omega}$, $k(a)=a$, so $k^{-1}(j(f)(a))=j_E(f)(a)$.

\end{proof}

To determine if $M_E$ is closed under ${<}\lambda$-sequences, it is enough to determine if $\rge(k)$ is closed under ${<}\lambda$-sequences.

\begin{prop}
\label{capture<lambdaseq}
Assume $\xi \geq \sup j[\lambda]$.
The following are equivalent:

\begin{enumerate}

\item ${}^{<\lambda} \rge(k) \subseteq \rge(k)$.

\item ${}^{<\lambda} \xi \subseteq \rge(k)$ and for $\gamma<\lambda$, $j[\gamma] \in \rge(k)$.

\end{enumerate}

\end{prop}

\begin{proof}

The forward direction is trivial.
We prove the backward direction.
By some simplification, assume that elements in ${}^{<\lambda} \rge(k)$ are of the form $\vec{x}=\langle j(f_\alpha)(\gamma_\alpha): \alpha<\gamma \rangle$, where $\gamma_\alpha<\xi$ and $\gamma<\lambda$.
Let $\vec{f}=\langle f_\alpha : \alpha<\gamma \rangle \in V$ and  $\vec{\gamma}=\langle \gamma_\alpha: \alpha<\gamma \rangle$.
Then the map $\alpha \mapsto$ the $\alpha$-th element of $j(\vec{f})(j[\gamma])$ computed at $\vec{\gamma}(\alpha)$ represents $x$.
By our assumption, $\vec{x} \in \rge(k)$.

\end{proof}

Similar proof shows that

\begin{prop}
\label{capturerhoseq}
Assume $\lambda=\rho^+$ for some cardinal $\rho$ and $\xi \geq \sup j[\rho]$.
The following are equivalent:

\begin{itemize}

\item ${}^\rho \rge(k) \subseteq \rge(k)$.

\item ${}^\rho[\xi] \subseteq \rge(k)$ and $j[\rho] \in \rge(k)$.

\end{itemize}

\end{prop}

\begin{proof}

Similar to the proof of Proposition \ref{capture<lambdaseq}.

\end{proof}

\begin{rmk}

In Proposition \ref{capture<lambdaseq} (Proposition \ref{capturerhoseq}), it might be possible that in some circumstances, e.g. $j(\lambda)$ ($j(\rho)$) is very high, there is $\xi<j[\lambda]$ ($\xi<j[\rho]$) which captures ${<}\lambda$-supercompactness ($\rho$-supercompactness). 

\end{rmk}

\begin{lemma}
\label{lengthcapturescc}
Assume GCH.

\begin{enumerate}
\label{lengthcap<lambdascc}
\item \label{<lambdaitem1}  There is a sequence $\Gamma = \langle x_\alpha : \alpha \in ON \rangle$ of elements $[ON]^{<\kappa}$ such that if $\gamma$ is a cardinal, $cf(\gamma) \geq \kappa$, then $\Gamma \restriction \gamma \supseteq [\gamma]^{<\kappa}$. 

\item \label{<lambdaitem2} Let $s: \kappa \to \kappa$. There is a sequence $\Gamma = \langle x_\alpha : \alpha \in ON \rangle$ of $[ON]^{<\kappa}$ such that if $\gamma \geq \kappa$ is a cardinal, and $\alpha$ is such that $\cf(\gamma) \geq s(\alpha)$, then $\Gamma \restriction \gamma \supseteq [\gamma]^{<s(\alpha)}$.

\item Let $j:V \to M$ witness $\kappa$ being ${<}\lambda$-supercompact. 
Then

\begin{itemize}

\item if $\xi$ is a cardinal in $M$, $\xi \geq \sup j[\lambda]$, and $\cf^M(\xi) \geq j(\kappa)$, then the $(\kappa,\xi)$-extender $E$ derived from $j$ preserves ${<}\lambda$-supercompactness of $\kappa$.

\item if there is a function $s:\kappa \to \kappa$ such that $j(s)(\kappa)=\lambda$, then for an ordinal $\xi$ which is a cardinal in $M$, $\xi \geq \sup j[\lambda]$, and $cf^M(\xi) \geq \lambda$, the $(\kappa,\xi)$-extender $E$ derived from $j$ preserves the ${<}\lambda$-supercompactness of $\kappa$.

\end{itemize}

\end{enumerate}

\end{lemma}

\begin{proof}

\begin{enumerate}

\item We proceed by induction that at each cardinal $\gamma$ with $\cf(\gamma) \geq \kappa$, the length of the sequence being built so far will have total length $\gamma$. 
We abusively write $\Gamma \restriction \gamma$ for the sequence of length $\gamma$ we have built so far.
The first cardinal is $\kappa$.
By GCH, $|\kappa^{<\kappa}|=\kappa$, so that building a sequence containing $[\kappa]^{<\kappa}$ of length $\kappa$ is trivial.
Suppose $\gamma$ is a cardinal, $\cf(\gamma) \geq \kappa$ and the $\Gamma \restriction \gamma^\prime$ is built for $\gamma^\prime<\gamma$.
If $\gamma=(\gamma^\prime)^+$, then the sequence had been built so far has length $\gamma^\prime$.
Since $|\gamma \setminus \gamma^\prime|=\gamma$ and $|\gamma^{<\kappa}|=\gamma$, we can list the sequence so that $\Gamma \restriction [\gamma^\prime,\gamma) \supseteq [\gamma]^{<\kappa}$.
If $\gamma$ is a limit cardinal and $\cf(\gamma) \geq \kappa$, then $[\gamma]^{<\kappa}=\cup_{\gamma^\prime<\gamma}[\gamma^\prime]^{<\kappa}$, so the sequence that has been built so far takes care of the cardinal $\gamma$ already. Hence, the proof is done.

\item A similar argument as in (\ref{<lambdaitem1}) works.

\item Assume $j:V \to M$ witnesses $\kappa$ being ${<}\lambda$-supercompact. 

\begin{itemize}

\item Let $\Gamma$ be as in (\ref{<lambdaitem1}).
Let $\vec{x}:=j(\Gamma)\restriction \xi$. 
Then $\vec{x} \supseteq ([\xi]^{<j(\kappa)})^M \supseteq [\xi]^{<\lambda}$.
We see that for $x \in [\xi]^{<\lambda}$, we have that $x=\vec{x}(\gamma)$ for some $\gamma<\xi$.
Since $\xi \geq \sup j[\lambda]$, we can also consider the case where $x$ is $j[\alpha]$ for any $\alpha<\lambda$ as well, so the proof is done.

\item Let $\Gamma$ be as in ($2$).
Note that in this case, $\cf^M(\xi) \geq \lambda =j(s)(\kappa)$, so $j(\Gamma) \restriction \xi \supseteq ([\xi]^{<\lambda})^M=[\xi]^{<\lambda}$. 
The rest of the proof is the same as in the previous bullet.

\end{itemize}

\end{enumerate}

\end{proof}

The following lemma can be proved in a similar fashion.

\begin{lemma}

\begin{enumerate}
\label{lengthcaprhoscc}

Assume GCH.

\item Let $s: \kappa \to \kappa$. There is a sequence $\Gamma = \langle x_\alpha : \alpha \in ON \rangle$ of $[ON]^{\kappa}$ such that if $\gamma \geq \kappa$ is a cardinal, and $\alpha$ is such that $\cf(\gamma) > s(\alpha)$, then $\Gamma \restriction \gamma \supseteq [\gamma]^{s(\alpha)}$.

\item Let $j:V \to M$ witness $\kappa$ being $\rho$-supercompact.
Then

\begin{itemize}

\item if $\xi$ is a cardinal in $M$, $\xi \geq \sup j[\rho]$, and $\cf^M(\xi) \geq j(\kappa)$, then the $(\kappa,\xi)$-extender $E$ derived from $j$ preserves $\rho$-supercompactness of $\kappa$.

\item if there is a function $s:\kappa \to \kappa$ such that $j(s)(\kappa)=\rho$, then for an ordinal $\xi$ which is a cardinal in $M$ and $cf^M(\xi) > \rho$, the $(\kappa,\xi)$-extender $E$ derived from $j$ preserves $\rho$-supercompactness of $\kappa$.

\end{itemize}

\end{enumerate}

\end{lemma}

\begin{rmk}

The $GCH$ assumption in Lemma \ref{lengthcap<lambdascc} and \ref{lengthcaprhoscc} can be weakened. 
The cardinal arithmetic requirement for $\xi$ to capture supercompactness depends on the value $\xi$.

\end{rmk}

\begin{rmk}

\begin{itemize}

\item If $j$ witnesses $\kappa$ being $\rho$-supercompact (equivalently ${<}\rho^+$-supercompact), then definability of $\rho$ is equivalent to definability of $\lambda$.
More specifically, if $j(s)(\kappa)=\lambda$, then there is a function $t$ such that $t(\alpha)=s(\alpha)^+$ on a measure-one set, and hence $j(t)(\kappa)=\rho$.

\item For a ${<}\lambda$-supercompact embedding $j:V \to M$, a few instances of $\xi$ where the $(\kappa,\xi)$-extender derived from $j$ preserves the ${<}\lambda$-supercompactness are $((\sup(j[\lambda]))^+)^M$,  $j(\lambda)$, and if there is a function $s: \kappa \to \kappa$ such that $j(s)(\kappa)=\lambda$, then $\xi=\sup(j[\lambda])$ also works.

\item For a $\rho$-supercompact embedding $j:V \to M$ with $\lambda=\rho^+$, some values $\xi$ where the $(\kappa,\xi)$-extender $E$ derived from $j$ preserves the $\rho$-supercompactness are $((\sup(j[\rho])^+)^M$ and $j(\rho)$.

\end{itemize}

\end{rmk}

Note that if a we derive a $(\kappa,j(\lambda))$-extender $E$ from $j$, then $j_E(\lambda)=j(\lambda)$.
In the future, when we introduce an $E$, we may say that $E$ is a $(\kappa,j_E(\lambda))$ without referring $j$.
Observe that among elementary maps $j:V \to M$ witnessing ${<}\lambda$-supercompactness, the lowest possible  cardinal of $j(\lambda)$ computed in $V$ is $\lambda$.
Similarly, among elementary maps $j:V \to M$ witnessing $\rho$-supercompactness, the lowest possible cardinal of $j(\rho)$ in $V$ is $\rho^+$.
We end this section by giving a definition of an extender in our context.

\begin{defn}

$E$ is a {\em $(\kappa,\xi)$-extender} if $E$ is a $(\kappa,\xi)$-extender derived from some elementary embedding $j$.

\end{defn}

\section{woodinized supercompact cardinals}
\label{woodinizedsupercompact}

Woodin introduced the definition of a Woodinized supercompact cardinal, which has an application on stationary tower. Some set theorists refers to Woodinized supercompact cardinal as Woodin-for-supercompactness cardinal.
We follow the definition of a Woodinized supercompact cardinal from Definition 9.26 from Foreman's chapter in \cite{handbook}.
This section is an analogue of Section 7 in \cite{jir2021}.

\begin{defn}

A cardinal $\delta$ is {\em Woodinized supercompact} if for every function $f: \delta \to \delta$ there is a cardinal $\kappa<\delta$ such that $\kappa$ is a closure point of $f$, i.e. $f[\kappa] \subseteq \kappa$, and there is an elementary embedding $j:V \to M$ such that 

\begin{enumerate}

\item $\crit(j)=\kappa$.

\item ${}^{j(f)(\kappa)} M \subseteq M$.

\end{enumerate}

\end{defn}

The definition resembles to the definition of a Woodin cardinal in the following sense: recall that $\delta$ is Woodin if for $f: \delta \to \delta$ there is a cardinal $\kappa<\delta$ such that $f[\kappa] \subseteq \kappa$, and there is an elementary embedding $j:V \to M$ such that $\crit(j)=\kappa$ and $V_{j(f)(\kappa)} \subseteq M$.

\begin{defn}

Let $A$ be a set, and $\gamma$ be a cardinal. A cardinal $\kappa$ is $(\gamma,A)$-{\em supercompact} if there exists an elementary embedding $j: V \to M$ such that $\crit(j)=\kappa$, and 

\begin{enumerate}

\item $\gamma<j(\kappa)$.

\item ${}^{\gamma} M \subseteq M$.

\item $j(A) \cap V_\gamma=A \cap V_\gamma$.

\end{enumerate}

\end{defn}

The following lemma is proved in \cite{perlmutterthesis}. 
We include the proof for the sake of completeness.

\begin{lemma}
\label{woodinscc1}
Let $\delta$ be a cardinal.
The following are equivalent.

\begin{enumerate}

\item $\delta$ is a Woodinized supercompact cardinal.

\item Let $A \subseteq V_\delta$.
Then the collection of $\kappa$ satisfying the following condition is stationary: for all $\gamma<\delta$, $\kappa$ is  $\gamma$-$A$-supercompact. 

\end{enumerate}

\end{lemma}

\begin{proof}

$(2) \rightarrow (1)$:
 Let $f: \delta \to \delta$, define $g(\kappa)=\max\{\kappa,f(\kappa)\}+2$.
Let $A$ code $f$ and $g$. 
By $(2)$, let $\kappa<\delta$ be $\gamma$-$A$-supercompact for all $\gamma<\delta$.
Take $\gamma=\sup g[ \kappa+1]$.
Hence, there is a map $j:V \to M$ such that $\crit(j)=\kappa$, $j(\kappa)>\gamma$, ${}^\gamma M \subseteq M$, and $A \cap V_{\gamma}=j(A) \cap V_{\gamma}$.
Since $A$ codes $f$ and $f(\kappa)<g(\kappa)$, we have that
 $f(\kappa)=j(f)(\kappa)$, so ${}^{j(f)(\kappa)}M \subseteq M$, and for $\alpha<\kappa$, $j(f)(\alpha)=f(\alpha)<j(\kappa)$, so $f(\alpha)<\kappa$.
 Hence, $f$ is closed under $\kappa$.
 
 $(1) \rightarrow (2)$: 
 Let $C \subseteq \delta$ be a club. 
 Define $f_i:\delta \to \delta$ for $i<3$ as follows.
\begin{align*}
f_0(\kappa)=\min(C \setminus \kappa).
\end{align*}

\begin{align*}
f_1(\kappa)=
\begin{cases}
0 & \text{ if } \kappa \text{ is } \gamma$-$A$-$\text{supercompact for all } \gamma<\delta, \\
\lambda & \text{ if } \lambda \text{ is the least such that } \kappa \text{ is not } \lambda$-A-$\text{supercompact}.
\end{cases}
\end{align*}

\begin{align*}
f_2(\kappa)=|V_{f_1(\kappa)} |+(f_0(\kappa)+1).
\end{align*}

Let $F(\kappa)=|\mathcal{P}(f_2(\kappa)^{<\kappa})|^+ +3$.
By Woodinized supercompactness, let $\kappa<\delta$ be such that $F[\kappa] \subseteq \kappa$, and there is an elementary embedding $j:V \to M$ such that $\crit(j)=\kappa$ and ${}^{j(F)(\kappa)}M \subseteq M$.
For $\alpha<\kappa$, $\alpha \leq f_0(\alpha)<F(\alpha)<\kappa$ and $f_0(\alpha) \in C$, so $C \cap \kappa$ is unbounded in $\kappa$.
Since $\crit(j)=\kappa$, we have $C \cap \kappa =j(C) \cap \kappa$ which is unbounded in $\kappa$, so $\kappa \in j(C)$.
To finish the proof, we will show that in $M$, $\kappa$ is $\gamma$-$j(A)$-supercompact for all $\gamma<j(\delta)$.
Suppose not, then $\kappa$ is not $j(f_1)(\kappa)$-$j(A)$-supercompact in $M$.
Since $j(f_2)(\kappa)<j(F)(\kappa)$, we can derive a supercompact measure $\mathcal{U}$ on $\mathcal{P}_\kappa(j(f_2)(\kappa))$ from $j$.
With the definition of $F$, $\mathcal{U} \in M$ and $M \models \mathcal{U}$ is a supercompact measure.
Let $j_\mathcal{U}: V \to M_\mathcal{U}=\Ult(V,\mathcal{U})$ and $j_\mathcal{U}^M: M \to \Ult(M,\mathcal{U})$.

 \begin{figure}[H]
 \includegraphics[scale = 0.75]{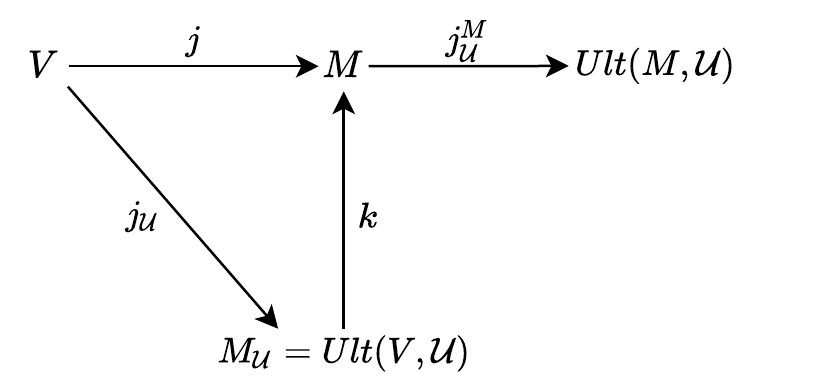}
\end{figure}

We will show that $j_\mathcal{U}^M$ witnesses $\kappa$ being $j(f_1)(\kappa)$-$j(A)$-supercompact in $M$, which is a contradiction.
Clearly $j_\mathcal{U}^M$ witnesses $\kappa$ being $j(f_1)(\kappa$)-supercompact in $M$.
It remains to show that 

\begin{center}

$j(A) \cap V_{j(f_1)(\kappa)}^M=j_\mathcal{U}^M(j(A)) \cap V_{j(f_1)(\kappa)}^M.$

\end{center}

Note that $M$ is closed under $j(F)(\kappa)$ sequences, and by the definition of $f_2$, $V_{j(f_1)(\kappa)}^{\Ult(M,\mathcal{U})}=V_{j(f_1)(\kappa)}^M$.
If we can show that

\begin{center}

$j(A) \cap V_{j(\kappa)}^M=j_\mathcal{U}^M(j(A)) \cap j_\mathcal{U}^M(V_\kappa)$,

\end{center}

then by intersecting both sides of the equation by $V_{j(f_1)(\kappa)}^M$, the proof will be completed.
It is easy to check that if $k:M_\mathcal{U} \to M$ is the factor map (namely $j_\mathcal{U} \circ k =j$), then $\crit(k)> \kappa$.
Hence,
\begin{center}

$j(A) \cap V_{j(\kappa)}^M=j(A \cap V_\kappa)=j_\mathcal{U}(A \cap V_\kappa)=j_\mathcal{U}^M(A \cap V_\kappa)=j_\mathcal{U}^M(j(A) \cap V_\kappa)=j_\mathcal{U}^M(j(A)) \cap j_\mathcal{U}^M(V_\kappa)$.

\end{center}

\end{proof}

We now characterize a Woodinized supercompact cardinal in a similar way to Lemma \ref{woodinscc1}, except with some Laver function.

\begin{lemma}
\label{woodinscc2}
Let $\delta$ be a cardinal.
The following are equivalent.

\begin{enumerate}

\item $\delta$ is Woodinized supercompact.

\item Let $A \subseteq V_\delta$.
Then the collection of $\kappa$ such that $\varphi(\kappa,\delta,A)$ holds is stationary, where $\varphi(\kappa,\gamma,A)$ is described as follows: there is a function $s: \kappa \to V_\kappa$ such that for each cardinal $\lambda \in (\kappa,\delta)$ and $x \in H_{|V_\lambda|^+}$, there is a supercompact measure $\mathcal{U}$ on $\mathcal{P}_\kappa(|V_\lambda|)$ an elementary embedding $j_\mathcal{U}:V \to M=\Ult(V,\mathcal{U})$ witnessing $\kappa$ being $\lambda$-$A$-supercompact, and $j_\mathcal{U}(s)(\kappa)=x$.

\end{enumerate}

\end{lemma}

\begin{proof}

$(2) \rightarrow (1)$: Since $(2)$ implies the second item in Lemma \ref{woodinscc1}, it is obvious.

$(1) \rightarrow (2)$:  First, note that by a standard argument, $\delta$ is strongly inaccessible.
Fix a bijection $H:\delta \setminus \{0\} \to V_\delta$.
Define $f_0,f_1:\delta \to \delta$ recursively as follows:

If there are $\lambda \in (\kappa,\delta)$ and $x \in H_{|V_\lambda|^+}$ such that there is no supercompactness measure $\mathcal{U}$ on $\mathcal{P}_\kappa(|V_\lambda|)$ witnessing $\lambda$-$A$-supercompact, such that $j_\mathcal{U}(H \circ (f_0 \restriction \kappa))(\kappa)=x$, then let $\lambda$ be the least such, and $\gamma$ be the least such that $H(\gamma)$ is an $x$ witnessing the argument above.
Otherwise, $f_0(\kappa)=0$.
Note that  if $\mathcal{U}$ be as above, then $j_\mathcal{U}(A) \cap V_\lambda=j_\mathcal{U}(A \cap V_\lambda) \cap V_\lambda$, so $j_\mathcal{U}$ witnesses $\kappa$ being $\lambda$-$A$-supercompact if and only if $\lambda$-$A \cap V_\lambda$-supercompact, so we may replace $\lambda$-$A$-supercompact by $\lambda$-$A \cap V_\lambda$-supercompact.

$f_1(\kappa)=rank(H(f_0(\kappa)))$ if $f_0(\kappa) \neq 0$, otherwise $f_1(\kappa)=0$.

Let $B$ code $f_0,f_1,H$ and $A$.
By $(2)$ of Lemma \ref{woodinscc1}, the collection of $\kappa<\delta$ being $\gamma$-$B$-supercompact for all $\gamma<\delta$ is stationary.
Fix such $\kappa$. 
Let $t: \delta \to V_\delta$ be defined as follows: $t(\alpha)=H(f_0(\alpha))$.
By the definition of $f_1$, for $\alpha<\kappa$, we have that $rank(t(\alpha))<\kappa$, so $\rge(t \restriction \kappa) \subseteq V_\kappa$.
Let $s=t \restriction \kappa$.
Suppose for a contradiction that there are a cardinal $\lambda \in (\kappa,\delta)$ and $x \in H_{\lambda^+}$ such that there is no supercompact measure $\mathcal{U}$ on $\mathcal{P}_\kappa(|V_\lambda|)$ such that when letting $j_\mathcal{U}:V \to M$ witnessing $\kappa$ being $\lambda$-$A$-supercompact , then $j(s)(\kappa)=x$.
Let $\rho=2^{2^{|V_\lambda|}}$ and let $j: V \to M$ witness $\kappa$ being $\rho$-$A$-supercompact.
Then $H_{|V_\lambda|^+} \subseteq M$ and by an agreement between $M$ and $V$, we have that every supercompact measure on $\mathcal{P}_\kappa(|V_\lambda|)$ is in $M$.
Furthermore, $j(A) \cap V_\lambda=j(A \cap V_\lambda) \cap V_\lambda$.
Furthermore, for every such supercompact measure $\mathcal{U}$, we have that $V_\lambda \subseteq \rge(j_\mathcal{U})$, thus, we have that

\begin{center}

$M \models$ there is no supercompact measure $\mathcal{U}$ on $\mathcal{P}_\kappa(|V_\lambda|)$ such that $j_\mathcal{U}$ witnesses $\kappa$ being $\lambda$-$A \cap V_\lambda$-supercompact, and $j_\mathcal{U}(s)(\kappa)=x$.

\end{center}

Let $y=j(s)(\kappa)$.
Then in $M$, there is $\mu \leq \lambda$ with no supercompact measure $\mathcal{U}$ such that $j_\mathcal{U}$ witnesses $\kappa$ being $\lambda$-$A \cap V_\lambda$-supercompact.
Derive a supercompact measure $\mathcal{U}$ on $\mathcal{P}_\kappa(|V_\lambda|)$.
Let $j_\mathcal{U}:V \to M_\mathcal{U}=\Ult(V,\mathcal{U})$ and $k:M_\mathcal{U} \to M$ be the factor map.
Then $\crit(k)>|V_\lambda|$.
A routine argument shows that $j_\mathcal{U}(s)(\kappa)=y$.
Since $\mathcal{U}$ witnesses $|V_\lambda|$-supercompactness, we have that $j_\mathcal{U}(A) \cap V_\lambda=j_\mathcal{U}(A \cap V_\lambda) \cap V_\lambda=j(A \cap V_\lambda) \cap V_\lambda=A \cap V_\lambda$.
Hence $j_\mathcal{U}$ witnesses $\kappa$ being $\lambda$-$A$-supercompact and $j(s)(\kappa)=y$, which is a contradiction.

\end{proof}

For $x \in V$, denote $j_E(x)$ by $x^{j_E}$ (for instance, $j(\kappa),j(\lambda)$ are denoted by $\kappa^{j_E},\lambda^{j_E}$, respectively).
Fix a sequence  $\langle E_\alpha: \alpha<\eta\rangle$ such that for $\alpha<\eta$, $E_\alpha$ is an $(\kappa_\alpha,\lambda^{j_{E_\alpha}})$-extender. 
For $\alpha<\eta$, define $\kappa_\alpha^j={\kappa_\alpha}^{j_{E_\alpha}}$, $\lambda_\alpha^j=\lambda^{j_{E_\alpha}}$ (so that $E_\alpha$ is an $(\kappa_\alpha,\lambda_\alpha^j$)-extender). 
For $\alpha  \in (0, \eta]$ define $\bar{\kappa}_\alpha=\sup_{\beta<\alpha}\kappa_\beta$, $\bar{\kappa}_\alpha^j=\sup_{\beta<\alpha} \kappa_\beta^j$, and $\bar{\lambda}_\alpha^j=\sup_{\beta<\alpha} \lambda_\beta^j$.
Let $\bar{\lambda}_0^j=\lambda$.
Note that $\bar{\kappa}_\alpha$, $\bar{\kappa}_\alpha^j$, and $\bar{\lambda}_\alpha^j$ are defined without mentioning $\kappa_\alpha$, $\kappa_\alpha^j$, and $\lambda_\alpha^j$, respectively.

Given an extender $E=\langle E_a: a \in [\xi]^{<\omega} \rangle$ and $\gamma<\xi$, we define $E \restriction \gamma$ as $\langle E_a: a \in [\gamma]^{<\omega} \rangle$.

\begin{defn}

Let $E$ and $F$ be a $(\kappa_E,\lambda^{j_E})$-extender and a $(\kappa_F,\lambda^{j_F})$-extender, respectively, where $\kappa_E<\kappa_F<\lambda<\lambda^{j_E}<\lambda^{j_F}$.
$E$ is {\em coherent in} $F$ if $j_F(E) \restriction \lambda^{j_E}=E$.
A sequence of extenders $\langle E_\alpha: \alpha<\eta \rangle$ is {\em coherent} if for $\beta<\alpha$, $E_\beta$ is coherent in $E_\alpha$.

\end{defn}

We sometimes say that $E$ and $F$ are coherent instead of saying that $E$ is coherent in $F$, however, it is important to be aware that saying that $E$ and $F$ are coherent is {\em not} the same as saying that $F$ and $E$ are coherent.
The conclusion of the following theorem will be our initial setting to our forcing construction.

\begin{thm} \label{deriveext}

Assume GCH.
Let $\delta$ be a Woodinized supercompact cardinal, and $\eta \in (0,\delta)$ is an ordinal.
Then there is an increasing sequence of cardinals $\langle \kappa_\alpha: \alpha<\eta \rangle$ such that $\eta<\kappa_0$, $\kappa_\alpha<\delta$ for all $\alpha$, and for every fixed cardinal $\lambda \in [\bar{\kappa}_\eta,\delta)$ , if  $\cf(\lambda) \geq \bar{\kappa}_\eta$, then there is a sequence of extenders $\langle E_\alpha:\alpha<\eta \rangle$ such that $E_\alpha$ is a $(\kappa_\alpha,\lambda_\alpha^j)$- extender such that by letting  $j_{E_\alpha}: V \to M_\alpha=\Ult(V,E_\alpha)$, then the following hold:

\begin{enumerate}

\item $j_{E_\alpha}$ witnesses $\kappa_\alpha$ being $\lambda$-supercompact.

\item there is a function $s_0: \kappa_0 \to \kappa_0$ such that $j_{E_0}(s_0)(\kappa_0)=\lambda$ (as a consequence, $\lambda<\kappa_0^j$).

\item for $\alpha \in (0,\eta)$,  there are functions $s_\alpha,t_\alpha^\beta:\kappa_\alpha \to \kappa_\alpha$ for $\beta<\alpha$ such that $j_{E_\alpha}(s_\alpha)(\kappa_\alpha)=\lambda$, $j_{E_\alpha}(t_\alpha^\beta)(\kappa_\alpha)=\lambda_\beta^j$ and $|\lambda_\beta^j|^V$ is regular (recall that $\lambda_\beta^j=j_{E_\beta}(\lambda)$).
As a consequence, for $\beta \leq \alpha$ there is a function $u_\alpha^\beta:\kappa_\alpha \to \kappa_\alpha$ such that $j_{E_\alpha}(u_\alpha^\beta)(\kappa_\alpha)=\bar{\lambda}_\beta^j$.
 
 \item The sequence $\langle E_\alpha: \alpha<\eta \rangle$ is coherent.

\end{enumerate}

Furthermore, let $\langle \gamma_\alpha: \alpha<\eta \rangle$ be a sequence of ordinals below $\delta$, then $E_\alpha$ can be constructed so that $\kappa_\alpha^j> \gamma_\alpha$.

\end{thm}

Note that from Theorem \ref{deriveext} for $\beta<\alpha<\eta$, we see that $E_\beta=j_{E_\alpha}(\gamma \mapsto E_\beta \restriction t_\alpha^\beta(\gamma))(\kappa_\alpha)$, so  $E_\beta \in M_\alpha$.
Furthermore, for $\beta<\alpha$, $\lambda_\beta^j<\kappa_\alpha^j<\lambda_\alpha^j$.

\begin{proof}[Proof of Theorem \ref{deriveext}]

Fix $\eta$.
By Lemma \ref{woodinscc2} (with $A= \emptyset$), let $\kappa_0<\delta$ be such that for each cardinal $\lambda \in [\kappa_0,\delta)$ such that $\cf(\lambda) \geq \kappa_0$, there is $j_{0,\lambda}:V \to M_{0.\lambda}$ witnessing $\kappa_0$ being $\max\{\lambda,\gamma_0^+\}$-supercompact, such that there is a map $s_{0,\lambda}: \kappa_0 \to \kappa_0$ with $j_{E_{0,\lambda}}(s_{0,\lambda})(\kappa_0)=\lambda$.
We may assume that $|j_{0,\lambda}(\lambda)|^V$ is regular.
The regularity of $|j_{0,\lambda}(\lambda)|^V$ can be assumed by making sure that $j_{0,\lambda}$ is an ultrapower map from some supercompact measure.
Note that $\cf(j_{0,\lambda}(\lambda))^{M_{0,\lambda}} \geq j_{0,\lambda}(\kappa_0)$.
Derive a $(\kappa_0,j_{0,\lambda}(\lambda))$-extender $E_{0,\lambda}$ from $j_{0,\lambda}$.
By Lemma \ref{lengthcaprhoscc}, $j_{E_{0,\lambda}}$ witnesses $\kappa_0$ being $\lambda$-supercompact, $j_{E_{0,\lambda}}(s_{0,\lambda})(\kappa_0)=\lambda$, and $\gamma_0<j_{E_{0,\lambda}}(\kappa_0)<j_{E_{0,\lambda}}(\lambda)$. 

Let $\alpha \in (0,\eta)$.
Assume for $\beta^\prime<\beta<\alpha$, $A_{\beta,\lambda}$, $\kappa_\beta$ are designated so that for every $\lambda \in \kappa_\beta,\delta)$ of cofinality at least $\bar{\kappa}_\beta$, $E_{\beta,\lambda}$, $s_{\beta,\lambda}$, $t_{\beta,\lambda}^{\beta^\prime}$ are defined in a way that by letting $j_{E_{\beta,\lambda}}:V \to M_{E_{\beta,\lambda}}=\Ult(V,E_{\beta,\lambda})$ with $\theta_{\beta,\lambda}^j=\lambda^{j_{E_{\beta,\lambda}}}$, we have that

\begin{enumerate}

\item $E_{\beta,\lambda}$ is an $(\kappa_\beta,\theta_{\beta,\lambda}^j)$-extender.

\item $j_{E_{\beta,\lambda}}$ witnesses $\kappa_\beta$ being $\lambda$-$A_\beta$-supercompact, and $|j_{E_{\beta,\lambda}}(\lambda)|^V$ is regular.

\item $j_{E_{\beta,\lambda}}(\kappa_\beta)>\gamma_\beta$.

\item $j_{E_{\beta,\lambda}}(s_\beta)(\kappa_\beta)=\lambda$ and $j_{E_{\beta,\lambda}}(t_\beta^{\beta^\prime})(\kappa_\beta)=\theta_{\beta^\prime,\lambda}^j$ for $\beta^\prime<\beta$.

\end{enumerate}

For $\lambda \in (\bar{\kappa}_\alpha,\delta)$, let $A_{\alpha,\lambda}$ code $A_{\beta,\lambda},E_{\beta,\lambda},s_{\beta,\lambda},t_{\beta,\lambda}^{\beta^\prime}$ for $\beta^\prime<\beta<\alpha$, and let $A_\alpha$ code $\langle A_{\alpha,\lambda}:\lambda \in [\bar{\kappa}_\alpha,\delta) \rangle$.
Let $\bar{\theta}_{\alpha,\lambda}^j=\sup_{\beta<\alpha} \theta_{\beta,\lambda}^j$.
By Lemma \ref{woodinscc2}, let $\kappa_\alpha \in (\bar{\kappa}_\alpha,\delta)$ be such that for $\lambda \in [\kappa_\alpha,\delta)$ with $\cf(\lambda) \geq \kappa_\alpha$, there is a map  $j_{\alpha,\lambda}:V \to M_{\alpha,\lambda}$ witnessing $\kappa_\alpha$ being $(\max\{(\bar{\theta}_{\alpha,\lambda}^j)^+,\gamma_\alpha^+\})$-$A_{\alpha,\lambda}$-supercompact, and there are $s_{\alpha,\lambda},t_{\alpha,\lambda}^\beta: \kappa_\alpha \to \kappa_\alpha$ for $\beta<\alpha$ such that $j_{\alpha,\lambda}(s_{\alpha,\lambda})(\kappa_\alpha)=\lambda$ and $j_{\alpha,\lambda}(t_{\alpha,\lambda}^\beta)(\kappa_\alpha)=\theta_{\beta,\lambda}^j$ (so the map $j_{\alpha,\lambda}$ is correct about $A_{\alpha,\lambda}$ up to $V_{\max \{ (\bar{\theta}_{\alpha,\lambda}^j)^+,\gamma_\alpha^+\} })$.
We may also assume that $|j_{\alpha,\lambda}(\lambda)|^V$ is regular by making sure that $j_{\alpha,\lambda}$ is an ultrapower map from some supercompact measure.
Derive a $(\kappa_\alpha,\lambda^{j_{\alpha,\lambda}})$-extender $E_{\alpha,\lambda}$.
By letting $j_{E_{\alpha,\lambda}}:V \to M_{E_{\alpha,\lambda}}=\Ult(V,E_{\alpha,\lambda})$, we have that

\begin{enumerate}

\item $j_{E_{\alpha,\lambda}}$ witnesses $\kappa_\alpha$ being $\lambda$-supercompact, and $|j_{E_{\alpha,\lambda}}(\lambda)|^V$ is regular.

\item $j_{E_{\alpha,\lambda}}(\kappa_\alpha)=j_{\alpha,\lambda}(\kappa_\alpha)>\gamma_\alpha$.

\item $j_{E_{\alpha,\lambda}}(s_{\alpha,\lambda})(\kappa_\alpha)=\lambda$.

\item For $\beta<\alpha$, $j_{E_{\alpha,\lambda}}(t_{\alpha,\lambda}^\beta)(\kappa_\alpha)=\theta_{\beta,\lambda}^j$.

\end{enumerate}

This finishes our construction.
Fix each $\lambda \in [\bar{\kappa}_\eta,\delta)$ with $\cf(\lambda) \geq \bar{\kappa}_\eta$. 
Set $j_\alpha=j_{\alpha,\lambda}, E_\alpha=E_{\alpha,\lambda}$ $s_\alpha=s_{\alpha,\lambda}$, and $t_\alpha^\beta=t_{\alpha,\lambda}^\beta$.
It remains to show that $\langle E_\alpha: \alpha<\eta \rangle$ is coherent.
Let $\beta<\alpha<\eta$.
With our preparation, $j_\alpha(E_\beta) \cap V_{j_\beta(\lambda)}=E_\beta \cap V_{j_\beta(\lambda)}$.
Also, $j_\beta(\lambda)=j_{E_\beta}(\lambda)$.
 For $a \in [j_\beta(\lambda)]^{<\omega}$, we have that 
 \begin{align*}
A \in j_{E_\alpha}(E_\beta)_a \Leftrightarrow A \in j_\alpha(E_\beta)_a \Leftrightarrow A \in (E_\beta)_a
\end{align*} 
In conclusion, $j_{E_\alpha}(E_\beta) \restriction \lambda^{j_{E_\beta}}=E_\beta$.

\end{proof}

\section{an analysis on a coherent sequence of extenders}
\label{analysisofextenders}

From now on, we always assume GCH.
Let $\langle \kappa_\alpha:\alpha<\eta \rangle, \lambda, \langle E_\alpha:\alpha<\eta \rangle, \langle s_\alpha : \alpha<\eta \rangle, \langle t_\alpha^\beta:\beta<\alpha<\eta \rangle$ and $\langle u_\alpha^\beta:\beta \leq \alpha<\eta \rangle$ be as in Theorem \ref{deriveext}, where $\lambda$ is regular.
Let $\bar{\kappa}_0=\eta^+$.
We also assume that $s_\alpha, t_\alpha^\beta,u_\alpha^\beta$ are strictly increasing, and their ranges are subsets of  $(\bar{\kappa}_\alpha,\kappa_\alpha)$.
Note that $\bar{\lambda}_{\beta+1}^j=\lambda_\beta^j$, we may assume that $u_\alpha^{\beta+1}=t_\alpha^\beta$, which implies that the notion $u_\alpha^{\beta+1}$ is redundant.
If $\alpha$ is limit, and $\beta<\alpha$, then $\lambda_\beta^j<\bar{\lambda}_\alpha^j$, so we may assume that $t_\alpha^\beta(\gamma)<u_\alpha^\alpha(\gamma)$ for all $\gamma$.
In general, if $f,g: \kappa_\alpha \to \kappa_\alpha$, $f$ represent an ordinal lower than the ordinal represented by $g$ in $M_\alpha=\Ult(V,E_\alpha)$, we will assume that $f(\gamma)<g(\gamma)$ for all $\gamma$.
Define $E_\alpha(\kappa_\alpha)$ as follows: For $A \subseteq \kappa_\alpha$, let $A \in E_\alpha(\kappa_\alpha)$ iff $\kappa_\alpha \in j_{E_\alpha}(A)$.
Then $E_\alpha(\kappa_\alpha)$ is just a normal measure on $\kappa_\alpha$, which is isomorphic to $E_\alpha(\{\kappa_\alpha\})$.

Fix $\alpha<\eta$.
Recall that $j_{E_\alpha}:V \to M_\alpha=\Ult(V,E_\alpha)$.
We refer to the notions (which will often be used for the rest of the paper) in the paragraph after Lemma \ref{woodinscc2}.
We now list all key ingredients in $M_\alpha$.
Then we reflect the ingredients down to a measure-one set with respect to $E_\alpha(\kappa_\alpha)$.
Let $\beta<\alpha$. 
We have the following:

\begin{itemize}

\item $\lambda=j_{E_\alpha}(\gamma \mapsto s_\alpha(\gamma))(\kappa_\alpha)$ and $\lambda$ is regular (both in $M_\alpha$ and in $V$).

\item $\lambda_\beta^j=j_{E_\alpha}(\gamma \mapsto t_\alpha^\beta(\gamma))(\kappa_\alpha)$.

\item $\cf^{M_\beta}(\lambda_\beta^j) > \kappa_\beta^j>\lambda$, so $\cf^{M_\alpha}(\lambda_\beta^j) \geq \cf^V(\lambda_\beta^j) > \lambda=j_{E_\alpha}(s_\alpha)(\kappa_\alpha)$.

\item $E_\beta=j_{E_\alpha}(\gamma \mapsto E_\beta \restriction t_\alpha^\beta(\gamma))(\kappa_\alpha)$.

\item $\kappa_\beta^j = j_{E_\alpha}(\gamma \mapsto j_{E_\beta \restriction t_\alpha^\beta(\gamma)}(\kappa_\beta))(\kappa_\alpha)$.

\item \begin{align}
j_{E_\alpha}(t_\alpha^\beta)(\kappa_\alpha)&=\lambda_\beta^j\\
&=j_{E_\beta}(\lambda)\\
&=j_{E_\alpha}(\gamma \mapsto j_{E_\beta \restriction t_\alpha^\beta(\gamma)}(s_\alpha(\gamma)))(\kappa_\alpha)
\end{align}

\end{itemize}

Furthermore, if  $\xi<\beta<\alpha$, we have that

\begin{enumerate}

\item $\lambda_\xi^j \leq \bar{\lambda}_\beta^j<\kappa_\beta^j<\lambda_\beta^j$, which directly translates to $j_{E_\alpha}(t_\alpha^\xi)(\kappa_\alpha) \leq j_{E_\alpha}(u_\alpha^\beta)(\kappa_\alpha) <\kappa_\beta^j<j_{E_\alpha}(t_\alpha^\beta)(\kappa_\alpha)$ (recall $\kappa_\beta^j = j_{E_\alpha}(\gamma \mapsto j_{E_\beta \restriction t_\alpha^\beta(\gamma)}(\kappa_\beta))(\kappa_\alpha)$).

\item \label{lambdaxi}
\begin{align*}
j_{E_\alpha}(\gamma \mapsto j_{E_\beta \restriction t_\alpha^\beta(\gamma)}(t_\beta^\xi)(\kappa_\beta))(\kappa_\alpha)&=j_{E_\beta}(t_\beta^\xi)(\kappa_\beta)\\
&=\lambda_\xi^j\\
&=j_{E_\alpha}(t_\alpha^\xi)(\kappa_\alpha).
\end{align*}

\item \label{lambdabarxi}
\begin{align*}
j_{E_\alpha}(\gamma \mapsto j_{E_\beta \restriction t_\alpha^\beta(\gamma)}(u_\beta^\xi)(\kappa_\beta))(\kappa_\alpha)&=j_{E_\beta}(u_\beta^\xi)(\kappa_\beta)\\
&=\bar{\lambda}_\xi^j\\
&=j_{E_\alpha}(u_\alpha^\xi)(\kappa_\alpha).
\end{align*}

\item \label{fullcommute}
\begin{align*}
E_\xi &=j_{E_\beta}(E_\xi) \restriction \lambda_\xi^j\\
&=j_{E_\alpha}(\gamma \mapsto j_{E_\beta \restriction t_\alpha^\beta(\gamma)}(E_\xi) \restriction t_\alpha^\xi(\gamma))(\kappa_\alpha)\\
&=j_{E_\alpha}(\gamma \mapsto j_{E_\beta \restriction t_\alpha^\beta(\gamma)}(E_\xi \restriction t_\alpha^\xi(\gamma)) \restriction t_\alpha^\xi(\gamma))(\kappa_\alpha).
\end{align*}

\end{enumerate}

Finally, we note that $\langle E_\beta: \beta<\alpha \rangle$ is coherent in $M_\alpha$.
We present the order of all relevant ordinals in $\Ult(V,E_\alpha)$  in Figure \ref{cardinalline} below, when $\xi<\beta<\alpha$.

 \begin{figure}[H]
 \centering
 \captionsetup{justification=centering} 
 \includegraphics[scale = 0.75]{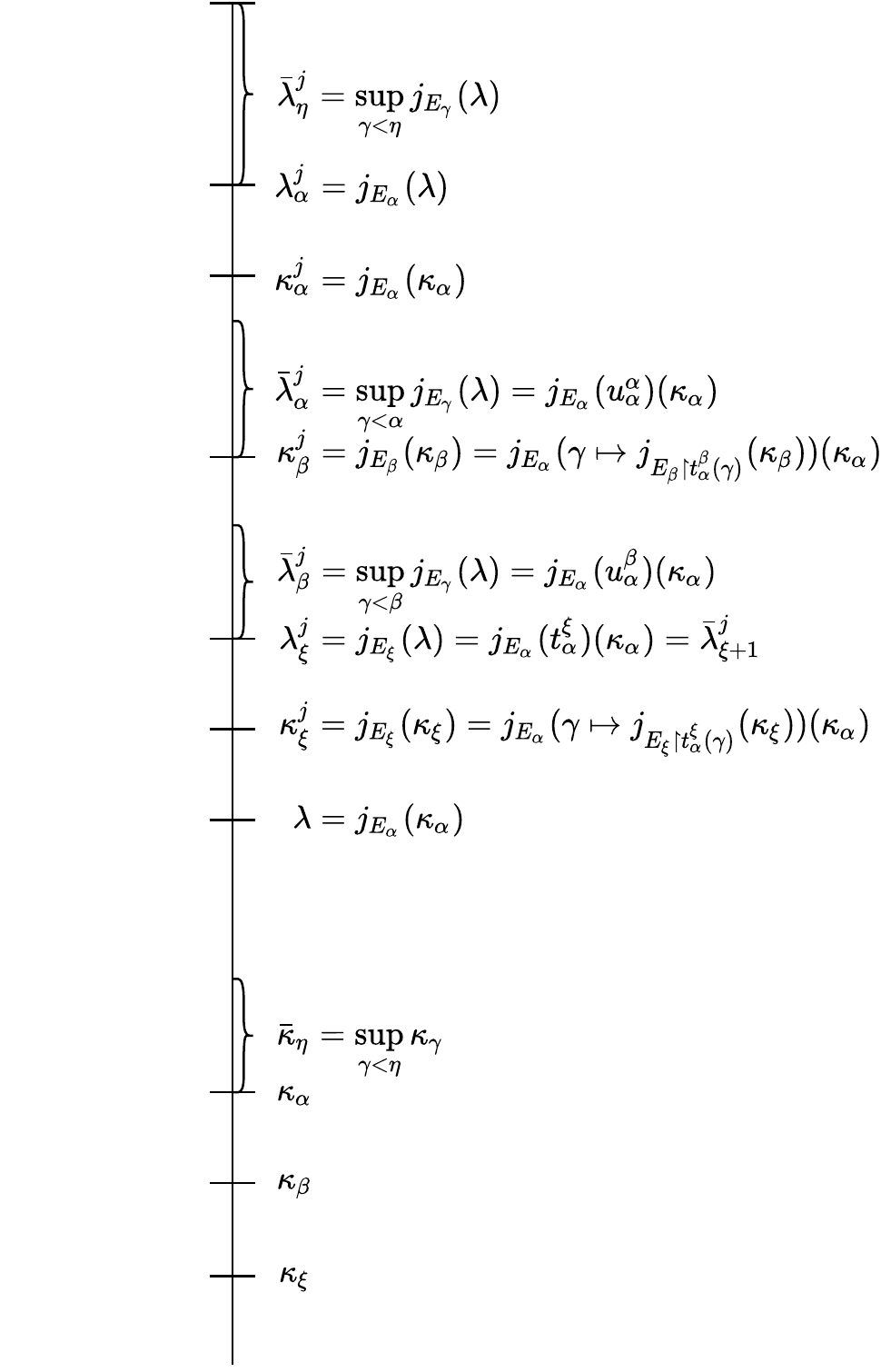}
  \caption{Important ordinals presented in $\Ult(V,E_\alpha)$}
  \label{cardinalline}
\end{figure}

Notice that in Figure \ref{cardinalline} we have a few ordinals indicated by parentheses.
The main reason is that we do not assume $\eta$ to be a limit ordinal.
From the bottommost parentheses, if $\eta=\alpha+1$, then $\bar{\kappa}_\eta=\kappa_\alpha$, otherwise $\bar{\kappa}_\eta>\kappa_\alpha$. 
We also have that $\bar{\lambda}_\beta^j \geq \lambda_\xi^j$, and they are equal if and only if $\beta=\xi+1$.
The analogues also explain the behaviors of $\bar{\lambda}_\alpha^j$ and $\bar{\lambda}_\eta^j$.

The intuition of the requirements for $\gamma$ to be $\alpha$-reflected in Definition \ref{reflect} is that the requirements are what  $\kappa_\alpha$ behaves in $M_\alpha=\Ult(V,E_\alpha)$. 
The collection of such $\gamma$ in Definition \ref{reflect} is of measure-one, which is stated in Lemma \ref{los}.

\begin{defn}
\label{reflect}
$\gamma<\kappa_\alpha$ is {\em $\alpha$-reflected for the sequence $\langle E_\beta: \beta \leq \alpha \rangle$} if

\begin{enumerate}

\item $\gamma > \bar{\kappa}_\alpha$ is regular.

\item $s_\alpha(\gamma)$ is regular.

\item If $\alpha>0$, then for $\beta<\alpha$, let $e_\beta=E_\beta \restriction t_\alpha^\beta(\gamma)$, then we have that

\begin{itemize}

\item $\bar{\kappa}_\alpha<s_\alpha(\gamma)<u_\alpha^\beta(\gamma)<j_{e_\beta}(\kappa_\beta)<t_\alpha^\beta(\gamma)$.

\item $\cf(t_\alpha^\beta(\gamma)) > s_\alpha(\gamma)$.

\item $j_{e_\beta}(s_\beta)(\kappa_\beta)=s_\alpha(\gamma)$.

\item $t_\beta^\alpha(\gamma)=j_{e_\beta}(s_\alpha(\gamma))$.

\item If $\alpha$ is limit, then $t_\alpha^\beta(\gamma)<u_\alpha^\alpha(\gamma)$.

\item $e_\beta$ witnesses $\kappa_\beta$ being $s_\alpha(\gamma)$-supercompact.

\item The sequence $\langle e_\beta: \beta<\alpha \rangle$ is coherent.

\end{itemize}

\item For $\beta<\alpha$, let $e_\beta=E_\beta \restriction t_\alpha^\beta(\gamma)$ be as above.
If $\xi<\beta<\alpha$, 

\begin{itemize}

\item $t_\alpha^\xi(\gamma) \leq u_\alpha^\beta(\gamma)<t_\alpha^\beta(\gamma)$.

\item $j_{e_\beta}(t_\beta^\xi)(\kappa_\beta)=t_\alpha^\xi(\gamma)$.

\item $j_{e_\beta}(u_\beta^\xi)(\kappa_\beta)=u_\alpha^\xi(\gamma)$.

\item $e_\xi=j_{e_\beta}(e_\xi) \restriction t_\alpha^\xi(\gamma)$.

\end{itemize}

\end{enumerate}

\end{defn}

From our analysis, the following lemma is trivial.

\begin{lemma}
\label{los}
$\{\gamma<\kappa_\alpha: \gamma \textrm{ is } \alpha\textrm{-reflected for the sequence } \langle E_\beta : \beta \leq \alpha \rangle\} \in E_\alpha(\kappa_\alpha)$.

\end{lemma}

We will abbreviate the term $\langle E_\beta: \beta \leq \alpha \rangle$ defined in Definition \ref{reflect} if the relevant sequence of extenders is clear from the context, so we sometimes say that $\gamma$ is $\alpha$-reflect.
We emphasize some important different features between $E_\beta$ and $e_\beta=E_\beta \restriction t_\alpha^\beta(\gamma)$ as defined in Definition \ref{reflect}.
Let $j_{e_\beta}:V \to \Ult(V,e_\beta)$.
Then $j_{e_\beta}$ witnesses $\kappa_\beta$ being $s_\alpha(\gamma)$-supercompact, where $s_\alpha(\gamma)<\kappa_\alpha$. 
Also, $j_{e_\beta}(\kappa_\beta)<j_{e_\beta}(s_\alpha(\gamma))=t_\alpha^\beta(\gamma)$.
The reason we mention these matters is because later while a lot of definitions are defined with respect to an extender, the definitions can be applied on the appropiate restrictions of extenders, which have different parameters.
For example, see a comment after Definition \ref{domain}.

We now introduce a notion of domain, which was first established by Gitik and Merimovich, for example see \cite{merimovich}.

\begin{defn}
\label{domain}
Let $\alpha<\eta$.
$d$ is an {\em $\alpha$-domain (with respect to $E_\alpha$)} if $d \in [\lambda_\alpha^j]^{ \lambda}$ is such that $ \lambda+1 \subseteq d$, for $\beta < \alpha$, $\kappa_\beta^j,\bar{\lambda}_\beta^j, \lambda_\beta^j \in d$, and $\kappa_\alpha^j,\bar{\lambda}_\alpha^j \in d$.
Fix an $\alpha$-domain $d$.
Define $\mc_\alpha(d)=(j_{E_\alpha} \restriction d)^{-1}$.
Finally, let $A \in E_\alpha(d)$ iff $\mc_\alpha(d) \in j_{E_\alpha}(d)$.

\end{defn}

Note that the notion of $\alpha$-domain actually depends on the structure the extender $E_\alpha$.
If $\gamma$ is $\alpha$-reflected and $e_\beta=E_\beta \restriction t_\alpha^\beta(\gamma)$, then Definition \ref{domain} is applicable for $e_\beta$, namely we can say that $d$ is a $\beta$-domain with respect to $e_\beta$, if $d \in [t_\alpha^\beta(\gamma)]^{s_\alpha(\gamma)}$ with certain containment.
 This matter will be investigated further with some forcings' features, which can be seen in, for instance,  Definition \ref{2extender1} and \ref{2extender01}.

We sometimes abbreviate $\mc_\alpha(d_\alpha)$ as $\mc_\alpha$ whenever $d_\alpha$ and the relevant extenders are clear from the context.
The notion of $\alpha$-domain  is not ambiguous in the following sense: if $d$ is an $\alpha$-domain, then we see that $\bar{\lambda}_\alpha^j \in d$ and for $\beta>\alpha$, $\bar{\lambda}_\beta \not \in d$, so the ordinal parameter used to define the domains is easily distinguished.

\begin{defn}
\label{alphaobject}
Let $d$ be an $\alpha$-domain.
$\mu$ is an {\em $\alpha$-object with respect to the domain $d$} if $\mu$ is a function such that

\begin{enumerate}

\item $\dom(\mu) \subseteq d$, $\rge(\mu) \subseteq \lambda$.

\item For $\beta < \alpha$, $\kappa_\beta^j,\bar{\lambda}_\beta^j, \lambda_\beta^j \in \dom(\mu)$, and $\kappa_\alpha,\kappa_\alpha^j,\bar{\lambda}_\alpha^j \in \dom(\mu)$.

\item $\bar{\kappa}_\alpha+1 \subseteq \dom(\mu) \cap \kappa_\alpha=\mu(\kappa_\alpha)<\kappa_\alpha$.

\item $\mu(\kappa_\alpha)$ is $\alpha$-reflected.

\item $\mu$ is order-preserving.

\item $\mu(j(\kappa_\alpha^j))=\kappa_\alpha^j$, as a consequence, $\rge(\mu \restriction \kappa_\alpha^j) \subseteq \kappa_\alpha$.
In particular, the values $\mu(\kappa_\alpha),\mu(\lambda),\mu(\kappa_\beta^j), \mu(\bar{\lambda}_\beta^j)$, and $\mu(\lambda_\beta^j)$ are below $\kappa_\alpha$ for $\beta<\alpha$.

\item For $\gamma<\mu(\kappa_\alpha)$, $\mu(\gamma)=\gamma$.

\item $s_\alpha(\mu(\kappa_\alpha))=\mu(\lambda)$, $t_\alpha^\beta(\mu(\kappa_\alpha))=\mu(\lambda_\beta^j)$, $u_\alpha^\beta(\mu(\kappa_\alpha))=\mu(\bar{\lambda}_\beta^j)$ for $\beta<\alpha$.

\item $|\dom(\mu)|=s_\alpha(\mu(\kappa_\alpha))$ (which is below $\kappa_\alpha$).

\item $\rge(\mu) \supseteq s_\alpha(\mu(\kappa_\alpha))>\mu(\kappa_\alpha)$.

\end{enumerate}

\end{defn}

\begin{defn}

For an $\alpha$-domain $d$, let $\OB_{E_\alpha}(d)=\{\mu: \mu \text{ is an } \alpha \text{ object with respect to } d\}$.
We may just write $\OB_\alpha(d)$ instead of $\OB_{E_\alpha}(d)$ if $E_\alpha$ and $d$ are clear from the context.

\end{defn}

We visualize a typical $\alpha$-object $\mu$ in Figure \ref{objectdiagram} below, where in the figure, $\beta<\alpha$ is fixed.

 \begin{figure}[H]
 \centering
 \captionsetup{justification=centering} 
 \includegraphics[scale = 1]{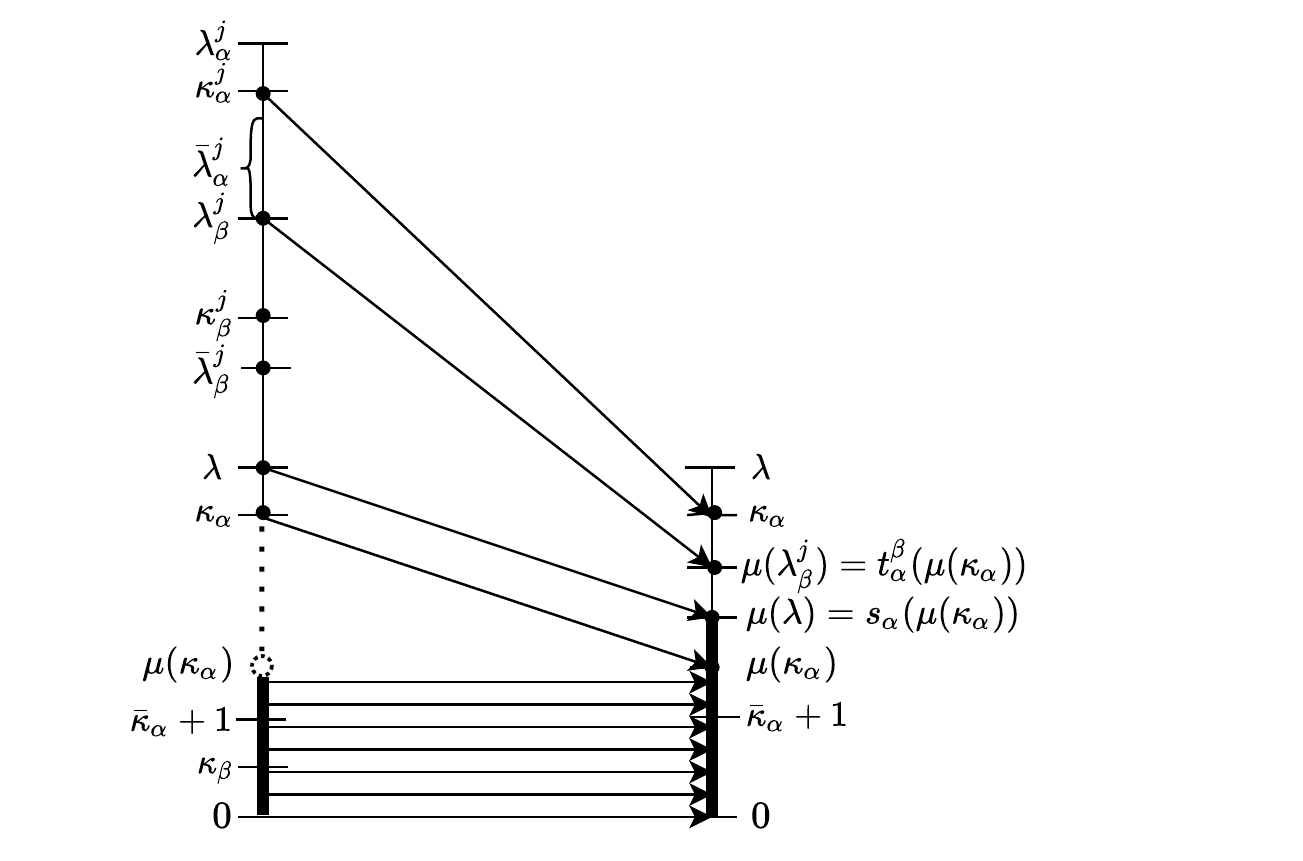}
  \caption{A diagram of a typical $\alpha$-object $\mu$}
  \label{objectdiagram}
\end{figure}

In Figure \ref{objectdiagram}, the left vertical line represents the domain and the right vertical line represents the range.
The arrows from the left-hand side to the right-hand side represents how the $\alpha$-object $\mu$ maps some certain values.
The parts which are highlighted with bold lines and the bold dots are guaranteed to be in the domain or the range. 
The parts in the domain which are represented by the dash line and the dash-line circle are guaranteed to be missing from the domains.
Note that the top ordinals, i.e. $\lambda_\alpha^j$ and $\lambda$ are not in the domain and the range, respectively.
The most important feature of the $\alpha$-objects, which is said in Definition \ref{alphaobject}, is that all important ordinals defined from $E_\beta$, including $\lambda$, are mapped to ordinals below $\kappa_\alpha$.
The notion of objects is not ambiguous in the sense that if $\mu$ is an $\alpha$-object, then $ |\dom(\mu)| \in (\bar{\kappa}_\alpha,\kappa_\alpha)$.

\begin{prop}

If $d$ is an $\alpha$-domain, then $\OB_{E_\alpha}(d) \in E_\alpha(d)$.

\end{prop}

\begin{proof}

We will show that $\mc_\alpha(d) \in j_{E_\alpha}(\OB_\alpha(d))$.
Abbreviate $\mc_\alpha(d)$ by $\mc_\alpha$.

\begin{enumerate}

\item $\dom(\mc_\alpha) \subseteq j_{E_\alpha}[d] \subseteq j_{E_\alpha}(d)$ and $\rge(\mc_\alpha)=d \subseteq \lambda_\alpha^j$.

\item Obvious by the definition of an $\alpha$-domain, and the definition of $\mc_\alpha$.

\item $\dom(\mc_\alpha) \cap \kappa_\alpha^j=\kappa_\alpha=\mc_\alpha(\kappa_\alpha^j)$ and $\kappa_\alpha<\kappa_\alpha^j$.

\item $\mc_\alpha(\kappa_\alpha^j)=\kappa_\alpha$.

\item $\mc_\alpha$ is order-preserving.

\item $\mc_\alpha(j_{E_\alpha}(\kappa_\alpha^j))=\kappa_\alpha^j$.

\item $\crit(j_{E_\alpha})=\kappa_\alpha$.
Hence for $\gamma<\kappa_\alpha$, $\mc(j_{E_\alpha}(\gamma))=\gamma=j_{E_{\alpha}}(\gamma)$ (note that $\kappa_\alpha=\mc(j_{E_\alpha}(\kappa_\alpha))$.

\item $j_{E_\alpha}(s_\alpha)(\mc_\alpha(j_{E_\alpha}(\kappa_\alpha)))=j_{E_\alpha}(s_\alpha)(\kappa_\alpha)=\lambda=\mc_\alpha(\lambda_\alpha^j)$. 
The rests are similar.

\item $|\mc_\alpha|=|d|=\lambda<\kappa_\alpha^j$.

\item $\rge(\mc_\alpha)) =d \supseteq \lambda=j_{E_\alpha}(s_\alpha)(\kappa_\alpha)=j_{E_\alpha}(s_\alpha)(\mc_\alpha(\kappa_\alpha^j)))$.

\end{enumerate}

\end{proof}

We assume that for every $A \in E_\alpha(d)$, $A \subseteq \OB_\alpha(d)$.
Note that although $E_\alpha(d)$ is only an $\kappa_\alpha$-complete ultrafilter, each $A \in E_\alpha(d)$ has size $\lambda$.
We may add extra properties into the definition of an $\alpha$-object as long as the properties are reflected from $\mc_\alpha$.
For example, for $\gamma< \lambda_\alpha^j$ such that $\gamma \in d$, if we let $\lambda_{E_\alpha,\gamma}$ be the least ordinal $\xi \leq \lambda$ such that $\gamma<j_{E_\alpha}(\xi)$, then there is a measure-one set of $\mu$ such that $\gamma \in \dom(\mu)$ and $\mu(\gamma)<\lambda_{E_\alpha,\gamma}$, 
If $j_{E_\alpha}(\gamma) \in d$, we can even assume that $\mu(j_{E_\alpha}(\gamma))=\gamma$.
For each fixed $a \subseteq d$ of size less than $\kappa_\alpha$, we may assume that each $\alpha$-object $\mu$ has a domain containing $a$.

\begin{prop}
 \label{addsmallset}
 If $a \subseteq d$ and $|a|< \kappa_\alpha$, then $\{ \mu: a \subseteq \dom(\mu)\} \in E_\alpha(d)$.

\end{prop}

\begin{defn}
\label{measureoneclosure}
Let $\beta<\alpha<\eta$, $d_\beta$ and $d_\alpha$ be $\beta$ and $\alpha$ domains respectively, and $d_\beta \subseteq d_\alpha$. 
Let $A \in E_\beta(d_\beta)$ and $\tau \in \OB_\alpha(d_\alpha)$.
Define $A_\tau$ as $\{\mu \in A: \dom(\mu) \cup \rge(\mu) \subseteq \dom(\tau)\}$.

\end{defn}

\begin{lemma}
\label{aux}
Let $\beta<\alpha<\eta$, $d_\beta$ and $d_\alpha$ be $\beta$ and $\alpha$ domains respectively, and $d_\beta \subseteq d_\alpha$. 
Let $A \in E_\beta(d_\beta)$.
Then 

$ \{\mu\in j_{E_\alpha}(A): \dom(\mu) \cup \rge(\mu) \subseteq \dom(\mc_\alpha(d_\alpha))\} = j_{E_\alpha}[A]$.
\end{lemma}

\begin{proof}

($\supseteq$): For $\mu \in A$, $|\mu|<\kappa_\beta<\kappa_\alpha$, so $\dom(j_{E_\alpha}(\mu)) \cup \rge(j_{E_\alpha}(\mu)) \subseteq j_{E_\alpha}[d_\beta \cup \lambda] \subseteq \dom(\mc_\alpha(d_\alpha))$.

($\subseteq$): Let $\mu \in j_{E_\alpha}(A)$ be such that $\dom(\mu) \cup \rge(\mu) \subseteq j_{E_\alpha}[d_\alpha]$.
Then each ordered pair $(\gamma_0,\gamma_1) \in \mu$ is of the form $(j_{E_\alpha}(\gamma_0^\prime),j_{E_\alpha}(\gamma_1^\prime))$.
Since every object in $A$ has size less than $\kappa_\beta$, so is $\mu$.
Thus, we have that $\mu = j_{E_\alpha}(\mu^\prime)$ for some $\mu^\prime$, and so $\mu^\prime \in A$.
Therefore, $\mu \in j_{E_\alpha}[A]$.

\end{proof}

\begin{defn}
\label{defndiag}
Let $\beta<\alpha<\eta$, $d_\beta$ and $d_\alpha$ be $\beta$ and $\alpha$ domains respectively, and $d_\beta \subseteq d_\alpha$.
For each $\mu \in \OB_\beta(d_\beta)$, let $A_\mu \in E_\alpha(d_\alpha)$.
Define the {\em diagonal intersection} of $\langle A_\mu: \mu \in \OB_\beta(d_\beta) \rangle$ as follows:

\begin{center}

$\triangle_{\mu \in \OB_\beta(d_\beta)} A_\mu =\{\tau: \forall \mu \in \OB_\beta(d_\beta) (\dom(\mu) \cup \rge(\mu) \subseteq \dom(\tau) \implies \tau \in A_\mu) \}$.

\end{center}

\end{defn}

\begin{lemma}
\label{diagintersect}

With the settings stated in Definition \ref{defndiag}, we have $\triangle_{\mu \in \OB_\beta(d_\beta)} A_\mu  \in E_\alpha(d_\alpha)$.

\end{lemma}

\begin{proof}

Let $\vec{B}=j_{E_\alpha}(\langle A_\mu: \mu \in \OB_\beta(d_\beta) \rangle)$.
Write $\vec{B}$ as a sequence $\langle B_\tau: \tau \in j_{E_\alpha}(\OB_\beta(d_\beta)) \rangle$, where $B_{j_{E_\alpha}(\mu)}=j_{E_\alpha}(A_\mu)$ for all $\mu \in \OB_\beta(d_\beta)$.
For $\mu \in \OB_\beta(d_\beta)$, $\mc_
\alpha(d_\alpha) \in B_{j_{E_\alpha}(\mu)}$.
Note that by Lemma \ref{aux}, the collection of $\mu \in j_{E_\alpha}(\OB_\beta(d_\beta))$ such that $\dom(\mu) \cup \rge(\mu) \subseteq \dom(\mc_\alpha(d_\alpha))$ is exactly $j_{E_\alpha}[\OB_\beta(d_\beta))]$ and $\mc_\alpha(d_\alpha) \in \cap_{\mu \in \OB_\beta(d_\beta)}B_{j_{E_\alpha}(\mu)}$, hence, $\mc_\alpha(d_\alpha) \in j_{E_\alpha} (\triangle_{\mu \in \OB_\beta(d_\beta)}A_\mu)$.

\end{proof}

Before we investigate further on the interactions between two extenders, we provide some conventions.
For each function $f$ whose domain is an $\alpha$-domain, and $f(\kappa_\alpha)$ is $\alpha$-reflected for $\langle E_\beta:\beta \leq \alpha \rangle$, we denote $s_\alpha(f(\kappa_\alpha))$, $t_\alpha^\beta(f(\kappa_\alpha))$,$u_\alpha^\beta(f(\kappa_\alpha))$, and $E_\beta \restriction t_\alpha^\beta(f(\kappa_\alpha))$ by $\lambda_\alpha(f)$, $\lambda_{\beta,\alpha}^j(f)$, $\bar{\lambda}_{\beta,\alpha}^j(f)$, and $e_{\beta,\alpha}(f)$, respectively.
Note that those values actually depend only on $f(\kappa_\alpha)$.

\begin{lemma}
\label{squishablemeasure}

Let $\beta<\alpha<\eta$, $d_\beta$ and $d_\alpha$ be $\beta$ and $\alpha$ domains respectively, and $d_\beta \subseteq d_\alpha$.
Let $A \in E_\beta(d_\beta)$.
Let $B$ be the collection of $\mu \in E_\alpha(d_\alpha)$ such that by letting $e_\beta=e_{\beta,\alpha}(\mu)$,

\begin{enumerate}

\item $\mu[d_\beta]$ is a $\beta$-domain with respect to $e_\beta$.

\item $\mu \circ A_\mu \circ \mu^{-1} \in e_\beta(\mu[ d_\beta])$, where $A_\mu$ is defined as in Definition \ref{measureoneclosure}.

\end{enumerate}

Then $B \in E_\alpha(d_\alpha)$.

\end{lemma}

\begin{proof}

First, notice that $\mc_\alpha(d_\alpha)[j_{E_\alpha}(d_\beta)]=d_\beta$ is a $\beta$-domain with respect to $E_\beta$.

To show the second item, by Lemma \ref{aux}, $\{ \tau \in j_{E_\alpha}(A) : \dom(\tau) \cup \rge(\tau) \subseteq \dom(\mc_\alpha(d_\alpha))\}=j_{E_\alpha}[A]$.
Note that 
\begin{align*}
& \mc_\alpha(d_\alpha) \circ j_{E_\alpha}[A] \circ \mc_\alpha(d_\alpha)^{-1}\\ = & \mc_\alpha(d_\alpha) \circ
\{ \tau \in j_{E_\alpha}(A): \dom(\tau) \cup \rge(\tau) \subseteq \dom(\mc_\alpha(d_\alpha))\} \circ \mc_\alpha(d_\alpha)^{-1}\\
=&j_{E_\alpha}(\mu \mapsto \mu \circ A_\mu \circ \mu^{-1})(\mc_\alpha(d_\alpha)).
\end{align*}

Let $\tau \in A$ and $f=\mc_\alpha(d_\alpha) \circ j_{E_{\alpha}}(\tau) \circ \mc_\alpha(d_\alpha)^{-1}$.
Note that $\gamma \in \dom(f)$ iff $\gamma \in d_\alpha$, $j_{E_\alpha}(\gamma) \in j_{E_\alpha}(\dom(\tau))$, and $j_{E_\alpha}(\tau(\gamma)) \in j_{E_\alpha}[d_\alpha]$.
Since $\rge(\tau) \subseteq \lambda \subseteq d_\alpha$ and $\dom(\tau) \subseteq d_\beta \subseteq d_\alpha$, we have that $\dom(f)= \dom(\tau)$.
A straightforward calculation shows that for $\gamma \in \dom(f)$, $f(\gamma)=\tau(\gamma)$.
Hence, $f=\tau$, and so $\mc_\alpha(d_\alpha) \circ j_{E_\alpha}[A] \circ \mc_\alpha(d_\alpha)^{-1}=A $. The proof is done.

\end{proof}

\begin{defn}

Let $\beta<\alpha$.
Fix the $\beta$-domain and the $\alpha$-domain $d_\beta$ and $d_\alpha$, respectively, and assume that $d_\beta \subseteq d_\alpha$.
Let $A \in E_\beta(d_\beta)$.
Let $\mu \in \OB_\alpha(d_\alpha)$.
We say that $\mu$ is {\em ($\beta$-)squishable with respect to $A$} if $\mu \circ A_\mu \circ \mu^{-1} \in e_\beta(\mu[d_\beta])$ where $e_\beta=e_{\beta,\alpha}(\mu)$.
The notion $\mu \circ A_\mu \circ \mu^{-1}$ is called {\em the conjugation of $A_\mu$ by $\mu$}.

\end{defn}

 \begin{figure}[H]
 \centering
 \captionsetup{justification=centering} 
 \includegraphics[scale = 0.3, left]{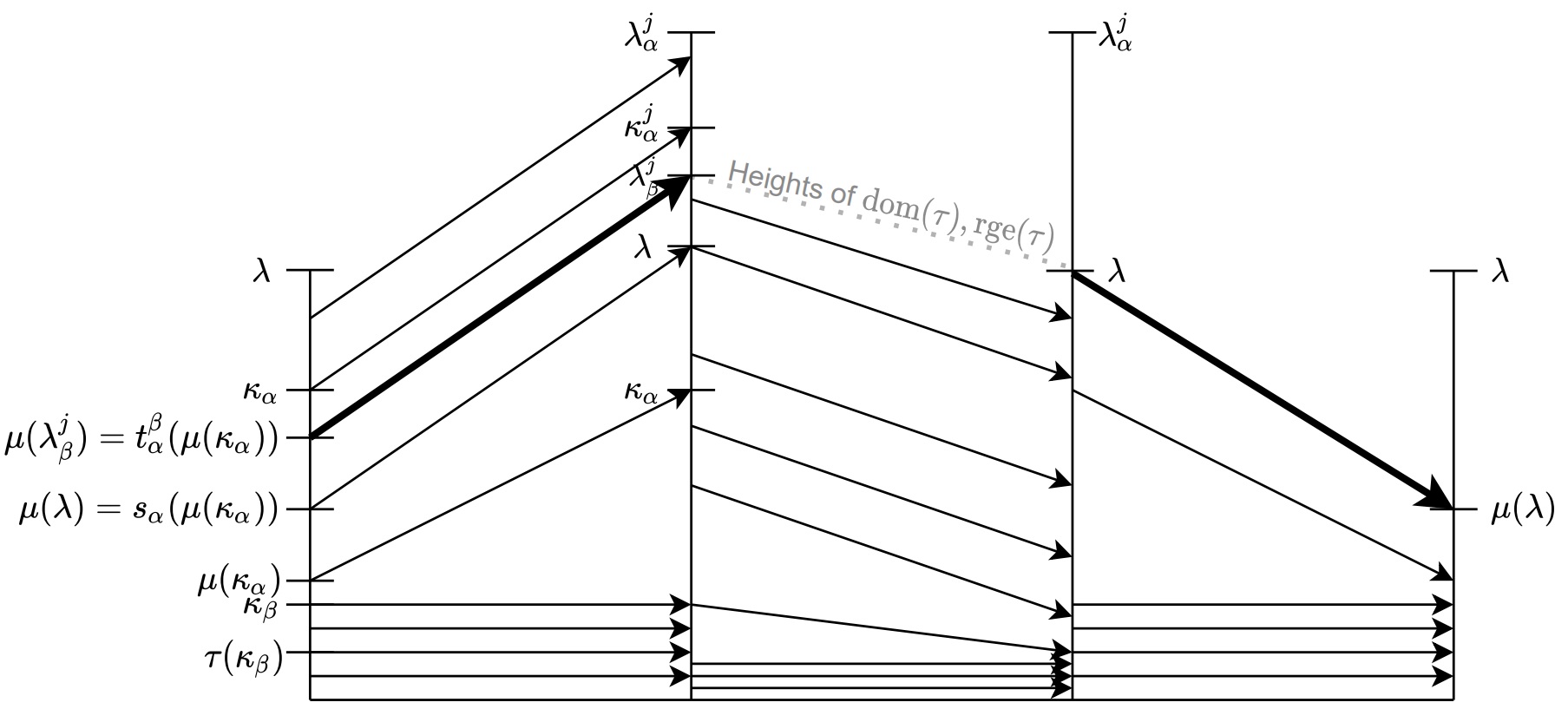}
  \caption{The conjugation $\mu \circ \tau \circ \mu^{-1}$ for an $\alpha$-object $\mu$ and a $\tau \in A_\beta$}
  \label{conjugate}
\end{figure}

In Figure \ref{conjugate}, we exhibit the situation when $\beta<\alpha<\eta$, $\tau \in A_\beta \in E_\beta(d_\beta)$, $\mu \in A_\alpha \in E_\alpha(d_\alpha)$.
From left to right, the diagram shows the maps $\mu{-1}$, $\tau$, $\mu$, respectively.
We can see from the bold arrows and the gray dot-line that the resulting conjugation is a partial function from $t_\alpha^\beta(\mu(\kappa_\alpha))$ to $s_\alpha(\mu(\kappa_\alpha))$.
A similar diagram, which is obtained by replacing the middle part of the diagram in Figure \ref{conjugate} by the function $f$ in Definition \ref{defnsquishablefunction}, also explains the situation in Definition \ref{defnsquishablefunction}.

\begin{defn}

Let $\gamma_0$ is be regular and $\cf(\gamma_1) > \gamma_0$, define a poset $\mathbb{A}(\gamma_0,\gamma_1)$ as the collection of functions $f$ such that $\dom(f) \subseteq \gamma_1$, $\rge(f) \subseteq \gamma_0$, and $|f| \leq \gamma_0$.
Define $f \leq g$ iff $f \supseteq g$.
We also define a Cohen forcing with a certain restriction on the range.

\end{defn}

The forcing is $\gamma_0^+$-closed and is $\gamma_0^{++}$-c.c., so $\mathbb{A}(\gamma_0,\gamma_1)$ preserves all cardinals and cofinalities. 
The forcing $\mathbb{A}(\gamma_0,\gamma_1)$ is equivalent to the Cohen forcing adding $|\gamma_1|$ new subsets of $\gamma_0^+$.

\begin{lemma}
\label{squishablefunction}

Let $\beta<\alpha<\eta$, $d_\beta$ and $d_\alpha$ be $\beta$ and $\alpha$ domains respectively, and assume that $d_\beta \subseteq d_\alpha$.
Let $f \in \mathbb{A}(\lambda,\lambda_\beta^j)$ with $\dom(f)=d_\beta$.
Let $B$ be the collection of $\mu \in E_\alpha(d_\alpha)$ such that by letting $e_\beta=e_{\beta,\alpha}(\mu)$, 

\begin{enumerate}

\item $f[\dom(\mu)] \subseteq \dom(\mu)$.

\item $\mu[d_\beta]$ is a $\beta$-domain with respect to $e_\beta$.

\item $\mu \circ f \circ \mu^{-1} \in \mathbb{A}(s_\alpha(\gamma),t_\alpha^\beta(\gamma))$.

\item if $ \rge(f \restriction \kappa_\beta^j) \subseteq \kappa_\beta$, then $\rge((\mu \circ f \circ \mu^{-1}) \restriction j_{e_\beta}(\kappa_\beta)) \subseteq \kappa_\beta$.

\end{enumerate}

Then $B \in E_\alpha(d_\alpha)$.

\end{lemma}

\begin{proof}

To check the first item, note that $j_{E_\alpha}(f)[j_{E_\alpha}[d_\alpha]]=j_{E_\alpha}[\rge(f)] \subseteq j_{E_\alpha}[\lambda] \subseteq j_{E_\alpha}[d_\alpha]=\dom(\mc_\alpha(d_\alpha))$.
The second item is proved in Lemma \ref{squishablemeasure}.
Next, we prove the third item.
Let $F=\mc_\alpha(d_\alpha) \circ j_{E_\alpha}(f) \circ \mc_\alpha(d_\alpha)^{-1}$.
Then we see that $\gamma \in \dom(F)$ iff $\gamma \in d_\alpha \cap \dom(f)=\dom(f)$, and for $\gamma \in \dom(F)$, $F(\gamma)=f(\gamma)$.
Hence, $F=f$.
Finally, observe that $E_\beta=j_{E_\alpha}(\mu \mapsto E_\beta \restriction e_{\beta,\alpha}(\mu))(\mc_\alpha(d_\alpha))$ and if $\rge(f \restriction \kappa_\beta^j) \subseteq \kappa_\beta$, then $\rge(F \restriction \kappa_\beta^j)= \rge(f \restriction \kappa_\beta^j) \subseteq \kappa_\beta$.

\end{proof}

We define a Cohen subforcing of the forcing of the form $\mathbb{A}(\gamma_0,\gamma_1)$ by adding a constraint on each Cohen condition.
The subforcing will have the same chain condition and closure.

\begin{defn}

Let $\mathbb{B}^{E_\beta}(\lambda,\lambda_\beta^j)$ be the collection of $f \in \mathbb{A}(\lambda,\lambda_\beta^j)$ such that $\rge(f \restriction \kappa_\beta^j) \subseteq \kappa_\beta$.
Let $\gamma<\kappa_\alpha$ be $\alpha$-reflected with respect to $\langle E_\beta: \beta \leq \alpha \rangle$ and $\beta<\alpha$.
Let $\lambda^\prime=s_\alpha(\gamma)$, $(\lambda^\prime_\beta)^j=t_\alpha^\beta(\gamma)$, and $e_\beta=E_\beta \restriction (\lambda^\prime_\beta)^j$.
Define $\mathbb{B}^{e_{\beta}}(\lambda^\prime,(\lambda^\prime_\beta)^j)$ as the collection of $f \in \mathbb{A}(\lambda^\prime,(\lambda^\prime_\beta)^j)$ such that $\rge(f \restriction \kappa_\beta^{j_{e_\beta}}) \subseteq \kappa_\beta$.

\end{defn}

\begin{defn}
\label{defnsquishablefunction}

Let $\beta<\alpha$.
Fix the $\beta$-domain and the $\alpha$-domain $d_\beta$ and $d_\alpha$, respectively, and $d_\beta \subseteq d_\alpha$.
Let $f \in \mathbb{B}^{E_\beta}(\lambda,\lambda_\beta^j)$, $\dom(f)=d_\beta$.
Let $\mu \in \OB_\alpha(d_\alpha)$ with $\gamma=\mu(\kappa_\alpha)$.
We say that $\mu$ is {\em ($\beta$-)squishable with respect to $f$} if $f[\dom(\mu)] \subseteq \dom(\mu)$, $\mu \circ f \circ \mu^{-1} \in \mathbb{B}^{e_{\beta,\alpha}(\mu)}(s_\alpha(\gamma),t_\beta^\alpha(\gamma))$.
The notion $\mu \circ f \circ \mu^{-1}$ is called {\em the conjugation of $f$ by $\mu$}.
\end{defn}

\begin{defn}

Let $\beta<\alpha$.
Fix $\mu \in E_\alpha(d_\alpha)$ for some $\alpha$-domain $d_\alpha$.
Assume $a$ is either $\langle f \rangle$ or $\langle f,A \rangle$ where $f \in \mathbb{B}^{E_\beta}(\lambda,\lambda_\beta^j)$ and $A \in E_\beta(\dom(f_\beta))$.
Then $\mu$ is {\em $a$-squishable} if $\mu$ is $f$ and $A$ squishable (if $A$ exists).

Let $p$ be the sequence $\langle p_\beta: \beta \in [\xi,\alpha)\rangle$ where $p_\beta$ is either $\langle f_\beta \rangle$ or $\langle f_\beta,A_\beta \rangle$, $f_\beta \in \mathbb{B}^{E_\beta}(\lambda,\lambda_\beta^j)$ and $A_\beta \in E_\beta(\dom(f_\beta))$, $\mu \in E_\alpha(d_\alpha)$ for some $\alpha$-domain $d_\alpha$.
Then $\mu$ is {\em $p$-squishble} if for $\beta \in [\xi,\alpha)$, $\mu$ is $p_\beta$-squishable.

\end{defn}

Our forcing will be of Prikry-type.
An important feature of the forcing is that once one performs an extension using a legitimate $\alpha$-object, (one of the requirements for a legitimate object is that it is squishable), then for $\beta<\alpha$, all the $\beta$th components appearing in the forcing will be ``squished".
To present a rough idea, in Lemma \ref{squishablemeasure}, if $\beta<\alpha$ and $A \in E_\beta(d_\beta)$ for some $\beta$-object, then the conjugation of $A_\mu$ by $\mu \in E_\alpha(d_\alpha)$ is of measure-one in $e_{\beta,\alpha}(\mu)(\mu[d_\beta])$.
The height of $e_{\beta,\alpha}(\mu)$ is below  $\kappa_\alpha$.
We will see that the ``squished" $\beta$th components in our forcing will lie in $V_{\kappa_\alpha}$. In fact,  they will live in $V_\gamma$ for some $\gamma<\kappa_\alpha$, where $\gamma$ depends on $\mu(\kappa_\alpha)$.

Later we define forcings and repeat some certain notions very often, so we give notions to compact our description.
If $p$ is of the form $\langle f\rangle$ or $\langle f,A \rangle$, where $f$ is a function, we denote $f, \dom(f),\rge(f)$, and $A$ by $f^p,d^p,r^p$ and $A^p$, respectively.
If $p=\langle p_\beta: \beta<\alpha \rangle$ where $p_\beta=\langle f_\beta\rangle $ or $\langle f_\beta,A_\beta \rangle$, we denote $f_\beta,\dom(f_\beta),\rge(f_\beta)$ and $A_\beta$ by $f_\beta^p,d_\beta^p,r_\beta^p$, and $A_\beta^p$, respectively.
We sometimes remove the superscript $p$ if it is clear from the context.

\section{forcing with a single extender}
\label{forcing1extender}
We begin with the simplest case, only one extender.
We drop all the subscripts $0$ here.
Recall that we have a $(\kappa,\lambda^j)$-extender $E$ (recall $\lambda^j=j(\lambda)$) where $j$ is the elementary embedding $j_E:V \to M=\Ult(V,E)$ witnesses $\kappa$ being $\lambda$-supercompact, $\lambda$ is regular, and there is a function $s: \kappa \to \kappa$ such that $j(s)(\kappa)=\lambda$.

\begin{defn}

We define a forcing $\mathbb{P}_E^\emptyset$ where the conditions are of the form $p=\langle f,A \rangle$ such that $f \in \mathbb{B}^E(\lambda,\lambda^j)$,  and by letting $d=\dom(f)$, we have that $d$ is a ($0$-)domain, and $A \in E(d)$.
For $p=\langle f^p,A^p \rangle$ and $q=\langle f^q,A^q \rangle$, in $\mathbb{P}^{\emptyset}$, we define $ p \leq q$ if $f^p \leq f^q$ and $A^p \restriction d^q\subseteq A^q$ (we call the last relation ``{\em $A^p$ projects down to a subset of $A^q$}").

\end{defn}

\begin{defn}

Define $\mathbb{P}_E^{\{0\}}$ as  the collection of $\langle f \rangle$ for $f$ in $  \mathbb{B}^E(\lambda,\lambda^j)$.
The ordering in $\mathbb{P}_E^{\{0\}}$ is just the usual ordering.

\end{defn}

\begin{defn}
\label{bfnotation}

Let $\mathbb{P}_E=\mathbb{P}_E^\emptyset \cup \mathbb{P}^{\{ 0\}}_E$.

\end{defn}

We drop the subscript $E$ to make notations simpler.
For example, we have seen that we wrote $\lambda^j$ instead of $\lambda^{j_E}$.
We may also write $\mathbb{P}$ to refer to $\mathbb{P}_E$.
Define a direct extension $\leq^*$ on $\mathbb{P}$ as $\leq_{\mathbb{P}^\emptyset} \cup \leq_{\mathbb{P}^{\{0\}}}$.
If $p \in \mathbb{P}^{\emptyset}$, we say that $p$ is {\em pure}, and we write the {\em support} of $p$ as $\supp(p)=\emptyset$.
Otherwise, $p$ is said to be {\em impure}, and we write the support of $p$ as $\supp(p)=\{0\}$.

\begin{defn}

Let $p \in \mathbb{P}^{\emptyset}$ and $\mu \in A^p$.
The {\em one-step extension of $p$ by $\mu$}, denoted by $p+\mu$, is simply just $f^p \oplus \mu$.
Note that $f^p \oplus \mu \in \mathbb{P}^{\{0\}}$.

\end{defn}

\begin{defn}

Define $\leq$ on $\mathbb{P}$ by $p \leq q$ if $p \leq^*q$ or $p$ is a direct extension of a one-step extension of $q$.

\end{defn}

It is easy to check that if $r \leq^* q+\mu$ and $q \leq^* p$, then $r \leq^* p+ (\mu \restriction d^p)$.
Using this fact, the relation $\leq$ is transitive.

\begin{thm}
\label{prikry1}
$\mathbb{P}$ has the Prikry property.

\end{thm}

\begin{proof}

Fix a Boolean value $b$.
If $p$ is impure, the proof is easy.
Suppose $p$ is pure.
Let $\mathbb{B}=\mathbb{B}^{E}(\lambda,\lambda^j)$ (this is exactly $\mathbb{P}^{\{0\}}$, but we want to distinguish them as $\mathbb{P}^{\{0\}}$ refers to impure conditions).
First, we show that there are 

\begin{enumerate}

\item $N \prec (H_\theta, \in ,<)$ for some sufficiently large regular cardinal $\theta$, where $<$ is a well-ordering on $H_\theta$, such that $|N|=\lambda$ and ${}^{<\kappa} N \subseteq N$ such that $\lambda, \lambda^j ,b,p, \mathbb{P} \in N$ and  $f^p,\lambda \subseteq N$.

\item $f^\prime \leq f^p$ which is $N$-generic, meaning for each open dense set $D$ in $N$ with respect to $\mathbb{B}$,  there is $f \in D \cap N$ such that $f^\prime \leq f \leq f^p$ (as a consequence $f^\prime \in D$).

\end{enumerate}

To accomplish the construction, we build an {\em internally approachable} chain of substructures of $(H_\theta, \in , <)$ for some sufficiently large regular cardinal $\theta$, and a well-ordering $<$ on $H_\theta$, namely a sequence $\langle N_\alpha:\alpha<\kappa \rangle$ such that $N_\alpha \prec H_\theta$, ${}^{<\kappa} N_\alpha \subseteq N_{\alpha+1}$, $|N_\alpha|=\lambda$, $\lambda,\lambda^j,b,p,\mathbb{P} \in N_0$, $\lambda, f \subseteq N_0$, $\langle N_\beta:\beta \leq \alpha \rangle \in N_{\alpha+1}$, and if $\alpha$ is limit, $N_\alpha= \bigcup_{\beta<\alpha}N_\alpha$.

Let $N= \bigcup_{\alpha<\kappa} N_\alpha$.
It remains to find $f^\prime$.
We build a decreasing sequence $\langle f^\alpha: \alpha<\kappa \rangle$ of elements in $\mathbb{B} \cap N$ such that $f^\alpha \in N_{\alpha+1}$, $f^0 \leq f^p$, and $f^\alpha$ meets every dense open set in $N_\alpha \cap \mathbb{B}$.
Note that this is possible since $\langle N_\beta: \beta \leq \alpha \rangle \in N_{\alpha+1}$ and $\mathbb{B}$ is $\lambda^+$-closed.
Take $f^\prime$ as a lower bound of $\langle f^\alpha: \alpha<\kappa \rangle$.

Now let $N$ and $f^\prime$ be as described above.
Let $d^\prime=\dom(f^\prime)$.
By our construction,
$d^\prime$ is simply just $N \cap \lambda^j$.
Let $A^\prime \in E(d^\prime)$, $A^\prime$ projects down to a subset of $A^p$.
For $\mu \in A^\prime$, $\dom(\mu) \subseteq d^\prime$, $\rge(\mu) \subseteq \lambda$ and $|\mu|<\kappa$, so $\mu \in N$.

Fix $\mu \in A^\prime$.
Define $D_\mu$ as the collection of $f \in \mathbb{A}(\lambda,\lambda^j)$ such that $\dom(\mu) \subseteq \dom(f)$, and if there is $g \leq f \oplus \mu$ such that $g  \parallel b$, then $f \oplus \mu \parallel b$.

\begin{claim}
\label{denseopen1}
$D_\mu$ is open dense below $f^p$ and is in $N$.

\end{claim}

\begin{claimproof}{(Claim \ref{denseopen1})} 

We actually prove that $D_\mu$ is open dense, which is stronger that the statement.
Clearly $D_\mu$ is open. 
Since $D_\mu$ is defined using parameters in $N$, $D_\mu \in N$.
To check the density of $D_\mu$, let $f \in \mathbb{A}(\lambda,\lambda^j)$.
By the density, assume $\dom(\mu) \subseteq \dom(f)$.
Find a condition $g \leq f \oplus \mu$ such that $g$ decides $b$.
Define $g^\prime$ with $\dom(g^\prime)=\dom(g)$, $g^\prime(\gamma)=f(\gamma)$ if $\gamma \in \dom(f)$, otherwise, $g^\prime(\gamma)=g(\gamma)$.
Clearly $g^\prime \leq f$ and $g^\prime \oplus \mu=g$, so $g^\prime \in D$.

\end{claimproof}{(Claim \ref{denseopen1})}

By genericity, $f^\prime \in D_\mu$.
We have the following property for $f^\prime$ (call it $(\star(f^\prime))$):

\begin{center}

$\boldsymbol{(\star(f^\prime))}$ For $\mu \in A^\prime$, if there is $g \leq f^\prime \oplus \mu$ such that $g$ decides $b$, then $f^\prime \oplus \mu$ decides $b$.

\end{center}

Let $A_0$ be the collection of $\mu \in A^\prime$ such that there is $g \leq f^\prime$ such that $g \oplus \mu \Vdash b$.
Hence, for $\mu \in A_0$, $f^\prime \oplus \mu \Vdash b$.
Let $A_1=A^\prime \setminus A_0$.
It is easy to check that if $\mu \in A_1$ then $f^\prime \oplus \mu \Vdash \neg b$.
Exactly one of $A_0$ and $A_1$ is of measure-one, call the one of measure-one $B$.
Let $p^*=\langle f^\prime,B \rangle$.

\begin{claim}
\label{finalcond1}
$p^* \leq^*p$ and $p^*$ decides $b$.

\end{claim}

\begin{claimproof}{(Claim \ref{finalcond1})}

Clearly $p^* \leq^* p$. 
Let $q \leq p^*$ be such that $q$ decides $b$.
Assume $q$ is impure.
Without loss of generality, assume $q\Vdash b$.
Then $q \leq^*p^*+\mu$ for some $\mu \in B$.
By the property of $(\star(f^\prime))$, $f^\prime \oplus \mu \Vdash b$.
Thus, $B=A_0$ and for every $\mu \in A_0$, $f^\prime \oplus \mu \Vdash b$. 
By the density, $p^* \Vdash b$.

\end{claimproof}{(Claim \ref{finalcond1})}

\end{proof}

Forcing with $\mathbb{P}$ is equivalent to adding $\lambda^j$ new Cohen subsets of $\lambda^+$, while preserving all cardinals and cofinalities.

\section{forcing defined from two extenders}
\label{forcing2extenders}

In this section we deal with the case where the forcing is defined from a sequence of extenders whose sequence length $2$.
The proof of the Prikry property in this section requires the Prikry property of the forcings defined sequences of extenders whose lengths are $1$ (i.e. the forcing in Section \ref{forcing1extender}).
In Section \ref{forcinganyextender}, we define a forcing with any length of a sequence of extenders, including any finite lengths.
The structure of the proof of the Prikry property in this section can apply to forcings defined from sequences of extenders whose sequences have finite lengths greater than $1$, assuming the Prikry property holds for forcings defined from sequences of extenders of shorter finite lengths.

\begin{defn}

A forcing $\mathbb{P}_{\langle E_0,E_1 \rangle}^\emptyset$ consists of conditions of the form $p=\langle p_0,p_1 \rangle$, where for $n=0,1$, $p_n=\langle f_n^p, A_n^p \rangle$, $f_n^p \in \mathbb{B}^{E_n}(\lambda,\lambda_n^j)$, $d_n^p=\dom(f_n^p)$ is an $n$-domain, $A_n \in E_n(d_n^p)$, and $d_0^p \subseteq d_1^p$.
A condition in $\mathbb{P}_{\langle E_0,E_1 \rangle}^\emptyset$ is said to be {\em pure}.

\end{defn}

\begin{defn}

A forcing $\mathbb{P}_{\langle E_0,E_1 \rangle}^{\{0\}}$ consists of conditions $p=\langle p_0,p_1 \rangle$ such that $p_0 \in \mathbb{P}_{E_0}^{\{0\}}$ and $p_1 \in \mathbb{P}_{E_1}^\emptyset$ and $d_0^p \subseteq d_1^p$.

\end{defn}

\begin{defn}
\label{2extender1}
A forcing $\mathbb{P}_{\langle E_0,E_1 \rangle}^{\{1\}}$ consists of a condition $p=\langle p_0,p_1 \rangle$ such that $p_0=\langle f_0^p,A_0^p \rangle$, $p_1=\langle f_1^p \rangle$ such that

\begin{enumerate}

\item $f_1^p \in \mathbb{B}^{E_1}(\lambda,\lambda_1^j)$,  $d_1^p$ is a $1$-domain with respect to $E_1$.

\item $f_1^p(\kappa_1)$ is $1$-reflected for the sequence $\langle E_0,E_1 \rangle$.

\item $f_0^p \in \mathbb{B}^{e_{0,1}(f_1^p)}(\lambda_1(f_1^p),\lambda_{0,1}^j(f_1^p))$
(Recall that $\lambda_{0,1}^j(f_1^p)$ is $t_1^0(f_1^p(\kappa_1))$ where $j_{E_1}(t_1^0)(\kappa_1)=\lambda_0^j$, $e_{0,1}(f_1^p)$ is $E_0 \restriction \lambda_{0,1}^j(f_1^p)$, which means that $\rge(f_0^p \restriction j_{e_{0,1}(f_1^p)}(\kappa_0) \subseteq \kappa_0$.

\item $d_0^p=\dom(f_0^p)$ is a $0$-domain with respect to $e_{0,1}(f_1^p)$.

\item $A_0^p \in e_{0,1}(f_1^p)(d_0^p)$.

\end{enumerate}

Equivalently, $p_0 \in \mathbb{P}^{\emptyset}_{e_{0,1}(f_1^p)}$ and $p_1 \in \mathbb{P}^{\{0\}}_{E_1}$ .

\end{defn}

\begin{defn}
\label{2extender01}
A forcing $\mathbb{P}_{\langle E_0,E_1 \rangle}^{\{0,1\}}$ consists of a condition $p=\langle p_0,p_1 \rangle$ such that $p_n=\langle f_n^p \rangle$ where

\begin{enumerate}

\item $f_1^p \in \mathbb{B}^{E_1}(\lambda,\lambda_1^j)$ and $d_1^p$ is a $1$-domain.

\item $f_1^p(\kappa_1)$ is $1$-reflected with respect to $\langle E_0,E_1 \rangle$. 

\item  $f_0^p \in \mathbb{B}^{e_{0,1}(f_1^p)}(\lambda_1(f_1^p),\lambda_{0,1}^j(f_1^p))$
(we recall some relevant notations in Definition \ref{2extender1}).

\item $d_0^p$ is a $0$-domain with respect to $e_{0,1}(f_1^p)$.

\end{enumerate}

Equivalently, $p_0 \in \mathbb{P}^{\{0\}}_{e_{0,1}(f_1^p)}$ and $p_1 \in \mathbb{P}^{\{0\}}_{E_1}$.

\end{defn}

For a condition $p$ which is described in Definition \ref{2extender1} or \ref{2extender01}, it is automatic that for each condition $p$, $d_0^p \cup r_0^p \subseteq d_1^p$. For $a \subseteq 2$, we define $\leq_{\mathbb{P}_{\langle E_0,E_1 \rangle}^a}$ to be the coordinate-wise ordering.
Let $\mathbb{P}_{\langle E_0,E_1 \rangle }=\cup_{a \subseteq 2} \mathbb{P}_{\langle E_0,E_1 \rangle}^a$.
Define a direct extension relation $\leq^*$ on $\mathbb{P}_{\langle E_0,E_1\rangle}$ to be $\cup_{a \supseteq 2} \leq_{\mathbb{P}_{\langle E_0,E_1 \rangle}^a}$.

We sometimes abbreviate $\mathbb{P}_{\langle E_0,E_1 \rangle}$ as $\mathbb{P}$, and abbreviate other forcings in a similar fashion, i.e. we abbreviate $\mathbb{P}_{\langle E_0,E_1 \rangle}^a$ as $\mathbb{P}^a$, written as $\supp(p)=a$.
Note that $p$ is {\em pure} if $\supp(p)=\emptyset$.
We say that $p$ has support $a$ if $p \in \mathbb{P}^a$.
We say $p$ is {\em impure} if $\supp(p) \neq \emptyset$.

\begin{defn}
\label{extendat0}
Let $p \in \mathbb{P}$, $0 \not \in \supp(p)$ and $\mu \in A_0^p$ is squishable. then we define the {\em one-step extension of $p$ by $\mu$} is a condition $q$, denoted by $p+\mu$, with $\supp(q)=a \cup \{0\}$ such that 

\begin{enumerate}

\item $q_0=\langle f_0^p \oplus \mu \rangle$.

\item $q_1=\langle f_1^p \rangle$ if $1 \not \in \supp(p)$, otherwise $q_1=\langle f_1^p,B_1 \rangle$, where $B_1=\{ \mu_1 \in A_1^p:  \mu_1 \text{ is } \langle f_0^p,A_0^p \rangle \text{-squishable and} \dom(\mu) \cup \rge(\mu) \subseteq \dom(\mu_1)\}$.

\end{enumerate}

\end{defn}

\begin{defn}
\label{onestepext2}
Let $p \in \mathbb{P}$, $1 \not \in \supp(p)$ and $\mu \in A_1^p$ is $p_0$-squishable.
Then {\em one-step extension of $p$ by $\mu$}, denoted by $p+\mu$ is a condition $q \in \mathbb{P}^{a \cup \{1\}}$ such that 

\begin{enumerate}

\item $f_1^q=f_1^p \oplus \mu$.

\item $f_0^q= \mu \circ f_0^p \circ \mu^{-1}$.

\item $A_0^q=\mu \circ (A_0^p)_\mu \circ \mu^{-1}$ (if $A_0^p$ exists). Recall that $(A_0^p)_\mu=\{\tau \in A_0^p: \dom(\tau) \cup \rge(\tau) \subseteq \dom(\mu)\}$.

\end{enumerate}

\end{defn}

$p$ is a $2$-step extension of $q$ if $p=(q+\mu)+\mu^\prime$ for some objects $\mu$ and $\mu^\prime$, and their one-step extensions are legitimate. 
We denote $(q+\mu)+\mu^\prime$ by $q+ \langle \mu, \mu^\prime \rangle$.
Finally, for $p,q \in \mathbb{P}$, we say that $p \leq q$ if $p\leq^* q$, $p$ is a direct extension of a $1$ or $2$-step extension of $q$.
Before we make an analysis further, we give a picture of a situation in Definition \ref{onestepext2}, where $p$ is a pure condition.

 \begin{figure}[H]
 \centering
 \captionsetup{justification=centering} 
 \includegraphics[scale = 0.75]{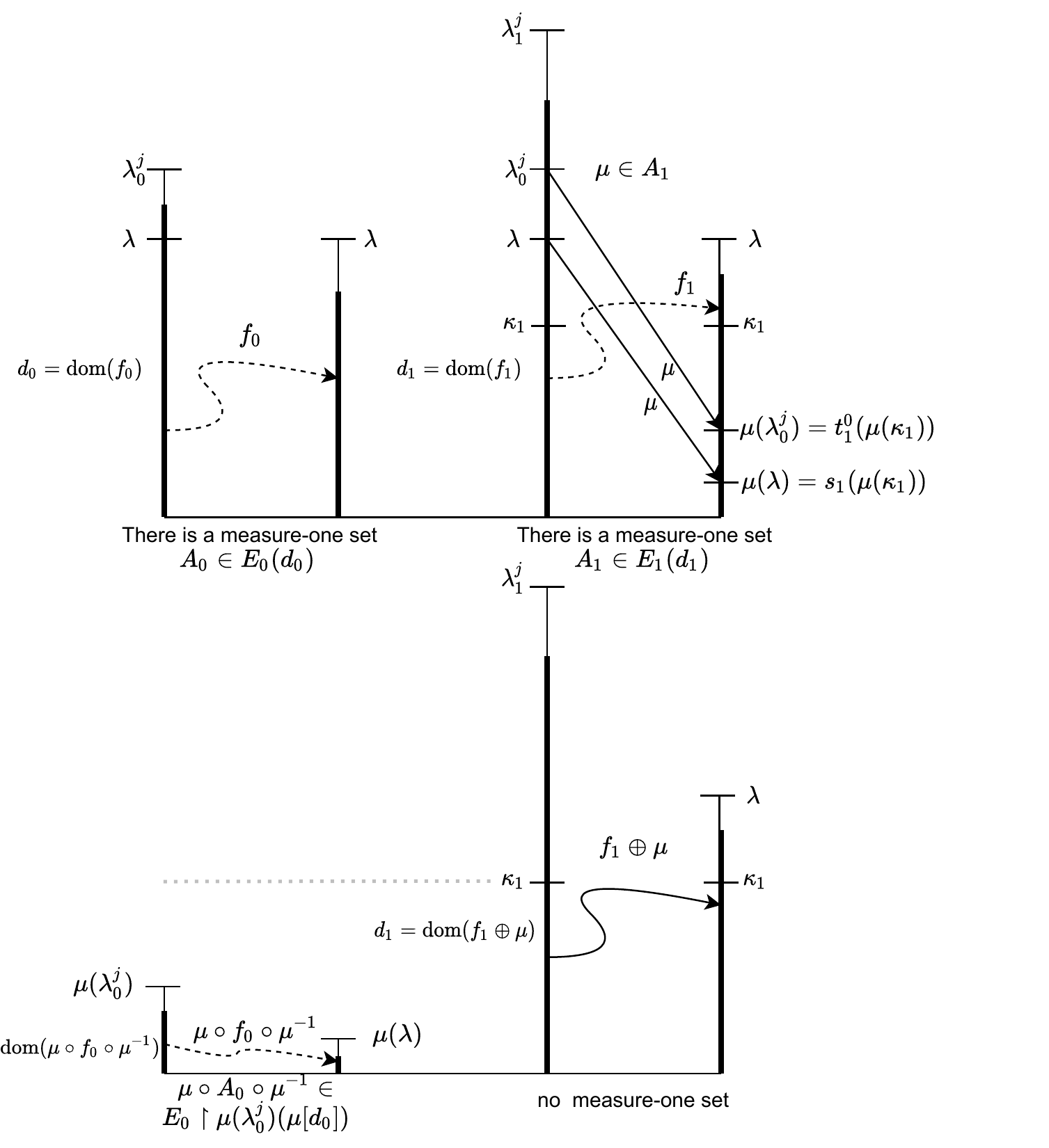}
  \caption{A one-step extension of a pure condition using a $1$-object}
  \label{extend2extenders}
\end{figure}

In Figure \ref{extend2extenders}, the first diagram shows how a pure condition $p$ looks like, while the second diagram is a condition which is a result from performing a $1$-step extension of $p$ by a $1$-object $\mu$.
The bold bars represent how high domains and ranges can be (they do \textbf{not} represent that domains and ranges are initial segments of ordinals, and in fact, they may not).
From the figure, we see that after performing the one-step extension, the first coordinate of the condition lies in $V_{\kappa_1}$ (actually, it lives in $V_{\mu(\lambda_0^j)}$).
The wiggle lines represent how the maps of Cohen parts are like.
The straight lines in the first diagram indicate the map from a few crucial values of $\dom(\mu)$.
We make an analysis on the ordering $\leq$ on $\mathbb{P}$.

\begin{lemma}
\label{feature}
\begin{enumerate}

\item \label{tran} Let $i \in \{0,1\}$.
Let $p$ be a condition such that $i \not \in \supp(p)$ and $q \leq^* p$.
Suppose $\mu \in A_i^q$, $\mu$ is $\langle E_n: n \not \in \supp(q) \rangle$-squishable.
Then $q+\mu \leq^* p+(\mu \restriction d_i^p)$.

\item The ordering $\leq$ is transitive.

\item Suppose $p$ is pure and $q$ is a $2$-step extension of $p$.
Then there are $\mu_i \in A_i^p$ for $i=0,1$ such that $q=p+ \langle \mu_0,\mu_1 \rangle=\langle \mu_1 \circ \mu_0 \circ \mu_1^{-1},\mu_1 \rangle$.

\end{enumerate}

\end{lemma}

\begin{proof}

\begin{enumerate}

\item We will show the case $p$ is pure and $i=1$ since other cases are easier to prove.
Let $\tau = \mu \restriction d_1^p$.
First note that since $q \leq^*p$, $\tau \in A_1^p$. 
It is straightforward to check that since $\mu$ is $q_0$-squishable, $\tau$ is $p_0$-squishable.
Let $r=q+\mu$.
Since $f_0^q \leq f_0^p$ and $\mu \supseteq \tau$,  $f_0^r=\mu \circ f_0^q \circ \mu^{-1} \leq \tau \circ f_0^p \circ \tau^{-1}$.
Since $A_0^q \restriction d_0^p \subseteq A_0^p$ and $\mu \supseteq \tau$, $\mu \circ (A_0^q)_\mu \circ \mu^{-1}$ projects down to a subset of $\tau \circ (A_0^p)_\tau \circ \tau^{-1}$.
Finally, $f_1^r=f_0^q \oplus \mu \leq f_0^p \oplus \tau$, so $r \leq^* p+\tau$.

\item Let $p \leq q \leq r$.
It is trivial if $p \leq^*q$, otherwise, we apply (\ref{tran}) to show that every object in $q$ used in the path to extend $q$ to $p$ corresponds to the restriction of the object to the corresponding domain in $r$, and hence, $p$ is a direct extension of some $n$-step extension of $r$, so $\leq$ is transitive.

\item From the context, the $1$-object we used to extend is always $\mu_1$.
It is important to be aware of the objects we are allowed to use to extend.
Note that there are two possible ways to extend $p$ using $2$ objects.
The first way is to first extend $p$ using a $0$-object $\mu_0$ in $A_0^p$, then extend $p+\mu_0$ further using a $1$-object $\mu_1$ in a subset of $A_1^p$, namely the set $B_1$ defined in Definition \ref{extendat0}.
The second way is to extend by a $1$-object $\mu_1$ first, and then a $0$-object $\mu_1 \circ \mu_0 \circ \mu_1^{-1} \in \mu_1 \circ (A_0^p)_{\mu_1} \circ \mu_1^{-1}$. 
In both cases, $\dom(\mu_0) \cup \rge(\mu_0) \subseteq \dom(\mu_1)$.
From the first way, we see that from Definition \ref{onestepext2}, $\mu_1$ is $\langle f_0^p,A_0^p \rangle$-squishable.
In the second way, since $\mu_1$ is $\langle f_0^p,A_0^p \rangle$-squishable and $\dom(\mu_0) \cup \rge(\mu_0) \subseteq \dom(\mu_1)$, it concludes that $\mu_1$ is $\langle f_0^p \oplus \mu_0 \rangle$-squishable.
In both cases, the orders of the objects to perform one-step extensions can be commuted.
Since the $1$-object we use is always $\mu_1$, it is enough to show that the final conditions of the extensions in both ways are equal at the $0$-th coordinate.
Equivalently, $\mu_1 \circ (f_0^p \oplus \mu_0) \circ \mu_1^{-1}=(\mu_1 \circ f_0^p \circ \mu_1^{-1}) \oplus (\mu_1 \circ \mu_0 \circ \mu_1^{-1})$.
The fact that $\dom(\mu_0) \cup \rge(\mu_0) \subseteq \dom(\mu_1)$ and $\mu_1$ is $p_0$-squishable implies that the equation holds.

\end{enumerate}

\end{proof}

\begin{prop}

Suppose that $p \in \mathbb{P}$ with $1 \in \supp(p)$.
Then $\mathbb{P}/p$ is factored into $\mathbb{P}_0=\mathbb{P}_{e_{0,1}(f_1^p)}$ and $\mathbb{P}_1=\mathbb{B}^{E_1}(\lambda,\lambda_1^j)$.
$\mathbb{P}_0$ is $\lambda_1(f_1^p)^{++}$-c.c. and $(\mathbb{P}_1, \leq^*)$ is  $\lambda^+$-closed.
If $0 \in \supp(p)$, then $\mathbb{P}_0$ is $\lambda_1(f_1^p)^+$-closed.

\end{prop}

\begin{proof}

Note that for each condition in $\mathbb{P}_0$, its Cohen part lives in $\mathbb{B}^{e_{0,1}(f_1^p)}(\lambda_1(f_1^p),\lambda_{0,1}^j(f_1^p))$. 
By a standard $\Delta$-system argument, the poset $\mathbb{B}^{e_{0,1}(f_1^p)}(\lambda_1(f_1^p),\lambda_{0,1}^j(f_1^p))$ is $\lambda_1(f_1^p)^{++}$-c.c., and is $\lambda_1(f_1^p)^+$-closed.
If $g_0,g_1 \in\mathbb{B}^{e_{0,1}(f_1^p)}(\lambda_1(f_1^p),\lambda_{0,1}^j(f_1^p))$ are compatible, say $g \leq g_0,g_1$, $B_0 \in e_{0,1}(f_1^p)(\dom(g_0))$, and $B_1 \in e_{0,1}(f_1^p)(\dom(g_1))$, then we can find $B$ such that $\langle g,B \rangle \leq \langle g_0, B_0 \rangle$, $\langle g_1,B_1 \rangle$ by finding $B \in e_{0,1}(f_1^p)(\dom(g))$ which projects down to subsets of $B_0$, $B_1$.

\end{proof}

\begin{lemma}
\label{integration1}

Let $p \in \mathbb{P}$ and $1 \not \in \supp(p)$.
Let $f^\prime \leq f_1^p$ with $d^\prime=\dom(f^\prime)$.
Let $A^\prime \in E_1(d^\prime)$, $A^\prime$ projects down to a subset of $A_1^p$.
Assume that for each $\mu \in A^\prime$ with $\mu \restriction d_1^p$ being $p_0$-squishable, we have $t(\mu) \leq^* (p+(\mu \restriction d_1^p))_0$ in $\mathbb{P}_{e_{0,1}(\mu)}$.
Then there is a condition $q \leq^*p$ such that

\begin{enumerate}

\item $f_1^q \leq f^\prime$, and $A_1^q$ projects down to $A^\prime$.

\item for $\tau \in A_1^q$ with $\mu=\tau \restriction d^\prime$, $(q+ \tau)_0=t(\mu)$.

\end{enumerate}

\end{lemma}

\begin{proof}

Let $d^\prime=\dom(f^\prime)$.
We may assume $p$ is pure (the case $\supp(p)=\{0\}$ is slightly simpler).
Let $p_0^*=j_{E_1}(\mu \mapsto t(\mu))(\mc_1(d^\prime))$.
Write $p_0^*=\langle f_0^*,A_0^* \rangle$ and $d_0^*=\dom(f_0^*)$.
Notice that

\begin{align*}
p_0^*  & \leq^* j_{E_1}(\mu \mapsto (p+(\mu \restriction d_1^p))_0)(\mc_1(d^ \prime)) \\ & =(j_{E_1}(p)+\mc_1(d^\prime)  \restriction j_{E_1}(d_1^p))_0\\
 &=(j_{E_1}(p)+ \mc_1(d_1^p))_0 \\ &= \langle \mc_1(d_1^p) \circ j_{E_1}(f_0^p) \circ \mc_1(d_1^p)^{-1}, \mc_1(d_1^p) \circ j_{E_1}[A_0^p] \circ \mc_1(d_1^p)^{-1} \rangle.
\end{align*}

The last equation follows from Lemma \ref{aux}. 
Exact calculations in Lemma \ref{squishablemeasure} and Lemma \ref{squishablefunction} show that

\begin{center}

$\mc_1(d_1^p) \circ j_{E_1}[A_0^p] \circ \mc_1(d_1^p)^{-1} =A_0^p$

\end{center}

and

\begin{center}

$\mc_1(d_1^p) \circ j_{E_1}(f_0^p) \circ \mc_1(d_1^p)^{-1}=f_0^p$.

\end{center}

Hence, $p_0^*$ is a direct extension of $p_0$.

Let $f_1^* \in \mathbb{B}^{E_1}(\lambda,\lambda_1^j)$ with $\dom(f_1^*)=d^\prime \cup d_0^*=:d_1^*$ be such that $f^*_1(\gamma)=f^\prime(\gamma)$ if $\gamma \in d^\prime$, otherwise $f^*_1(\gamma)=0$.

Since $d_0^* \subseteq d_1^*$, we have that

\begin{center}

$\mc_1(d_1^*) \circ j_{E_1}(f_0^*) \circ \mc_1(d_1^*)^{-1}=f_0^*$ 

\end{center}

and

\begin{center}

$\mc_1(d_1^*) \circ j_{E_1}[A_0^*] \circ \mc_1(d_1^*)^{-1} = A_0^*$.

\end{center}

Let $A_1^*$ be the collection of $\tau \in E_1(d_1^*)$ such that the following hold:

\begin{enumerate}

\item $\tau \restriction d^\prime \in A^\prime$, which implies that $\tau  \restriction d_0^p \in A_0^p$.

\item by letting $\mu=\tau \restriction d^\prime$, we have that $\tau \circ f_0^* \circ \tau^{-1}=f_0^{t(\mu)}$ and $\tau \circ (A_0^*)_\tau \circ \tau^{-1} = A_0^{t(\mu)}$.

\end{enumerate}

Then $A_1^*$ is of measure-one.
Let $q= \langle p_0^*,\langle f_1^*,A_1^* \rangle \rangle$.
The condition $q$ is as required.

\end{proof}

\begin{thm}

$\mathbb{P}$ has the Prikry property.

\end{thm}

\begin{proof}

Let $b$ be a Boolean value, and $p \in \mathbb{P}$.
Let $\mathbb{B}_1=\mathbb{B}^{E_1}(\lambda,\lambda_1^j)$.
We will divide into two cases.

\textbf{\underline{CASE I}: $1 \in \supp(p)$.}
Note that $\theta:=|\mathbb{P}_{e_{0,1}(f_1^p)}| <\kappa_1$. Enumerate $\mathbb{P}_{e_{0,1}(f_1^p)}$ as $\{t_\alpha: \alpha<\theta \}$.
Build a decreasing sequence $\{f_{1,\alpha}:\alpha \leq \theta\}$ recursively such that 

\begin{enumerate}

\item $f_{1,0} =f_1^p$.

\item for $\alpha \leq \theta$ limit, let $f_{1,\alpha}=\cup_{\beta<\alpha} f_{1,\beta}$.

\item for $\alpha=\beta+1$, if there is $g\leq f_{1,\beta}$ such that $t_\beta ^\frown \langle g \rangle$ decides $b$, then $t_\beta^\frown \langle f_{1,\alpha} \rangle$ decides $b$.

\end{enumerate}

The construction is straightforward, and can be proceeded until and including the stage $\theta$ since $\mathbb{B}_1$ is $\lambda^+$-closed and $\theta<\kappa_1$.
Let $f_1^\prime=f_{1,\theta}$.
Note that $e_{0,1}(f_1^p)=e_{0,1}(f_1^\prime)$ since $f_1^p(\kappa_1)=f_1^\prime(\kappa_1)$.
We record the following property about $f_1^\prime$ (call it $(\star(f_1^\prime))$): 

\begin{center}

$\boldsymbol{(\star(f_1^\prime))}$ For each $t \in \mathbb{P}_{e_{0,1}(f_1^p)}$, if there is $g \leq f_1^\prime$
such that $t^\frown \langle g \rangle$ decides $b$, then $t ^\frown \langle f_1^\prime \rangle$ decides $b$.

\end{center}

Let $\dot{G}$ be the canonical name for a generic object of the forcing $\mathbb{P}_{e_{0,1}(f_1^p)}$.
By the Prikry property for a forcing defined from one extender, let $p_0^* \leq^* p_0$ such that $p_0^*$ decides the following statement:

\begin{center}

$\varphi \equiv \exists p^\prime \in \dot{G}(p^\prime {}^\frown \langle f_1^\prime \rangle \parallel b)$.

\end{center}

If the decision is positive, we may extend $p_0^*$ further so that exactly one of the following holds:

\begin{itemize}

\item $p_0^* \Vdash \exists p^\prime \in \dot{G}(p^\prime {}^\frown \langle f_1^\prime \rangle \Vdash b)$.

\item $p_0^* \Vdash \exists p^\prime \in \dot{G}(p^\prime {}^\frown \langle f_1^\prime \rangle \Vdash \neg b)$.

\end{itemize}

If the decision is negative, by extending further, assume $p_0^*$ is such that $p_0^* \Vdash \nexists p^\prime \in \dot{G}(p^\prime{}^\frown \langle f_1^\prime \rangle \Vdash b)$.
Let $p^*=p_0^* {}^\frown \langle f_1^\prime \rangle$.

\begin{claim}
\label{case1prik2}
$p^* \leq^* p$ and $p^*$ decides $b$.

\end{claim}{(\ref{case1prik2})}

\begin{claimproof}{(Claim \ref{case1prik2})}

Clearly $p^* \leq^* p$.
Let $q \leq p^*$ be a condition deciding $b$, so $q_0 \leq p_0^*$.
Without loss of generality, assume $q \Vdash b$.
By $\star(f_1^\prime)$, $q_0^\frown \langle f_1^\prime \rangle \Vdash b$.
We claim that $p_0^* \Vdash \varphi$. 
Suppose not, let $G$ be a generic for $\mathbb{P}_{e_{0,1}(f_1^p)}$ containing $q_0$, so $p_0^* \in G$.
Then there is no $r \in G$ such that $r^\frown \langle f_1^\prime \rangle$ decides $b$, but $q_0^\frown \langle f_1^\prime \rangle \Vdash b$, so we reach a contradiction.

Similar proof eliminates the case of positive decision with respect to the Boolean value $\neg b$.
Hence, we conclude that $p_0^* \Vdash \exists p^\prime \in \dot{G}(p^\prime {}^\frown \langle f_1^* \rangle \Vdash b)$. 
We now show that $p^* \Vdash b$.
Suppose not, let $r \leq p^*$ be such that $r \Vdash \neg b$.
By the property $(\star(f_1^\prime))$, $r_0^\frown \langle f_1^* \rangle \Vdash \neg b$.
Let $G$ be a generic containing $r_0$, so containing $p_0^*$, find $r' \in G$ such that $r' {}^\frown \langle f_1^* \rangle \Vdash b$.
Now if $\tilde{r} \leq r',r_0$ and $\tilde{r} \in G$, we have that $\tilde{r} ^\frown \langle f_1^* \rangle \Vdash b,\neg b$, which is a contradiction.
Hence, the proof is done.

\end{claimproof}{(Claim \ref{case1prik2})}

\textbf{\underline{CASE II}: $1 \not \in \supp(p)$.}
Let $N \prec H_\theta$ for some sufficiently large regular cardinal $\theta$ such that $N$ is internally approachable in the way described as in Theorem \ref{prikry1}, witnessed by a sequence of elementary submodels of length $\kappa_1$, $|N|=\lambda$,  ${}^{<\kappa_1} N \subseteq N$. 
, $\lambda,\lambda_1^j,b,p, \mathbb{P} \in N$,  $f_1^p,\lambda \subseteq N$, and there is  $f_1^\prime \leq f_1^p$ be $N$-generic over $\mathbb{B}_1$, and $d_1^\prime=\dom(f_1^\prime)=N \cap \lambda_1^j$. 
The construction of such $N$ and $f_1^\prime$ is similar to  the proof in Theorem \ref{prikry1}.
Let $A_1^\prime \in E_1(d_1^\prime)$, $A_1^\prime \restriction d_1^p \subseteq A_1^p$.
For $\mu \in A_1^\prime$, $\dom(\mu) \subseteq d^\prime_1$, $\rge(\mu) \subseteq \lambda$, and $|\mu|<\kappa_1$, so $\mu \in N$.

Fix $\mu \in A_1^\prime$.
Define $D_\mu$ as the collection of $f \in \mathbb{B}_1$ such that $\dom(\mu) \subseteq \dom(f)$, and for every $t \in \mathbb{P}_{e_{0,1}(\mu)}$, if there is $g \leq f_1^p \oplus \mu$ such that $t^\frown \langle g \rangle$ decides  $b$, then $t^\frown \langle f^\prime \oplus \mu \rangle$ also decides $b$.

\begin{claim}
\label{finalcond11}
$D_\mu$ is open dense below $f_1^p$ and is in $N$.

\end{claim}

\begin{claimproof}{(\ref{finalcond11})}

Clearly $D_\mu$ is open. 
Since $D_\mu$ is defined using parameters in $N$, $D_\mu \in N$.
To show density of $D_\mu$, let $f\leq f_1^p$.
We may assume $\dom(\mu) \subseteq \dom(f)$.
Note that $\nu:=|\mathbb{P}_{e_{0,1}(\mu)}|<\kappa_1$.
Enumerate $\mathbb{P}_{e_{0,1}(\mu)}$ as $\{ t_\alpha: \alpha<\nu \}$.
Build a sequence $\langle f_{1,\alpha}: \alpha \leq \nu \rangle$ such that 

\begin{enumerate}

\item $f_{1,0}=f$.

\item for $\alpha \leq \nu$ limit, let $f_{1,\alpha}=\cup_{\beta<\alpha} f_{1,\beta}$.

\item for $\alpha=\beta+1$, find $f_{1,\alpha} \leq f_{1,\beta}$ such that if there is a $g \leq f_{1,\beta} \oplus \mu$ such that $t_\alpha {}^\frown \langle g \rangle$ decides $b$, then $t_\alpha {}^\frown \langle f_{1,\alpha} \oplus \mu \rangle$ decides $b$.

\end{enumerate}

Note that $\mathbb{B}_1$ is $\lambda^+$-closed
We can apply the argument from the proof of Claim \ref{denseopen1} for a construction on the successor stages.
Then $f_1^* :=f_{1,\nu} \in D_\mu$.

\end{claimproof}{(\ref{finalcond11})}

By a genericity of $f_1^\prime$, we have that $f_1^\prime \in D_\mu$.
We record the following property for $f_1^\prime$ (and call it $(\star(f_1^\prime))$): 

\begin{center}

$\boldsymbol{(\star(f_1^\prime))}$ For each $\mu \in A_1^\prime$ and $t \in \mathbb{P}_{e_{0,1}(\mu)}$, if there is $g \leq f_1^\prime \oplus \mu$ such that $t^\frown \langle g \rangle$ decides $b$, then so is $t^\frown \langle f_1^\prime \oplus \mu \rangle$.

\end{center}

Fix $\mu \in A_1^\prime$.
Let $\dot{G}$ be the canonical name for the forcing $\mathbb{P}_{e_{0,1}(\mu)}$.
By the Prikry property for a forcing defined from one extender, let $t(\mu) \leq^* (p+ (\mu \restriction d_1^p))_0$ such that $t(\mu)$ decides the following statement:

\begin{center}

$\varphi \equiv \exists t \in \dot{G}(t^\frown \langle f_1^\prime \oplus \mu \rangle \parallel b)$.

\end{center}

We may extend $t(\mu)$ further under the direct extension so that exactly one of the following decisions on $t(\mu)$ holds:

\begin{itemize}

\item $t(\mu) \Vdash \exists t \in \dot{G}(t^\frown \langle f_1^\prime \oplus \mu \rangle \Vdash b)$.

\item $t(\mu) \Vdash \exists t \in \dot{G}(t^\frown \langle f_1^\prime\oplus \mu \rangle \Vdash \neg b)$.

\item $t(\mu) \Vdash \nexists t \in \dot{G}(t^\frown \langle f_1^\prime\oplus \mu \rangle \parallel b)$.

\end{itemize}

Shrink $A_1^\prime$ so that every $\mu \in A_1^\prime$ makes the same decision.
Apply Lemma \ref{integration1} for $f^\prime:=f_1^\prime \leq f_1^p$, $A^\prime:=A_1^\prime \in E_1(\dom(f_1^\prime))$ and $t(\mu)$ for each $\mu \in A_1^\prime$, to find $q \leq^* p$ such that 

\begin{enumerate}

\item $f_1^q \leq f_1^\prime$.

\item For $\tau \in A_1^q$ with $\mu=\tau \restriction d_1^\prime$, $(q+\tau)_0=t(\mu)$.

\end{enumerate}

\begin{claim}
\label{finalcond12}
$q$ decides $b$.

\end{claim}

\begin{claimproof}{(Claim \ref{finalcond12})}

Let $r \leq q$ such that $r$ decides $b$. Without loss of generality, assume $r \Vdash b$.
and $r \leq q+\tau$ for some $\tau \in A_1^q$.
Let $\mu=\tau \restriction d_1^\prime$.
Hence, $r_0 \leq t(\mu)$.
Since $f_1^r \leq f_1^\prime \oplus \mu$, by $(\star(f_1^\prime))$, $r_0^\frown \langle f_1^\prime \oplus \mu \rangle \Vdash b$.
A proof which is similar to the proof in case I of Claim \ref{case1prik2}  shows that $t(\mu) \Vdash \exists t \in \dot{G} (t^\frown \langle f_1^* \rangle \Vdash b)$.
Hence, for every $\tau^\prime \in A_1^q$ with $\mu^\prime= \tau^\prime \restriction d_1^\prime$, $t(\mu^\prime) \Vdash \exists t \in \dot{G} (t^\frown \langle f_1^* \oplus \mu \rangle \Vdash b)$.
Again, an argument which is similar to the argument in the proof in case I of Claim \ref{case1prik2} shows that $t(\mu^\prime)^\frown \langle f_1^\prime \oplus \mu^\prime \rangle \Vdash b$.
Thus, by the density, $q \Vdash b$.

\end{claimproof}{(Claim \ref{finalcond12})}

\end{proof}

\section{forcing with countably infinite extenders}
\label{forcingomegaextender}

Assume the result from Theorem \ref{deriveext} for $\eta=\omega$.
As mentioned in Section \ref{forcing2extenders}, we use the Prikry property of the forcings defined from sequences of extenders of finite lengths.

\begin{defn}

A forcing $\mathbb{P}_{\langle E_n: n<\omega \rangle}$ consists of conditions $p$, and a {\em support of $p$}, which is $\supp(p) \in [\omega]^{<\omega}$, of the form $p=\langle p_n: n< \omega \rangle$, where

\begin{align*}
p_n=
\begin{cases}
\langle f_n \rangle & \text{ if } n \in \supp(p), \\
\langle f_n,A_n \rangle & \text{ otherwise},
\end{cases}
\end{align*}
and for each $n<\omega$ such that $\supp(p) \setminus (n+1) \neq \emptyset$, let $n^*=\min(\supp(p) \setminus (n+1))$, otherwise, $n^*$ is undefined, then the following hold:

\begin{enumerate}

\item for each $n$ where $n^*$ does not exist,

\begin{itemize}
  
\item if $n \in \supp(p)$ (which is exactly when $n=\max(\supp(p))$), then $f_n \in \mathbb{B}^{E_n}(\lambda,\lambda_n^j)$.
 $d_n=\dom(f_n)$ is an $n$-domain with respect to $E_n$, and if $n>0$, $f_n(\kappa_n)$ is $n$-reflected for the sequence $\langle E_m: m \in [\max\{0,\max(\supp(p) \cap n)\},n) \rangle$.

\item if $n \not \in \supp(p)$ (which is exactly when $n>\max(\supp(p))$), 

\begin{itemize}

\item $f_n \in \mathbb{B}^{E_n}(\lambda,\lambda_n^j)$.

\item let $d_n=\dom(f_n)$ is an $n$-domain with respect to $E_n$, and $A_n \in E_n(d_n)$.

\end{itemize}

\end{itemize}

\item for each $n$ where $n^*$ exists, 

\begin{itemize}

\item if $n \in \supp(p)$ (which is exactly when $n=\max(\supp(p) \cap n^*)$), then $f_n \in \mathbb{B}^{e_{n,n^*}(f_{n^*})}(\lambda_{n^*}(f_{n^*}),\lambda_{n,n^*}^j(f_{n^*}))$, $d_n:=\dom(f_n)$ is an $n$-domain with respect to $e_{n,n^*}(f_{n^*})$, and if $n>0$, $f_n(\kappa_n)$ is $n$-reflected for the sequence $\langle e_{m,n^*}(f_{n^*}): m \in [\max\{0,\max(\supp(p) \cap n)\},n )\rangle$.
Recall that $\lambda_{n^*}(f_{n^*})=s_{n^*}(f_{n^*}(\kappa_{n^*}))$ where $j_{E_{n^*}}(s_{n^*})(\kappa_{n^*})=\lambda$, $\lambda_{n,n^*}^j(f_{n^*})=t_{n^*}^n(f_{n^*}(\kappa_{n^*}))$ where $j_{E_{n^*}}(t_{n^*}^n)(\kappa_{n^*})=\lambda_n^j$ and $\lambda_n^j=j_{E_n}(\lambda)$, and $e_{n,n^*}(f_{n^*})=E_n \restriction t_{n^*}^n(f_{n^*}(\kappa_{n^*}))$.

\item if $n \not \in \supp(p)$ (which is exactly when $n \in (\max\{\max(\supp(p) \cap n^*),0\},n^*)$), 

\begin{itemize}

\item $f_n \in \mathbb{B}^{e_{n,n^*}(f_{n^*})}(\lambda_{n^*}(f_{n^*}),\lambda_{n,n^*}^j(f_{n^*}))$, and $d_n:=\dom(f_n)$ is an $n$-domain with respect to $e_{n,n^*}(f_{n^*})$.

\item  $\rge(f_n^p \restriction j_{e_{n,n^*}(f_{n^*})}(\kappa_n)) \subseteq \kappa_n$.

\item $A_n \in e_{n,n^*}(f_{n^*})(d_n)$.

\end{itemize}

\end{itemize}

\item for $\max(\supp(p)) \leq m < n < \omega$, $d_m \subseteq d_n$.

\item if $m<\omega$ is such that $m^*$ exists, then for each $m<n<m^*$, $d_m \subseteq d_n$.

\end{enumerate}

\end{defn}

A condition $p$ is {\em pure} if $\supp(p)=\emptyset$.
Otherwise, $p$ is said to be {\em impure}.

\begin{defn}

For $p,q \in \mathbb{P}$.
We say that $p$ is a {\em direct extension} of $q$, denoted by $p \leq^*q$ if

\begin{enumerate}

\item $\supp(p)=\supp(q)$.

\item for all $n<\omega$, $f_n^p\leq f_n^q$.

\item for $n \not \in \supp(p)$, $A_n^p \restriction d_n^q \subseteq A_n^q$.

\end{enumerate}

\end{defn}

Notice that if $p \leq^*q$, then for $n \in \supp(p)$, $f_n^p(\kappa_n)=f_n^q(\kappa_n)$.
Thus, $\lambda_n(f_n^p)=\lambda_n(f_n^q)$, and for $m<n$, $\lambda_{m,n}^j(f_n^p)=\lambda_{m,n}^j(f_n^q)$, and $e_{m,n}(f_n^p)=e_{m,n}(f_n^q)$.

\begin{defn}

Let $p \in \mathbb{P}$ and $n \not \in \supp(p)$.
Let $\mu \in A_n^p$ be $\langle p_m: m \in [\max\{0,\max(\supp(p) \cap n)\},n) \rangle$-squishable.
The {\em one-step extension of $p$ by $\mu$} is the condition $q$, denoted by $p+\mu$ such that

\begin{enumerate}

\item $\supp(q)=\supp(p) \cup \{n\}$.

\item for $m<\max\{\max(\supp(p) \cap n),0\}$,  $q_m=p_m$.

\item if $n^*=\min(\supp(p) \setminus (n+1))$ exists, then for $m \in (n,n^*)$, $q_m=\langle f_m^p, B_m \rangle$ where $B_m = \{\tau \in A_m^p: \dom(\mu) \cup \rge(\mu)  \subseteq \dom( \tau) \text{ and } \tau \text{ is } \langle f_n^p,A_n^p \rangle \text{-squishable}  \}$, and for $m \geq n^*$, $q_m=p_m$.

\item if $n^*$ does not exist, then for $m>n$, $q_m=\langle f_m^p, B_m \rangle$ where $B_m = \{\tau \in A_m^p: \dom(\mu) \cup \rge(\mu) \subseteq \dom( \tau)\text{ and } \tau \text{ is } \langle f_n^p,A_n^p \rangle \text{-squishable} \}$. 

\item $f_n^q=f_n^p \oplus \mu$.

\item for $m \in [\max\{\max(\supp(p) \cap n),0\},n)$, $f_m^q=\mu \circ f_m^p \circ  \mu^{-1}$ and $A_m^q=\mu \circ (A_m^p)_\mu \circ \mu^{-1}$.

\end{enumerate}

\end{defn}

 \begin{figure}[H]
 \centering
 \captionsetup{justification=centering} 
  \centering
 \includegraphics[width=435 pt,height=262 pt]{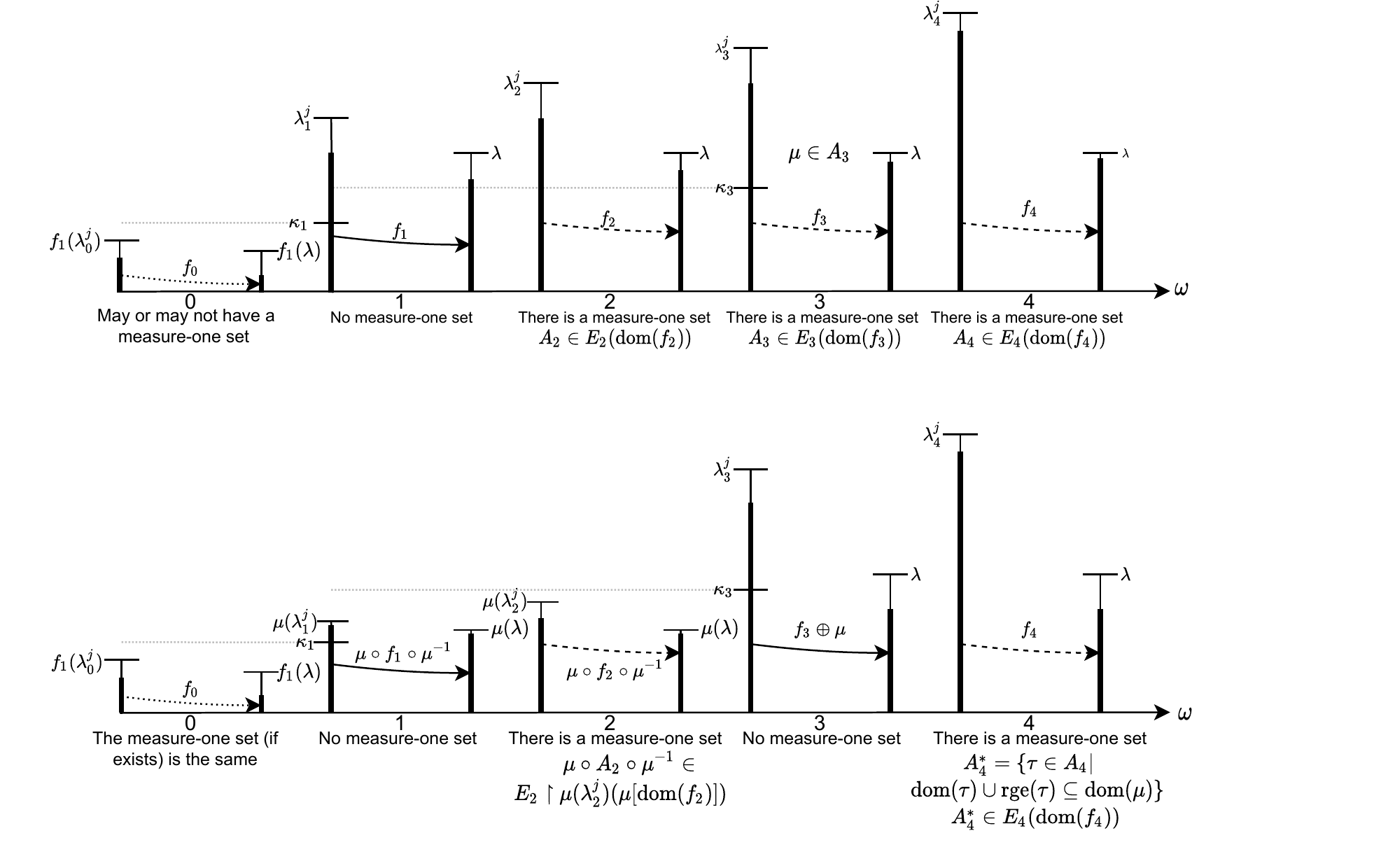}
  \caption{A one-step extension of a condition $p$ with $\{1\} \subseteq \supp(p) \subseteq \{0, 1\}$ using a $3$-object $\mu$}
  \label{extendomegaextenders}
\end{figure}

We exemplify on how one-step extension works.
From Figure \ref{extendomegaextenders}, we start with a condition $p$ with $\{1\} \subseteq \supp(p) \subseteq \{0,1\}$, which is exhibited at the upper-half of the figure.
Each $p_n$ is either $\langle f_n \rangle$ or $\langle f_n,A_n \rangle$.
Fix a $3$-object $\mu$. 
The condition $p+\mu$ is shown at the lower-half of the figure.
The $n$-th columns show how the $n$-th coordinate of the conditions look like.
Thick bars represent the heights of domains and ranges can be (note that they do \textbf{not} represent that domains and ranges are initial segments of ordinals, and in fact, they may not). 
The solid arrows represents maps of Cohen functions lying in the condition on the coordinates which are in the support of $p$, dash-line arrows represent otherwise.
We distinguish $f_0$ by a dot-line arrow, since we $0$ may or may not be in the support of $p$.
The gray line in the upper-half figure emphasizes the fact that $p_0 \in V_{\kappa_1}$.
$(p+\mu)_0$ is just $p_0$, so it remains unchanged.
$p_1$ and $p_2$ are now conjugated by $\mu$.
As a result, $(p+\mu)_1$ and $(p+\mu)_2$ live in $V_{\kappa_3}$ (which are emphasized by the upper gray line in the lower-half figure).
$f_3$ is (partially) overwritten by $\mu$.
For $n>4$, the measure-one set shrinks to make sure that every $n$-object keeps information about $\mu$ in a certain way.

We define an $n$-step extension recursively as follows:
 $p$ is an {\em $n$-step extension of $q$} for $n>1$ if $p=(q+\langle \mu_0, \dots, \mu_{n-2} \rangle) + \mu_{n-1}$, under the condition that for $i<n$, $\mu_i$ is legitimate to perform a $1$-step extension into $ q+\langle \mu_0, \dots, , \mu_{i-1} \rangle$.
Define $p \leq q$ if $p$ is a direct extension of some $n$-step extension of $q$ ($n$ can be $0$). 
The proof of the following lemma is similar to the proof in Lemma \ref{feature}, except that our forcing is more complicated.

\begin{lemma}
\label{transitivity2}
\begin{enumerate}

\item Let $n<\omega$.
Let $p$ be a condition such that $n \not \in \supp(p)$ and $q \leq^* p$.
Suppose $\mu \in A_n^q$ and $q+\mu$ is valid.
Then $q+\mu \leq^* p+(\mu \restriction d_n^p)$.

\item The ordering $\leq$ is transitive.

\item \label{commute} Suppose $p$ is pure and $q$ is a $n$-step extension of $p$ at coordinates $k_0< \dots < k_{n-1}$.
Then there are $\mu_i \in A_{k_i}^p$ for $i<n$ such that $q=p+ \langle \mu_0, \dots, \mu_{n-1} \rangle$.
As a consequence, the order of the objects we use to extend $p$ to $q$ can be commuted modulo squishing.

\end{enumerate}

\end{lemma}

We explain the meaning of (\ref{commute}) in Lemma \ref{transitivity2} by giving an example.
Suppose $q$ is a $3$-step extension of $p$ with $\tau_2,\tau_0$, and $\tau_1$, respectively, where $\tau_2 \in A_{k_2}^p$, $\tau_0 \in A_{k_0}^{p+\tau_2}$, $\tau_1 \in A_{k_1}^{p+\langle \tau_2,\tau_0\rangle}$, and $k_0<k_1<k_2$.
Let $\mu_2=\tau_2$.
Then one can check that for $i=0,1$, $\tau_i$ is $\mu_2 \circ \mu_i \circ \mu_2^{-1}$ for some $\mu_i \in A_{k_i}^p$.
With some calculations, it is true that $q=p +\langle \mu_0,\mu_1,\mu_2 \rangle$. 
In fact, for every permutation $\sigma$ of $\{k_0,k_1,k_2\}$, there is a unique way to perform three $1$-step extensions of $p$ where the $1$-step extensions are performed in the order of $\sigma$.
In our example, if $\sigma(0)=1, \sigma(1)=2$, and $\sigma(2)=0$, one can see that $q=p+ \langle \mu_1, \mu_2, \mu_1 \circ \mu_0 \circ \mu_1^{-1} \rangle$.  

\begin{lemma}

$\mathbb{P}$ is $\lambda^{++}$-c.c.

\end{lemma}

\begin{proof}

Let $\{ p^\alpha: \alpha<\lambda^{++} \}$ be a collection of conditions in $\mathbb{P}$.
We first show that we may assume without loss of generality that $p_\alpha$ is pure for all $\alpha$.
Since $2^{\aleph_0}<\lambda^{++}$, we may shrink the set and assume that all $p_\alpha$s' have the same supports, say $a$.
Suppose $a \neq \emptyset$, let $n=\max(a)$.
Hence, for $\alpha<\lambda^{++}$, $p^\alpha \restriction n \in V_{\kappa_n}$.
Shrink the set so that $p^\alpha \restriction n$are all the same for all $\alpha$.
As a consequence, we can assume $p^\alpha$ is pure for all $\alpha$.
For each $n<\omega$, $d_n^{p^\alpha} \in [\lambda_n^j]^{\lambda}$.
By GCH, we shrink the set so that $ \langle d_n^{p^\alpha}:\alpha<\lambda^{++} \rangle$ forms a $\Delta$-system.
Let $r$ be the root.
Since $|r|=\lambda$, we shrink the collection so that $f_n^{p^\alpha} \restriction r$ are all the same.
Thus, $f_n^{p^\alpha}$ and $f_n^{p^\beta}$ are all compatible for $n<\omega$ and $\alpha,\beta<\lambda^{++}$.
Measure-one sets are compatible, hence, there is a pair of compatible conditions.

\end{proof}

If $p \in \mathbb{P}$ and $n \in \supp(p)$, then $\mathbb{P} / p$ factors into two posets $\mathbb{P}_0=((\mathbb{P}/p) \restriction n)_{\langle e_{m,n}(f_n^p)):m<n \rangle}$ and $\mathbb{P}_1=(\mathbb{P}/p) \setminus n$.

\begin{lemma}

$\mathbb{P}_0$ is $\lambda_n(f_n^p)^{++}$-c.c.
If $ n+1 \not \in \supp(p)$,  $(\mathbb{P}_1,\leq^*)$ is $\kappa_{n+1}$-closed.
If $n+1 \in \supp(p)$, $(\mathbb{P}_1,\leq^*)$ is $\lambda_{n+1}(f_{n+1}^p)^+$-closed.

\end{lemma}

\begin{proof}

If $n+1 \not \in \supp(p)$, then the part that makes $\mathbb{P}_1$ has the lowest closure is the $n+1$th measure-one set: the $\leq^*$ relation on $n$-th coordinate is $\theta$-closed for some $\theta>\kappa_{n+1}$.
The correspond ultrafilter is $\kappa_{n+1}$-complete, so $(\mathbb{P}_1,\leq^*)$ is $\kappa_{n+1}$-closed.

\end{proof}

\begin{lemma}
\label{integration2}
Let $p \in \mathbb{P}$ and $n \not \in\supp(p)$.
Let $f^\prime \leq f_n^p$ with $d^\prime=\dom(f^\prime)$. 
Assume that if $n^*=\min(\supp(p) \setminus (n+1))$ exists, $A^\prime \in e_{n,n^*}(f_{n^*})(d^\prime)$, and if $n^*$ does not exist, $A^\prime \in E_n(d^\prime)$.
Fix $f_m$ for $m>n$ such that $f_m \leq f_m^p$ and $d_m:=\dom(f_m) \supseteq d^\prime$, and if $n<m<m^\prime$, then $\dom(f_m) \subseteq \dom(f_{m^\prime})$.
Suppose that for each $\mu \in A^\prime$, there is a condition $t(\mu) \leq^* (p+ (\mu \restriction d_n^p)) \restriction n$, and there is $\vec{r}(\mu)=\langle f_m,A_m(\mu): m>n \rangle$ such that $A_m(\mu)$ is of measure-one with respect to $d_m$. 
Then there is a condition $q \leq^* p$ such that 

\begin{enumerate}

\item $f_n^q \leq f^\prime$ and $A_n^q$ projects down to $A^\prime$.

\item for $\tau \in A_n^q$ with $\mu= \tau \restriction d^\prime$, 

\begin{itemize}

\item $(q+\tau) \restriction n = t(\mu)$.

\item $(q+\tau) \setminus (n+1) \leq^* \vec{r}(\mu)$.

\end{itemize}

\end{enumerate}

\end{lemma}

\begin{proof}

The proof is modified from Lemma \ref{integration1}.
We will only point out key differences.
Assume $p$ is pure for simplicity.
For $m>n$, let $A_m=\triangle_{\mu \in A^\prime} A_m(\mu)$, so $A_m \in E_m(\dom(f_m))$.
One can modify the proof of Lemma \ref{integration1} to construct $q \restriction (n+1)$.
For $m>n$, define $q_m= \langle g_m, B_m \rangle$ such that $\dom(g_m)=\dom(f_m) \cup \dom(f_n^q)$, $g_m(\gamma)=f_m(\gamma)$ if $\gamma \in \dom(f_m)$, otherwise, $g_m(\gamma)=0$, $B_m \in E_m(\dom(g_m))$ and $B_m$ projects down to a subset of $A_m$ and $A_m^p$.
Fix $\tau \in A_n^q$ and $\mu=\tau \restriction d^\prime$ and $m>n$.
Then $(q+\tau)_m=\langle g_m, C_m \rangle$ where $C_m \subseteq \{\chi \in A_m^q: \dom(\tau) \cup \rge(\tau) \subseteq \dom(\chi)\}$.
Clearly $g_m \leq f_m$.
For $\chi \in C_m$, $\dom(\mu)=\dom(\tau \restriction d^\prime) \subseteq \dom(\chi \restriction d^\prime)$.
Since $\rge(\mu)=\rge(\tau \restriction d^\prime)$, it is a subset of $\dom(\chi)$.
Since $d^\prime \supseteq \lambda$, $\rge(\mu) \subseteq \dom(\chi \restriction d^\prime)$.
Since $d^\prime \subseteq d_m$ and $\chi \restriction d_m \in A_m$ we have that $\chi \restriction d_m \in A_m(\mu)$.
Hence, $(q+\tau)_m \leq ^* \langle f_m,A_m(\mu) \rangle$. 

\end{proof}

\begin{thm}
\label{prikryomega}
$\mathbb{P}$ has the Prikry property.

\end{thm}

\begin{proof}

Let $b$ be a Boolean value.
We will  start by proving the Prikry property for a pure condition $p \in \mathbb{P}$.
Later, we will explain how to modify the proof for impure conditions.
Our goal is to build a $\leq^*$-decreasing sequence $\langle p^n: n<\omega \rangle$ such that $p^0 \leq^*p$ and if $q \leq p$, $q$ decides $b$ and $\max(\supp(q))=n$, then $p^n$ decides $b$.
Then we claim at the end of the proof that by letting $p^*$  to be a $\leq^*$-lower bound of $\langle p^n: n<\omega \rangle$, then $p^*$ decides $b$.
We give a convention for the proof.
Let $p$ be a pure condition.
If there is $q \leq^*p$ such that $q$ decides $b$, then we are done.
Suppose not.
Let $\mathbb{B}_n=\mathbb{B}^{E_n}(\lambda,\lambda_n^j)$.
If $f \in \mathbb{B}_n$ and $\mu$ is an $n$-object with respect to $\dom(f)$, then $f \oplus \mu \in \mathbb{B}_n$.

Construction of $p^0$:
Let

\begin{center}

$\mathbb{Q}_0 = \{ \langle f_n:n < \omega \rangle \in \prod_{n<\omega} \mathbb{B}_n \mid \langle \dom(f_n): n< \omega \rangle$ is $\subseteq$-increasing $\}$

\end{center}

where $\prod_{n<\omega} \mathbb{B}_n$ is a full support product, and

\begin{center}

$\mathbb{R}_0 =\{ (f,\vec{r}) \in \mathbb{B}_0 \times (\mathbb{P}_{\langle E_n: n>0 \rangle},\leq^*): \dom(f) \subseteq \dom(f_1^r)\}$,

\end{center}

where the first coordinate of $\vec{r}$ is $\langle f_1^r,A_1^r \rangle$.
By adapting the proof in Theorem \ref{prikry1}, we find 

\begin{enumerate}

\item $N_0 \prec H_\theta$ for some sufficiently large regular cardinal $\theta>\bar{\lambda}_\omega^j$ such that $N_0$ is internally approachable in the way described as in Theorem \ref{prikry1}, witnessed by a sequence of elementary submodels of length $\kappa_0$, $|N_0|=\lambda$ and ${}^{<\kappa_0} N_0 \subseteq N_0$.
In addition, $\lambda ,b,p, \mathbb{P} \in N_0$, $\lambda,d_0^p \subseteq N_0$, and for all $n<\omega$, $\lambda_n^j \in N_0$

\item $\langle f_n^0: n < \omega \rangle$ which is $(N_0,\mathbb{Q}_0)$-generic, and $\langle f_n^0:n< \omega \rangle \leq_{\mathbb{Q}_0} \langle f_n^p: n<\omega \rangle$.

\item $d_0^0=N_0 \cap \lambda_0^j$.

\end{enumerate}

The construction is similar to the explanation at the beginning of Theorem \ref{prikry1}. 
Let $A_n^0 \in E_n(d_n^0)$ be a measure-one set projecting down to $A_n^p$ for all $n$.
For $\mu \in A_0^0$, $|\mu|<\kappa_0$, $\dom(\mu) \subseteq d_n^0$ and $\rge(\mu) \subseteq \lambda$, so $\mu \in N_0$.
For $\mu \in A_0^0$, define $D_\mu$ as the collection of $(f,\vec{r}) \in \mathbb{R}_0$ such that $\dom(\mu) \subseteq \dom(f)$ and if there are $g \leq f \oplus \mu$ and $\vec{r}^\prime \leq^* \vec{r}$ such that $\langle g \rangle^\frown \vec{r}^\prime$ decides $b$, then $\langle f \oplus \mu \rangle^\frown \vec{r}$ decides $b$ the same way.
Let $\pi: \mathbb{R}_0 \to \mathbb{Q}_0$ be the natural projection and $D_\mu^\prime = \pi[D_\mu]$.

\begin{claim}
\label{denseopen3}
$D_\mu$ is open dense below $\langle f_0^p \rangle^\frown \langle p_n : n>0 \rangle$ and $D_\mu \in N_0$. As a consequence, $D_\mu^\prime$ is an open dense set and $D_\mu^\prime \in N_0$.

\end{claim}

\begin{claimproof}{(Claim \ref{denseopen3})}

It is easy to check that $D_\mu$ is open.
$D_\mu$ is defined using $\mathbb{R}_0$, $\mu$, and $b$, so $D_\mu \in N_0$.
It remains to show that $D_\mu$ is dense below $\langle f_0^p \rangle^\frown \langle p_n : n>0 \rangle$.
Let $\langle f \rangle^\frown \vec{r} \leq_{\mathbb{R}_0} \langle f_0^p \rangle ^\frown \langle p_n:n>0 \rangle$.
By the density, assume $\dom(\mu) \subseteq \dom(f)$.
If there are $g \leq f \oplus \mu$ and $\vec{r}^\prime \leq^* \vec{r}$ such that $\langle g \rangle^\frown \vec{r}^\prime$ decides $b$, let $g^\prime$ be a function with $\dom(g^\prime)=\dom(g)$ and $g^\prime(\gamma)=f(\gamma)$ for $\gamma \in \dom(f)$, otherwise $g^\prime(\gamma)=g(\gamma)$, then $\langle g^\prime \rangle {}^\frown \vec{r}^\prime \in D_\mu$ and $\langle g^\prime \rangle {}^\frown \vec{r}^\prime \leq_{\mathbb{R}_0} \langle f \rangle ^\frown \vec{r}$.

\end{claimproof}{(Claim \ref{denseopen3})}

By genericity, we have that $\langle f_n^0: n< \omega \rangle \in D_\mu^\prime$.
Find $A_n^0(\mu) \in E_n(d_n^0)$ for $n>0$ such that $\langle f_0^0 \rangle^\frown \langle \langle f_n^0,A_n^0(\mu) \rangle: n>0 \rangle \in D_\mu$.
For $n>0$, define $(A_n^0)^* = A_n^0 \cap \triangle_{\mu \in A_0^0} A_n^0(\mu)$.
By Lemma \ref{diagintersect}, $(A_n^0)^* \in E_n(d_n^0)$.
Let $q^0= \langle f_0^0,A_0^0 \rangle^\frown \langle \langle f_n^0,(A_n^0)^* \rangle: n>0 \rangle$.
For $\mu \in A_0^0$, $n>0$, we have that $A_n^{q^0+\mu}\subseteq \{ \tau \in (A_n^0)^*: \dom(\mu) \cup \rge(\mu) \subseteq \dom(\tau)\} \subseteq A_n^0(\mu)$.
Thus, we have the following property for $q^0$ (call it $(\star(q^0))$: 

\begin{center}

$\boldsymbol{(\star(q^0))}$ For $\mu \in A_0^0$, if there are $f \leq (q^0+\mu)_0$ (which is $f_0^0 \oplus \mu$), and $\vec{r} \leq^* (q^0+\mu) \setminus 1$ such that $\langle f \rangle^\frown \vec{r}$ decides $b$, then $q^0+\mu$ decides $b$ the same way.

\end{center}

Let 

\begin{align*}
A_0 &= \{\mu \in A_0^0: \exists f \leq f_0^0 \oplus\mu (\exists \vec{r} \leq^* (q^0+\mu) \setminus 1(\langle f \rangle^\frown \vec{r} \Vdash b))\} \\
A_1 &= \{ \mu \in A_0^0 \setminus A_0: \exists f \leq f_0^0 \oplus \mu (\exists \vec{r} \leq^* (q^0+\mu) \setminus 1(\langle f \rangle^\frown \vec{r} \Vdash \neg b))\} \\
A_2&= A_0^0 \setminus (A_0 \cup A_1).
\end{align*}

Note that $\mu \in A_0$ implies that $q^0+\mu \Vdash b$, and $\mu \in A_1$ implies that $q^0+\mu \Vdash \neg b$.
There exists unique $i=\{0,1,2\}$ such that $A_i$ is of measure-one, call it $i_0$.
Let 
\begin{center}

$p^0 = \langle f_0^0,A_{i_0} \rangle ^\frown \langle \langle f_n^0,(A_n^0)^* \rangle: n>0 \rangle$.

\end{center}

We now show that if there is $q \leq p^0$ with $\supp(q)= \{0\}$ and $q$ decides $b$, then $p^0$ decides $b$.
Let $q$ be such a witness, and without loss of generality, suppose that $q \Vdash b$ and $q \leq^* p^0+\mu$ for some $\mu \in A_0^{p^0}$.
Then $f_0^q \leq f_0^0 \oplus \mu=(q^0+\mu)_0$ and  $q \setminus 1 \leq^* (p^0+\mu) \setminus 1=(q^0+\mu) \setminus 1$.
By $\star(q^0)$, we have that $q^0+\mu \Vdash b$ and $\mu \in A_0$.
This means $i_0=0$.
For every extension $r$ of $p^0$ with $0 \in \supp(r)$, we have $r \leq p+ \mu^\prime$ for some $\mu^\prime \in A_0$.
With the property $\star(q^0)$, we have that $q^0+\mu^\prime \Vdash b$, and hence, $r \Vdash b$.
By the density, $p^0 \Vdash b$.

Construction of $p^{m+1}$:

Let

\begin{center}

$\mathbb{Q}_{m+1} = \{ \langle f_n:n >m \rangle \in \prod_{n>m} \mathbb{B}_n \mid \langle \dom(f_n): n>m \rangle$ is $\subseteq$-increasing$\}$

\end{center}

where $\prod_{n>m} \mathbb{B}_n$ is a full support product, and

\begin{center}

$\mathbb{R}_{m+1} =\{ (f,\vec{r}) \in \mathbb{B}_{m+1} \times (\mathbb{P}_{\langle E_n: n>m+1 \rangle},\leq^*): \dom(f) \subseteq \dom(f_{m+2}^r)\}$,

\end{center}

where the first coordinate of $\vec{r}$ is $\langle f_{m+2}^r,A_{m+2}^r \rangle$.
By adapting the proof in Theorem \ref{prikry1}, we find 

\begin{enumerate}

\item $N_{m+1} \prec H_\theta$ for some sufficiently large regular cardinal $\theta>\bar{\lambda}_\omega^j$ such that $N_0$ is internally approachable in a way described as in Theorem \ref{prikry1}, witnessed by a sequence of elementary submodels of length $\kappa_{m+1}$, $|N_{m+1}|=\lambda$ and ${}^{<\kappa_{m+1}} N_{m+1} \subseteq N_{m+1}$.
Also, $\lambda ,b,p^m, \mathbb{P} \in N_{m+1}$, $\lambda,f_{m+1}^{p^m} \subseteq N_{m+1}$, and for all $n>m$, $\lambda_n^j \in N_{m+1}$

\item $\langle f_n^{m+1}: n >m \rangle$ which is $(N_{m+1},\mathbb{Q}_{m+1})$-generic, and $\langle f_n^{m+1}:n> m \rangle \leq_{\mathbb{Q}_{m+1}} \langle f_n^{p^m}: n>m \rangle$.

\item $d_n^{m+1}=\dom(f_n^{m+1})=N_{m+1} \cap \lambda_n^j$ for all $n>m$. 
In particular, $d_{m+1}^{m+1} \subseteq N_{m+1}$.

\end{enumerate}

The construction is similar to the explanation at the beginning of Theorem \ref{prikry1}. 
Let $A_n^{m+1} \in E_n(d_n^{m+1})$ be a measure-one set projecting down to $A_n^{p^m}$ for all $n>m$.
Since for $\mu \in A_{m+1}^{m+1}$, $|\mu|<\kappa_{m+1}$, $\dom(\mu) \subseteq d_{m+1}^{m+1}$ and $\rge(\mu) \subseteq \lambda$, so $\mu \in N_{m+1}$.
Fix $\mu \in A_{m+1}^{m+1}$.
Let $\mathbb{P}_{m+1,\mu}= \mathbb{P}_{\langle e_{n,m+1}(\mu):n \leq m \rangle}$.
Note that $|\mathbb{P}_{m+1,\mu}|< \kappa_{m+1}$.
Let $\langle t_\alpha: \alpha< |\mathbb{P}_{m+1,\mu}| \rangle$ be an enumeration in $N_{m+1}$ of conditions in $\mathbb{P}_{m+1,\mu}$.
 Define $D_\mu$ as the collection of $(f,\vec{r}) \in \mathbb{R}_{m+1}$ such that $\dom(\mu) \subseteq \dom(f)$ and for each $\alpha<|\mathbb{P}_{m+1,\mu}|$, if there are a function $g \leq f \oplus \mu$ and $\vec{r}^\prime \leq^* \vec{r}$ such that $t_\alpha^\frown \langle g \rangle^\frown \vec{r}^\prime$ decides $b$, then $t_\alpha^\frown \langle f \oplus \mu \rangle^\frown \vec{r}$ decides $b$ the same way.
Let $\pi: \mathbb{R}_{m+1} \to \mathbb{Q}_{m+1}$ be the natural projection and $D_\mu^\prime = \pi[D_\mu]$.

\begin{claim}
\label{denseopen4}
$D_\mu$ is open dense below $\langle f_{m+1}^{p^m} \rangle^\frown \langle p^m_n : n>m+1 \rangle$ and $D_\mu \in N_{m+1}$.
As a consequence, $D_\mu^\prime$ is an open dense set and $D_\mu^\prime \in N_{m+1}$.

\end{claim}

\begin{claimproof}{Claim (\ref{denseopen4})}

It is easy to check that $D_\mu$ is open.
$D_\mu$ is defined using $\mathbb{R}_{m+1}$, $\mu$, and $b$, so $D_\mu \in N_{m+1}$.
It remains to show that $D_\mu$ is dense below $\langle f_{m+1}^{p^m} \rangle^\frown \langle p_n : n>{m+1} \rangle$.
Let $\langle f \rangle^\frown \vec{r} \leq_{\mathbb{R}_{m+1}} \langle f_{m+1}^{p^m} \rangle ^\frown \langle p_n:n>m+1 \rangle$.
By density, assume $\dom(\mu) \subseteq \dom(f)$.
Build $\langle (f_{m+1,\alpha},\vec{r}_{m+1,\alpha}): \alpha \leq | \mathbb{P}_{m+1,\mu}| \rangle$ as follows:

\begin{enumerate}

\item $f_{m+1,0}=f$ and $\vec{r}_{m+1,0}=\vec{r}$.

\item If $\alpha$ is limit, let $(f_{m+1,\alpha}$  $\vec{r}_{m+1,\alpha})$ be a $\mathbb{R}_{m+1}$-lower bound of $\langle ( f_{m+1,\beta},\vec{r}_{m+1,\beta} ): \beta<\alpha \rangle$. 

\item For $\alpha< |\mathbb{P}_{m+1,\mu}|$, ask whether there are $g \leq f_{m+1,\alpha} \oplus \mu$ and $\vec{r}^\prime \leq^* \vec{r}_{m+1,\alpha}$ such that $t_\alpha^\frown \langle g \rangle^\frown \vec{r}^\prime$ decides $b$.
If the answer is positive, let $f_{m+1,\alpha}$ be a function with $\dom(f_{m+1,\alpha})=\dom(g)$, $f_{m+1,\alpha}(\gamma)=f_{m+1,\alpha}(\gamma)$ if $\gamma \in \dom(f_{m+1,\alpha})$, otherwise $f_{m+1,\alpha+1}(\gamma)=g(\gamma)$, and let $\vec{r}_{m+1,\alpha+1}=\vec{r}^\prime$.
If the answer is negative, let $f_{m+1,\alpha+1}=f_{m+1,\alpha}$ and $\vec{r}_{m+1,\alpha+1}=\vec{r}_{m+1,\alpha}$.

\end{enumerate}

The argument at the successor stages is similar to the argument in Claim \ref{denseopen1}.
 The construction proceeds to the end since $\mathbb{R}_{m+1}$ is $\kappa_{m+1}$-closed.
Let $g^\prime = f_{m+1,|\mathbb{P}_{m+1,\mu}|}$ and $\vec{r}^\prime =\vec{r}_{m+1,|\mathbb{P}_{m+1,\mu}|}$.
Then $\langle g^\prime \rangle{}^\frown \vec{r}^\prime \in D_\mu$ and $\langle g^\prime \rangle^\frown \vec{r} \leq_{\mathbb{R}_{m+1}} \langle g\rangle^\frown \vec{r}$.

\end{claimproof}{Claim (\ref{denseopen4})}

Thus, by genericity, we have that $\langle f_n^{m+1}: n>m \rangle \in D_\mu^\prime$.
Find $A_n^{m+1}(\mu) \in E_n(d_n^{m+1})$ for $n>m+1$ such that $\langle f_{m+1}^{m+1} \rangle^\frown \langle \langle f_n^{m+1},A_n^{m+1}(\mu) \rangle: n>m+1 \rangle \in D_\mu$.
For $n>{m+1}$, let $(A_n^{m+1})^* = A_n^{m+1} \cap \triangle_{\mu \in A_{m+1}^{m+1}} A_n^{m+1}(\mu)$.
By Lemma \ref{diagintersect}, $(A_n^{m+1})^* \in E_n(d_n^{m+1})$.
Let $q^{m+1}= p^m \restriction (m+1)^\frown \langle f_{m+1}^{m+1},A_{m+1}^{m+1} \rangle^\frown \langle \langle f_n^{m+1},(A_n^{m+1})^* \rangle: n>m+1 \rangle$.
For $\mu \in A_{m+1}^{m+1}$, $n>m+1$, we have that $A_n^{q^{m+1}+\mu}\subseteq \{ \tau \in (A_n^{m+1})^*: \dom(\mu) \cup \rge(\mu) \subseteq \dom(\tau)\} \subseteq A_n^{m+1}(\mu)$.
Thus, we have the following property for $q^{m+1}$ (call it $(\star(q^{m+1}))$: 

\begin{center}

$\boldsymbol{(\star(q^{m+1}))}$: For $\mu \in A_{m+1}^{m+1}$, for $t \in \mathbb{P}_{m+1,\mu}$, if there are $f \leq (q^{m+1}+\mu)_{m+1}$ (which is $f_{m+1}^{m+1} \oplus \mu$), and $\vec{r} \leq^* (q^{m+1}+\mu) \setminus (m+2))$ such that $t^\frown \langle f \rangle^\frown \vec{r}$ decides $b$, then $t^\frown (q^{m+1}+\mu) \setminus (m+1)$ decides $b$ the same way.

\end{center}
Fix $\mu \in A_{m+1}^{m+1}$.
Let $\dot{G}$ be the canonical name for the generic object for $\mathbb{P}_{m+1,\mu}$.
By the induction hypothesis, find $t(\mu) \leq^* p^m \restriction (m+1)$ which decides the following sentence:

\begin{center}

$\varphi \equiv \exists t \in \dot{G} (t^\frown (q^m + \mu) \setminus (m+1) \parallel b)$.

\end{center}

Let

\begin{align*}
A_0 &= \{\mu \in A_{m+1}^{m+1}: t(\mu) \Vdash \exists t \in \dot{G} (t^\frown (q^{m+1} + \mu) \setminus (m+1) \Vdash b)\} \\
A_1 &= \{ \mu \in A_{m+1}^{m+1} \setminus A_0: t(\mu) \Vdash \exists t \in \dot{G} (t^\frown (q^{m+1} + \mu) \setminus (m+1) \Vdash \neg b)\} \\
A_2&= A_{m+1}^{m+1} \setminus (A_0 \cup A_1).
\end{align*}

There exists unique $i \in \{0,1,2\}$ such that $A_i$ is of measure-one.
Let $i_{m+1}$ be as such.
We now apply Lemma \ref{integration2} to the following setting: the $``p"$ in Lemma \ref{integration2} is $p^m$, $n=m+1$, $f^\prime=f_{m+1}^{m+1}$, $A^\prime = A_{i_{m+1}}$, $t(\mu)$ and $\vec{r}(\mu)=\langle f_n^{m+1},A_n^{m+1}(\mu): n>m+1 \rangle$ for each $\mu \in A_{i_{m+1}}$ as described earlier, to find $p^{m+1} \leq ^* p^m$ such that

\begin{enumerate}

\item $f_{m+1}^{p^{m+1}} \leq f_{m+1}^{m+1}$ and $A_{m+1}^{p^{m+1}}$ projects down to $A_{i_{m+1}}$.

\item for $\tau \in A_{m+1}^{p^{n+1}}$ with $\mu=\tau \restriction d_{m+1}^{m+1}$, 

\begin{itemize}

\item $(p^{m+1}+ \tau) \restriction (m+1) =t(\mu)$.

\item $(p^{m+1}+\tau) \setminus (m+2) \leq^* \vec{r}(\mu)$.

\end{itemize}

\end{enumerate}

We now show that if there is $q \leq p^{m+1}$ with $\max(\supp(q))=m+1$ and $q$ decides $b$, then $p^{m+1}$ decides $b$ the same way $q$ does.
Let $q \leq p^{m+1}$ be a condition deciding $b$ and $\max(\supp(q))=m+1$. Without loss of generality, assume $q \Vdash b$.
Suppose $q \leq p^{m+1}+ \tau$ for some $\tau \in A_{m+1}^{p^{m+1}}$.
Let $\mu=\tau \restriction d_{m+1}^{m+1}$.
By the property $\star(q^{m+1})$, we have that $q \restriction  (m+1)^\frown (q^{m+1}+\mu) \setminus (m+1) \Vdash b$. 
Note that $q \restriction (m+1) \leq (p^{m+1}+\tau) \restriction (m+1)=t(\mu)$.
We claim that $i_{m+1}=0$.
Suppose $i_{m+1}=1$ (the case $i_{m+1}=2$ is similar).
Let $G$ be generic containing $q \restriction (m+1)$, hence, containing $t(\mu)$.
Let $t \in G$ be such that $t^\frown (q^{m+1} +\mu) \restriction (m+1) \Vdash \neg b$.
We may assume $t \leq q \restriction (m+1)$, but then $t^\frown (q^{m+1}+\mu) \setminus (m+1) \Vdash b,\neg b$, which is a contradiction.
Thus, $i_{m+1}=0$.
A similar proof shows that already $t(\mu)^\frown (q^{m+1}+\mu) \setminus (m+1) \Vdash b$, and hence, $p^{m+1}+\tau \Vdash b$.
Since $i_{m+1}=0$, for every $\tau^\prime \in A_{m+1}^{p^{m+1}}$ with $\mu^\prime= \tau^\prime \restriction d_{m+1}^{m+1}$, $\mu^\prime \in A_0$.
A similar argument shows that $p^{m+1}+\tau^\prime \Vdash b$.
By a density argument, $p^{m+1} \Vdash b$.
We finish the argument for $p^{m+1}$.
Recall that $p^*$ is a $\leq^*$-lower bound of $\langle p^n: n<\omega \rangle$.

\begin{claim}
\label{finalcond2}
$p^* \leq^*p$ and $p^*$ satisfies the Prikry property.

\end{claim}

\begin{claimproof}{(Claim \ref{finalcond2})}

To show that there is a condition $q \leq^* p^*$ which decides $b$, let $q \leq p^*$ such that $q$ decides $b$.
If $q \leq^* p^*$, then $q \leq^* p$, but this contradicts our assumption that there is no direction extension of $p$ deciding $b$.
Hence, $q$ is impure, let $n=\max(\supp(p))$.
Since $q \leq p^n$, by our construction of $p^n$, we have that $p^n$ decides $b$ the same way $q$ does.
Since $p^* \leq^* p^n$, $p^*$ also decides $b$, which is a contradiction.
Hence, we are done.

\end{claimproof}{(Claim \ref{finalcond2})}

We finished the proof of the Prikry property for pure conditions.
We now give an outline of the proof when we start with an impure condition $p$ with a focus on the modifications.
Assume that $p$ is impure. 
We will give an outline on how to  construct a $\leq^*$-decreasing sequence $\langle p^m: m< \omega \rangle$ satisfying the properties described as in the  proof starting with a pure condition.
The construction will split into three cases: $p^m$ for $m \not \in \supp(p)$ such that $\supp(p) \setminus (m+1) \neq \emptyset$, $p^m$ for $m \in \supp(p)$ and $p^m$ for $m>\max(\supp(p))$.

$p^m$ for $m \not \in \supp(p)$ which $\supp(p) \setminus (m+1)$ is nonempty: the construction is similar as in the pure case, except that the parameters change.
Let $m^*=\min(\supp(p) \setminus (m+1))$, $\lambda_{m^*}=\lambda_{m^*}(f_{m^*}^p)$.
For $n \in [m,m^*)$, let $\lambda_{n,m^*}^j=\lambda_{n,m^*}^j(f_{m^*}^p)$,  and $\mathbb{B}_{n,m^*}=B_{n,m^*}^{e_{n,m^*}(f_{m^*}^p)}(\lambda_{m^*},\lambda_{n,m^*}^j)$.

Let

\begin{align*}
\bar{\mathbb{Q}}_{m}= & \{ \langle f_n:n  \geq m \rangle \in \prod_{n \in [m,m^*)} \mathbb{B}_{n,m^*} \mid  \\
& \langle \dom(f_n): n \in [m,m^*) \rangle \text{ is } \subseteq\text{-increasing}\}, \\ 
\bar{\mathbb{R}}_{m} = & \{ (f,\vec{r}) \in \mathbb{B}_{m,m^*} \times (\mathbb{P}_{\langle e_{n,m^*}(f_{m^*}^p): n \in [m+1,m^*) \rangle},\leq^*) \mid \\
& \dom(f) \subseteq \dom(f_{m+1}^r)\}, \\
\bar{\mathbb{S}}_m = &((\mathbb{P}/p) \setminus m^*,\leq^*)
\end{align*}

where the first coordinate of $\vec{r}$ appearing in $\bar{\mathbb{R}}_m$ has the first coordinate $f_{m+1}^r$
Recall that $\kappa_{m^*}(f_{m^*}^p)$ ,$\lambda_{m^*}$ are regular, and greater than $\bar{\kappa}_{m^*}$.
Note that $\bar{\mathbb{Q}}_m$ is $\lambda_{m^*}^+$-closed
Let $\bar{N}_m \prec H_\theta$ for some sufficiently large regular cardinal $\theta>\bar{\lambda}_\omega^j$ containing ``enough" information, such that $|\bar{N}_m|=\lambda_{m^*}$ and ${}^{<\kappa_m} N_m \subseteq N_m$.
Let $\langle f_n^m:  n \in [m,m^*) \rangle^\frown \vec{r}$ be $(N_{m+1},\bar{\mathbb{Q}}_m \times \bar{\mathbb{S}}_m)$-generic below $\langle f^{p^{m-1}}_n: n \in [m,m^*) \rangle^\frown p^{m-1} \setminus m^*$, where $p^{-1}=p$.
Let $d_n^m=\dom(f_n^m)$.
Let $A_n^m \in e_{n,m^*}(f_{m^*}^p)(d_n^m)$ be a measure-one set projecting down to $A_n^{p^{m-1}}$ for all $n \in [m,m^*)$.
Fix $\mu \in A_m^m$, so $\mu \in \bar{N}_m$.
Let $\bar{\mathbb{P}}_{m,\mu}= \mathbb{P}_{\langle e_{n,m}(\mu):n < m \rangle}$.
 Define $\bar{D}_\mu$ as the collection of $(f,\vec{r}_0,\vec{r}_1) \in \bar{\mathbb{R}}_m \times \bar{\mathbb{S}}_m$ such that $\dom(\mu) \subseteq \dom(f)$ and for each $t \in \mathbb{P}_{m+1,\mu}$, if there are a function $g \leq f \oplus \mu$ and $\vec{r}_0^\prime \leq^* \vec{r}_0$, $\vec{r}_1^\prime \leq^* \vec{r}_1$ such that $t^\frown \langle g \rangle^\frown \vec{r}_0^\prime {}^\frown \vec{r}_1^\prime$ decides $b$, then $t^\frown \langle f \oplus \mu \rangle^\frown \vec{r}_0 {}^\frown \vec{r}_1$ decides $b$ the same way.
Let $\pi: \bar{\mathbb{R}}_m \times \bar{\mathbb{S}}_m \to \bar{\mathbb{Q}}_m \times \bar{\mathbb{S}}_m$ be the natural projection and $\bar{D}_\mu^\prime = \pi[\bar{D}_\mu]$.
One can check that $\bar{D}_\mu$ is an open dense set in $\bar{N}_m$. 

Thus, we have that $\langle f_n^m:  n \in [m,m^*) \rangle^\frown \vec{r} \in \bar{D}_\mu^\prime$.
Find $A_n^m(\mu) \in e_{n,m^*}(f_{m^*}^p)(d_n^m)$ for $n \in [m+1,m^*)$ such that $\langle f_m^m \rangle^\frown \langle \langle f_n^m,A_n^m(\mu) \rangle: n \in [m+1,\mu^*) \rangle^\frown \vec{r} \in D_\mu$.
For $n \in [m+1,m^*)$, let $(A_n^m)^* = A_n^m \cap \triangle_{\mu \in A_m^m} A_n^m(\mu)$.
By Lemma \ref{diagintersect}, $(A_n^{m+1})^* \in e_{n,m^*}(f_{m^*}^p)(d_n^m)$.
Let $q^m= p^{m-1} \restriction m^\frown \langle f_m^m,A_m^m \rangle^\frown \langle \langle f_n^m,(A_n^m)^* \rangle: n>m+1 \rangle^\frown \vec{r}$ (if $m=0$ the first term $p^{m-1}$ does not exist).
Thus, we have the maximizing property for $q^m$ (same as in $\star(q^m)$ in the pure case).
If $m=0$, shrink the set $A_m^m$ so that every object in $A_m^m$ behaves the same, and then we can form $p^m$ by just $q^m$ with the first shrunk measure-one set.
If $m>0$, for $\mu \in A_m^m$, find $t(\mu) \in \bar{\mathbb{P}}_{m,\mu}$ which is a direct extension of $(q^m+\mu) \restriction m$ deciding certain statement.
Shrink the measure-one set $A_m^m$ so that for $\mu \in A_m^m$, $t(\mu)$ decides the statement in the same direction.
Use Lemma \ref{integration2} to form $p^m$.

$p^m$ for $m \in \supp(p)$:
Assume $m>0$ (the case $m=0$ is simpler). 
Let $\bar{\mathbb{P}}_m=\mathbb{P}_{\langle e_{n,m}(f_m^p):n<m \rangle}$.
Let $\theta=|\bar{\mathbb{P}}_m|$, then $\theta<\kappa_m$. 
Enumerate the conditions in $\bar{\mathbb{P}}_m$ as $\langle t_\alpha:\alpha<\theta \rangle$.
Build a $\leq^*$-decreasing sequence $\langle \vec{r}_\alpha: \alpha \leq \theta \rangle$ in $(\mathbb{P}/p) \setminus m$ such that 

\begin{enumerate}

\item $\vec{r}_0=p^{m-1} \setminus m$.

\item if $\alpha \leq \theta$ is a limit ordinal, let $\vec{r}_\alpha$ be a $\leq^*$-lower bound of $\langle \vec{r}_\beta:\beta<\alpha\rangle$.

\item for $\alpha<\theta$, if $\vec{r}_\alpha$ is built, and there is $\vec{r} \leq^* \vec{r}_\alpha$ such that $t_\alpha^\frown \vec{r}$ decides $b$, let $\vec{r}_{\alpha+1}$ be such a $\vec{r}$, otherwise $\vec{r}_{\alpha+1}=\vec{r}_\alpha$.

\end{enumerate}

The construction proceeds to the $\theta$-th stage, since $((\mathbb{P}/p) \setminus m,\leq^*)$ is $\kappa_m^+$-closed.
Let $\vec{r}^*=\vec{r}_\theta$.
Let $\dot{G}$ be the canonical name of a generic object for $\bar{\mathbb{P}}_m$.
Let $t \leq^* p^{m-1} \restriction m$ be deciding the following statement:

\begin{center}

$\varphi \equiv \exists t^\prime \in \dot{G}(t^\prime {}^\frown \vec{r} \parallel b)$.

\end{center}

By extending $t$ regarding the direct extension if necessary, assume that either $t \Vdash \exists t^\prime \in \dot{G}(t^\prime {}^\frown \vec{r} \Vdash b)$, $t \Vdash \exists t^\prime \in \dot{G}(t^\prime {}^\frown \vec{r} \Vdash \neg b)$, or $t \Vdash \nexists t^\prime \in \dot{G}(t^\prime {}^\frown \vec{r} \Vdash b)$. 
We finish the construction by letting $p^m=t^\frown \vec{r}$.

$p^m$ for $m> \max(\supp(p))$: the construction is exactly the same as for the construction of $p^m$ for the pure case.

We finish the proof of Theorem \ref{prikryomega}.

\end{proof}

\begin{thm}
\label{strongprikry1}
$\mathbb{P}$ has the strong Prikry property. Namely, for each dense open set $D \subseteq \mathbb{P}$ and $p \in \mathbb{P}$, there is a condition $q \leq^*p$ and a finite subset $a$ of $\omega$ ($a$ can be empty) such that

\begin{enumerate}

\item $a \cap \supp(p) =\emptyset$.

\item every $|a|$-step extension of $q$ using objects from $\{A_n^q:n \in a\}$ lies in $D$.

\end{enumerate}

\end{thm}

\begin{proof}

 (Sketch) The proof has the same structure as in the proof of the Prikry property.
  We only emphasize the key different ingredients from the proof of the Prikry property. For more details, consult the proof of the Prikry property.
  We will also only prove for pure conditions. 
  The proof for arbitrary conditions can be modified as in the proof of the Prikry property for impure conditions.

Let $p$ be a pure condition.
Fix an open dense set $D$.
We will build a $\leq^*$-decreasing sequence $\langle p^m : m<\omega \rangle$, such that if there is $q \leq p^m$ with $q \in D$, $\max(\supp(q))=m$ such that for $r \leq q^m$ with $r \in D$, and $\max(\supp(r))=m$, we have $\supp(q) \leq \supp(r)$ in the usual well-ordering in $[OR]^{<\omega}$, then for every $\vec{\mu} \in \prod\limits_{n \in \supp(q)} A_n^{p^m}$, we have $p^m+\vec{\mu} \in D$.
  It will then be routine to check that a lower bound of the sequence $\langle  p^m : m<\omega \rangle$ will satisfy the condition for the strong Prikry property.
  Let $\mathbb{B}_n=\mathbb{B}^{E_n}(\lambda,\lambda_n^j)$.

\textbf{Construction of} $p^0$: Let

\begin{center}

$\mathbb{Q}_0 = \{ \langle f_n:n < \omega \rangle \in \prod_{n<\omega} \mathbb{B}_n \mid \langle \dom(f_n): n< \omega \rangle$ is $\subseteq$-increasing$\}$

\end{center}

where $\prod_{n<\omega} \mathbb{B}_n$ is a full support product, and

\begin{center}

$\mathbb{R}_0 =\{ (f,\vec{r}) \in \mathbb{B}_0 \times (\mathbb{P}_{\langle E_n: n>0 \rangle},\leq^*): \dom(f) \subseteq \dom(f_1^r)\}$,

\end{center}

where $f_1^r$ is the first Cohen part of $\vec{r}$. 
Fix a sufficiently large regular cardinal $\theta$. Build an elementary submodel $N_0 \prec H_\theta$ of size $\lambda$ which is the union of an internally approachable chain of length $\kappa_0$, $N_0$ is closed under ${<}\kappa_0$-sequences containing enough information. 
Let $\langle f_n^0: n < \omega \rangle$ be $(N_0,\mathbb{Q}_0)$-generic, and $\langle f_n^0:n< \omega \rangle \leq \langle f_n^p: n<\omega \rangle$.
Let $d_n^0=\dom(f_n^0)$.
Let $A_n^0 \in E_n(d_n^0)$ be a measure-one set projecting down to $A_n^p$ for all $n$.

Fix $\mu \in A_0^0$.
Define $D_\mu$ as the collection of $(f,\vec{r})\in \mathbb{R}_0$ such that $\dom(\mu) \subseteq \dom(f)$, and if there is $(g,\vec{r}^\prime) \leq_{\mathbb{R}_0} (f \oplus \mu, \vec{r})$ such that $\langle g \rangle^\frown \vec{r}^\prime \in D$, then $\langle f \oplus \mu \rangle^\frown \vec{r} \in D$.
Let $\pi: \mathbb{R}_0 \to \mathbb{Q}_0$ be the natural projection.
Let $D_\mu^\prime = \pi[D_\mu]$.
Then $D_\mu$ is a dense open subset of $\mathbb{Q}_0$ and is in $N_0$. Hence, $\langle f_n^0: n \in \omega \rangle \in D_\mu^\prime$.
Let  $\langle A_n^0(\mu):n>0 \rangle$ be such that $\langle f_0^0 \rangle^\frown \langle \langle f_n^0,A_n^0(\mu):n>0 \rangle \in D_\mu$.
For $n>0$, let $(A_n^0)^* = A_n^0 \cap \triangle_{\mu \in A_0^0}A_n^0(\mu)$.
Let $q^0= \langle f_0^0,A_0^0 \rangle^\frown \langle \langle f_n^0,(A_n^0)^* \rangle: n>0 \rangle$.
For $\mu \in A_0^0$, $n>0$, we have that $A_n^{q^0+\mu}\subseteq \{ \tau \in (A_n^0)^*: \dom(\mu) \cup \rge(\mu) \subseteq \dom(\tau)\} \subseteq A_n^0(\mu)$.
Thus, we have the following property for $q^0$ (call it $(\star(q^0))$: 

\begin{center}

$\boldsymbol{(\star(q^0))}$ for $\mu \in A_0^0$, if there are $f \leq (q^0+\mu)_0$ (which is $f_0^0 \oplus \mu$), and $\vec{r} \leq^* (q^0+\mu) \setminus 1$ such that $\langle f \rangle^\frown \vec{r} \in D$, then $q^0+\mu \in D$.

\end{center}

Shrink $A_0^0$ to a measure-one set $B$ so that 

\begin{enumerate}

\item either for every $\mu \in A_0^0$, $q^0+\mu \in D$,

\item or for every $\mu \in A_0^0$, $q^0+\mu \not \in D$.

\end{enumerate}

From $q^0$, shrink the measure-one set $A_0^0$ to $B$ and call the new condition $p^0$.
Here is the property of $p^0$: if there is $q \leq p^0$ such that $\supp(q)=\{0\}$ and $q \in D$, then for every $\mu \in A_0^{p^0}$, $p^0+\mu \in D$.

\textbf{Construction of} $p^{m+1}$: Suppose $p^m$ is constructed. Let

\begin{center}

$\mathbb{Q}_{m+1} = \{ \langle f_n:n >m \rangle \in \prod_{n>m} \mathbb{B}_n \mid \langle \dom(f_n): n>m \rangle$ is $\subseteq$-increasing $\}$

\end{center}

and,

\begin{center}

$\mathbb{R}_{m+1} =\{ (f,\vec{r}) \in \mathbb{B}_{m+1} \times (\mathbb{P}_{\langle E_n: n>m+1 \rangle},\leq^*): \dom(f) \subseteq \dom(f_{m+2}^r)\}$,

\end{center}

where the first coordinate of $\vec{r}$ is $\langle f_{m+2}^r,A_{m+2}^r \rangle$.

Let $N_{m+1}$ be the union of an internally approachable chain of elementary substructure of $H_\theta$ for sufficiently large regular $\theta$, the length of the chain is $\kappa_{m+1}$, where $N_{m+1}$ contains ``enough" information, $|N_{m+1}|=\lambda$, ${}^{<\kappa_{m+1}}N_{m+1} \subseteq N_{m+1}$.
Let $\langle f_n^{m+1}: n >m \rangle$ be $(N_{m+1},\mathbb{Q}_{m+1})$-generic, and $\langle f_n^{m+1}:n> m \rangle \leq_{\mathbb{Q}_{m+1}} \langle f_n^{p^m}: n>m \rangle$.
Let $d_n^{m+1}=\dom(f_n^{m+1})$.
Let $A_n^{m+1} \in E_n(d_n^{m+1})$ be a measure-one set projecting down to $A_n^{p^m}$ for all $n>m$.
Fix $\mu \in A_{m+1}^{m+1}$.
Let $\mathbb{P}_{m+1,\mu}= \mathbb{P}_{\langle e_{n,m+1}(\mu):n \leq m \rangle}$.
Define $D_\mu$ as the collection of $(f,\vec{r}) \in \mathbb{R}_{m+1} $ such that $\dom(\mu) \subseteq \dom(f)$, and for each $t \in \mathbb{P}_{m+1,\mu}$, if there is $(g,\vec{r}^\prime) \leq_{\mathbb{R}_{m+1}} (f \oplus \mu,\vec{r})$ such that $t^\frown \langle g \rangle^\frown \vec{r}^\prime \in D$, then $t^\frown \langle f\oplus \mu \rangle^\frown \vec{r} \in D$. 
Let $\pi:\mathbb{R}_{m+1} \to \mathbb{Q}_{m+1}$ be the natural projection, and $D_\mu^\prime = \pi[D_\mu]$.
The set $D_\mu$ is a dense, open subset of $\mathbb{R}_{m+1}$, and $D_\mu \in N_{m+1}.$ 
Hence, we have $\langle f_n^{m+1}: n >m \rangle\in D_\mu^\prime$.
Let $A_n^{m+1}(\mu)$ be such that $\langle f_{m+1}^{m+1} \rangle^\frown \langle \langle f_n^{m+1},A_n^{m+1}(\mu) \rangle: n>m+1 \rangle \in D_\mu$.
For $n>m+1$, let $(A_n^{m+1})^*=A_n^{m+1} \cap \triangle_{\mu \in A_{m+1}^{m+1}}A_n^{m+1}(\mu)$.
Let $q^{m+1}= (p^m \restriction (m+1))^\frown \langle f_{m+1}^{m+1},A_{m+1}^{m+1} \rangle^\frown \langle \langle f_n^{m+1}, (A_n^{m+1})^* \rangle: n>m+1 \rangle$.
For $\mu \in A_{m+1}^{m+1}$ and $n>m+1$, $A^{q^{m+1}+\mu}_n \subseteq A_n^{m+1}(\mu)$.
Thus, we have the following property called $(\star(q^{m+1}))$:

\begin{center}

$\boldsymbol{(\star(q^{m+1}))}$: For $\mu \in A_{m+1}^{m+1}$ and $t \in \mathbb{P}_{m+1,\mu}$, if there are $g \leq f_{m+1}^q \oplus \mu$ and $\vec{r} \leq^* (q^{m+1}+\mu) \setminus (m+2)$ such that $t^\frown \langle g \rangle^\frown \vec{r} \in D$, then $t^\frown (q^{m+1}+\mu) \setminus (m+1) \in D$. 

\end{center}

Fix $\mu \in A_{m+1}^{m+1}$.
Outside $N_{m+1}$, let $\bar{D}_\mu$ be the collection of
$t \in \mathbb{P}_{m+1,\mu}$ such that either

\begin{center}
$t^\frown ((q^{m+1}+\mu) \setminus (m+1)) \in D$,
\end{center}

or for all $g \leq f_{m+1}^{q^{m+1}} \oplus \mu$,  $\vec{r} \leq^* (q^{m+1}+\mu) \setminus (m+2)$,
\begin{center}
$t ^\frown \langle g \rangle^\frown \vec{r} \not \in D$.
\end{center}
 
 We can use the property of $D_{\mu}$ to show that $\bar{D}_\mu$ is open dense (in fact, $\bar{D}_\mu=\mathbb{P}_{m+1,\mu})$.
  Use the induction hypothesis to find $t(\mu) \leq^* (q^{m+1}+\mu) \restriction (m+1)$, with the least $[\omega]^{<\omega}$-element in the lexicographic order $\vec{n}^\mu=\{n_0,\dots, n_{k(\mu)-1}^\mu\}$ such that $n^\mu_0< \dots <n^\mu_{k(\mu)-1}$,  such that for every $\vec{\tau} \in \prod\limits_{n \in \vec{n}^\mu}A_n^{t(\mu)}$, we have $t(\mu)+\vec{\tau} \in \bar{D}_\mu$.

For each $\mu \in A_{m+1}^{m+1}$, define
$F_\mu:A_{n_0^\mu}^{t(\mu)} \times \dots  \times A_{n_{k(\mu)-1}^\mu}^{t(\mu)} \to 2$ by $F_{\mu}(\tau_0, \dots, \tau_{k(\mu)-1})=1$ if and only if 

\begin{center}
$(t(\mu) + \langle \tau_0, \dots , \tau_{k(\mu)-1} \rangle)^\frown (q^{m+1}+\mu) \setminus (m+1)  \in D$.
\end{center}

We have a measure one set $B_{n^\mu_i}^{t(\mu)} \subseteq A_{n_i^\mu}^{t(\mu)}$ for all $i<k(\mu)$ such that
$F \restriction B_{n_0^\mu}^{t(\mu)} \times \dots B_{n_{k(\mu)-1}^\mu}^{t(\mu)}$ is constant. Shrink the measure one sets $A_{n_0^\mu}^{t(\mu)}, \dots, A_{n_{k(\mu)-1}^\mu}^{t(\mu)}$ inside $t(\mu)$ to $B_{n_0^\mu}^{t(\mu)}, \dots, B_{n_{k(\mu)-1}^\mu}^{t(\mu)}$, respectively.
Call the resulting condition $t^*(\mu)$. By the shrinking of measure one sets in $t(\mu)$, we have arranged that

\begin{enumerate}[label=(S\arabic*)]

\item \label{S1} either ($t^*(\mu) + \langle \tau_0, \dots, \tau_{k(\mu)-1} \rangle)^\frown (q^{m+1}+\mu) \setminus (m+1) \in D$ for all $\vec{\tau}$ in the product of measure-one sets $B_{n_i^\mu}^{t(\mu)}$,
\item \label{S2} or for all $\vec{\tau}$ in the product of measure-one sets $B_{n_i^\mu}^{t(\mu)}$, there are no $g \leq f_{m+1}^{m+1} \oplus \mu$, and $\vec{r} \leq^* (q^{m+1}+\mu) \setminus (m+2)$ such that $(t^*(\mu) + \langle \tau_0, \dots, \tau_{k(\mu)-1}) \rangle\frown \langle g \rangle^\frown \vec{r} \in D$.
\end{enumerate}

Shrink $A_{m+1}^{q^{m+1}}$ further so that every $\mu$ satisfies \ref{S1}, or every $\mu$ satisfies \ref{S2}. If every $\mu$ satisfies \ref{S1}, shrink further so that there is a sequence $\vec{n}_{m+1}$ such that for every $\mu \in A_{m+1}^\prime$, $\vec{n}_\mu:=\langle n_0^{t(\mu)}, \dots, n_{k(\mu)-1}^{t(\mu)} \rangle=\vec{n}_{m+1}$.

Observe that $t^*(\mu) \leq^* (p^m + (\mu \restriction \dom(f_{m+1}^{p^m}))) \restriction (m+1)$. 
Use Lemma \ref{integration2} to integrate these components (with $t^*(\mu)$, not $t(\mu)$) together to form a condition $p^{m+1} \leq^* p^m$. Hence, $f_{m+1}^{p^{m+1}} \leq f_{m+1}^{m+1}$, $A_{m+1}^{p^{m+1}}$ projects down to $A_{m+1}^{m+1}$, and $\tau \in A_{m+1}^{p^{m+1}}$ with $\mu=\tau \restriction \dom(f_{m+1}^{p^m})$, we have that $(p^{m+1}+\tau) \restriction (m+1) = t^*(\mu)$  and for $(p^{m+1}+\tau) \setminus (m+1) \leq^* (q^{m+1}+\mu) \setminus (m+1)$.
This completes the construction of $p^{m+1}$.
Here is what we have: if $q$ is an extension of $p^{m+1}$ such that $\supp(q)$ is the least in the lexicographic order in $[\omega]^{<\omega}$, $\max(\supp(q))=m+1$, and $q \in D$, then every extension $q^{\prime}$ of $p^{m+1}$ with $\supp(q^{\prime})=\supp(q)$ is in $D$. 
Now let $p^*$ be a $\leq^*$-lower bound of $\langle p^n: n<\omega \rangle$.

\begin{claim} 
\label{finalcondstrong}
$p^*$ satisfies the strong Prikry property.
\end{claim}

\begin{claimproof}{(Claim \ref{finalcondstrong})}

(sketch) Let $q\leq p^*$ with $q \in D$. Assume $q$ is not pure with the least $\supp(q)$ in the lexicographic order in $[\omega]^{< \omega}$, meaning if there is $r \leq p^*$ with $r \in D$, then $\supp(q) \leq \supp(r)$.
 Enumerate $\supp(q)$ in increasing order as $n_0<\dots < n_{k-1}$. If $n_{k-1}=0$, then the proof is easy. Assume $n_{k-1}=m+1$. Using the notations from the construction of $p^{m+1}$, we have that for every $\tau \in A^{q}_{m+1}$, $\tau \restriction d_{m+1}^{m+1}$ satisfies the property \ref{S1}, and $\vec{n}_{m+1}=\langle n_0, \dots,n_{k-2} \rangle$. By the way we shrank $A_{m+1}^{m+1}$, for every $\vec{\mu}\in \prod\limits_{n \in \vec{n}_{m+1} \cup \{m+1\}} A_n^{p^{m+1}}$, we have $p^{m+1}+ \vec{\mu} \in D$.
 This implies that for $\vec{\mu} \in \prod_{n\in \vec{n}_{m+1} \cup \{m+1\}} A_n^{p^*}$, $p^*+\vec{\mu} \in D$.
 Hence, the proof is done.
 
\end{claimproof}{(Claim \ref{finalcondstrong})}

This completes the proof of Theorem \ref{strongprikry1}.

\end{proof}

\section{forcing with arbitrarily many extenders}
\label{forcinganyextender}

We state without proofs in this section.
All the proofs are close to the proofs in Section \ref{forcingomegaextender}.
The structure of the proof of the Prikry property is almost the same as the proof in Section \ref{forcingomegaextender}: for each condition $p$ and a Boolean value $b$, build a $\leq^*$-decreasing sequence $\langle p^\alpha: \alpha<\eta \rangle$.
The construction of $p^\alpha$ for successor $\alpha$ is similar to the construction of $p^{m+1}$ in Theorem \ref{prikryomega}.
If $\alpha$ is limit, we only take $p^\alpha$ as a $\leq^*$-lower bound of $\langle p^\beta: \beta<\alpha \rangle$.
Assume $\eta>0$ is an arbitrary ordinal and the result from Theorem \ref{deriveext} holds for $\eta$.

\begin{defn}

A forcing $\mathbb{P}_{\langle E_\alpha: \alpha<\eta \rangle}$ consists of conditions $p$, and a {\em support of $p$}, which is $\supp(p) \in [\eta]^{<\omega}$, of the form $p=\langle p_\alpha: \alpha< \eta \rangle$, where

\begin{align*}
p_\alpha=
\begin{cases}
\langle f_\alpha \rangle & \text{ if } \alpha \in \supp(p), \\
\langle f_\alpha ,A_\alpha \rangle & \text{ otherwise},
\end{cases}
\end{align*}
and for each $\alpha<\eta$ such that $\supp(p) \setminus (\alpha+1) \neq \emptyset$, let $\alpha^*=\min(\supp(p) \setminus (\alpha+1))$, then the following hold:

\begin{enumerate}

\item for each $\alpha$ where $\alpha^*$ does not exist,

\begin{itemize}
  
\item if $\alpha \in \supp(p)$ (which is exactly when $\alpha=\max(\supp(p))$), then $f_\alpha\in \mathbb{B}^{E_\alpha}(\lambda,\lambda_\alpha^j)$,  $d_\alpha:=\dom(f_\alpha)$ is an $\alpha$-domain with respect to $E_\alpha$, and if $\alpha>0$, $f_\alpha(\kappa_\alpha)$ is $\alpha$-reflected for the sequence $\langle E_\beta: \beta \in [\max\{0,\max(\supp(p) \cap \alpha)\},\alpha) \rangle$

\item if $\alpha \not \in \supp(p)$ (which is exactly when $\alpha>\max(\supp(p))$, 

\begin{itemize}

\item $f_\alpha \in\mathbb{B}^{E_\alpha}(\lambda,\lambda_\alpha^j)$.

\item Let $d_\alpha=\dom(f_\alpha)$ is an $\alpha$-domain with respect to $E_\alpha$, and $A_\alpha \in E_\alpha(d_\alpha)$.

\end{itemize}

\end{itemize}

\item for each $\alpha$ where $\alpha^*$ exist, then 

\begin{itemize}

\item if $\alpha \in \supp(p)$ (which is exactly when $\alpha=\max(\supp(p) \cap \alpha^*))$), $f_\alpha \in \mathbb{B}^{e_{\alpha,\alpha^*}(f_{\alpha^*})}(\lambda_{\alpha^*}(f_{\alpha^*}),\lambda_{\alpha,\alpha^*}^j(f_{\alpha^*}))$,  $d_\alpha:=\dom(f_\alpha)$ is an $\alpha$-domain with respect to $e_{\alpha,\alpha^*}(f_{\alpha^*})$, and if $\alpha>0$, then $f_\alpha(\kappa_\alpha)$ is $\alpha$-reflected for the sequence $\langle e_{\beta,\alpha^*}^j(f_{\alpha^*}): \beta \in [\max\{0,\max(\supp(p) \cap \alpha)\},\alpha) \rangle$.
Recall that $\lambda_{\alpha^*}(f_{\alpha^*})=s_{\alpha^*}(f_{\alpha^*}(\kappa_{\alpha^*}))$ where $j_{E_{\alpha^*}}(s_{\alpha^*})(\kappa_{\alpha^*})=\lambda$, $\lambda_{\alpha,\alpha^*}^j(f_{\alpha^*})=t_{\alpha^*}^n(f_{\alpha^*}(\kappa_{\alpha^*}))$ where $j_{E_{\alpha^*}}(t_{\alpha^*}^\alpha)(\kappa_{\alpha^*})=\lambda_\alpha^j$ and $\lambda_\alpha^j=j_{E_\alpha}(\lambda)$, and $e_{\alpha,\alpha^*}(f_{\alpha^*})=E_\alpha \restriction t_{\alpha^*}^\alpha(f_{\alpha^*}(\kappa_{\alpha^*}))$.

\item if $\alpha \not \in \supp(p)$ (which is exactly when $\alpha \in (\max(\supp(p) \cap \alpha^*),\alpha^*)$ and $\max(\emptyset)=-1$), 

\begin{itemize}

\item $f_\alpha \in \mathbb{B}^{e_{\alpha,\alpha^*}(f_{\alpha^*})}(\lambda_{\alpha^*}(f_{\alpha^*}),\lambda_{\alpha,\alpha^*}^j(f_{\alpha^*}))$, and $d_\alpha:=\dom(f_\alpha)$ is an $\alpha$-domain with respect to $e_{\alpha,\alpha^*}(f_{\alpha^*})$.

\item $\rge(f_\alpha \restriction j_{e_{\alpha,\alpha^*}(f_{\alpha^*})}(\kappa_\alpha)) \subseteq \kappa_\alpha$.

\item $A_\alpha \in e_{\alpha,\alpha^*}(f_{\alpha^*})(d_\alpha)$.

\end{itemize}

\end{itemize}

\item For $\max(\supp(p)) \leq \beta < \alpha < \eta$, $d_\beta \subseteq d_\alpha$.

\item If $\beta<\eta$ is such that $\beta^*$ exists, then for each $\beta<\alpha<\beta^*$, $d_\beta \subseteq d_\alpha$.

\item if $\alpha<\eta$ is limit and $\alpha \in \supp(p)$, then 

\begin{itemize}

\item if $\alpha^*$ does not exist, then $f_\alpha(\lambda)=s_\alpha(f_\alpha(\kappa_\alpha))$, and for $\beta<\alpha$, $f_\alpha(\lambda_\beta^j)=t_\alpha^\beta(f_\alpha(\kappa_\alpha))$.
Also, $f_\alpha(\bar{\lambda}_\alpha^j)=u_\alpha^\alpha(f_\alpha(\kappa_\alpha))$, which is greater than $t_\alpha^\beta(f_\alpha(\kappa_\alpha))$ for all $\beta<\alpha$ (recall $j_{E_{\alpha}}(u_\alpha^\alpha)(\kappa_\alpha)=\bar{\lambda}_\alpha^j=\sup_{\beta<\alpha}(\lambda_\beta^j)$, where $\lambda_\beta^j=j_{E_{\beta}}(\lambda)$.

\item if $\alpha^*$ exists, then
$f_\alpha(\lambda_{\alpha^*}(f_{\alpha^*}))=s_\alpha(f_\alpha(\kappa_\alpha))$, and for $\beta<\alpha$, $f_\alpha(\lambda_{\beta,\alpha^*}^j)=t_\alpha^\beta(f_\alpha(\kappa_\alpha))$.
Also, $f_{\alpha}(\bar{\lambda}_{\alpha,\alpha^*}^j)=u_\alpha^\alpha(f_\alpha(\kappa_\alpha))$, which is greater than $t_\alpha^\beta(f_\alpha(\kappa_\alpha))$ for all $\beta<\alpha$.

\end{itemize}

\end{enumerate}

\end{defn}

A condition $p$ is {\em pure} if $\supp(p)=\emptyset$.
Otherwise, $p$ is said to be {\em impure}.

\begin{defn}

For $p,q \in \mathbb{P}$.
We say that $p$ is a {\em direct extension} of $q$, denoted by $p \leq^*q$ if

\begin{enumerate}

\item $\supp(p)=\supp(q)$.

\item for all $\alpha<\eta$, $f_\alpha^p\leq f_\alpha^q$.

\item for $\alpha \not \in \supp(p)$, $A_\alpha^p \restriction d_\alpha^q \subseteq A_\alpha^q$.

\end{enumerate}

\end{defn}

Notice that if $p \leq^*q$, then for $\alpha \in \supp(p)$, $f_\alpha^p(\kappa_\alpha)=f_\alpha^q(\kappa_\alpha)$, so all parameters defined from $f_\alpha^p(\kappa_\alpha)$ are the same as those defined from $f_\alpha^q(\kappa_\alpha)$. 

\begin{defn}

Let $p \in \mathbb{P}$ and $\alpha \not \in \supp(p)$.
Let $\mu \in A_\alpha^p$ be $\langle p_\beta: \beta \in [\max(\supp(p) \cap \alpha),\alpha) \rangle$-squishable.
The {\em one-step extension of $p$ by $\mu$} is the condition $q$, denoted by $p+\mu$, such that

\begin{enumerate}

\item $\supp(q)=\supp(p) \cup \{\alpha\}$.

\item for $\beta<\max(\supp(p) \cap \alpha)$, $q_\beta=p_\beta$.

\item if $\alpha^*=\min(\supp(p) \setminus (\alpha+1))$ exists, then for $\beta \in (\alpha,\alpha^*)$, $q_\beta=\langle f_\beta^p, B_\beta \rangle$ where $B_\beta = \{\tau \in A_\beta^p: \tau$ is $\langle f_\alpha^p,A_\alpha^p \rangle$-squishable, and $\dom(\mu) \cup \rge(\mu)  \subseteq \dom( \tau)\}$, and for $\beta \geq \alpha^*$, $q_\beta=p_\beta$.

\item if $\alpha^*$ does not exist, then for $\beta>\alpha$, $q_\beta=\langle f_\beta^p, B_\beta \rangle$ where $B_\beta = \{\tau \in A_\beta^p: \tau$ is $\langle f_\alpha^p,A_\alpha^p \rangle$-squishable and $\dom(\mu) \cup \rge(\mu)  \subseteq \dom( \tau)\}$. 

\item $f_\alpha^q=f_\alpha^p \oplus \mu$.

\item for $\beta \in [\max(\supp(p) \cap \alpha),\alpha)$, $f_\beta^q=\mu \circ f_\beta^p \circ  \mu^{-1}$ and $A_\beta^q=\mu \circ (A_\beta^p)_\mu \circ \mu^{-1}$.

\end{enumerate}

\end{defn}

We define an $n$-step extension recursively as follows:
$p$ is an {\em $n$-step extension of $q$} for $n>1$ if $p=(q+\langle \mu_0, \cdots, \mu_{n-2} \rangle) + \mu_{n-1}$, under the condition that for $i<n$, $\mu_i$ is legitimate to perform a $1$-step extension into $ q+\langle \mu_0, \cdots, , \mu_{i-1} \rangle$.
Define $p \leq q$ if $p$ is a direct extension of some $n$-step extension of $q$ ($n$ can be $0$).

\begin{lemma}
\label{transitivity3}
\begin{enumerate}

\item Let $\alpha<\eta$.
Let $p$ be a condition such that $\alpha \not \in \supp(p)$ and $q \leq^* p$.
Suppose $\mu \in A_\alpha^q$, $q+\mu$ is valid.
Then $q+\mu \leq^* p+(\mu \restriction d_\alpha^p)$.

\item The ordering $\leq$ is transitive.

\item \label{commute2} Suppose $p$ is pure and $q$ is a $n$-step extension of $p$ at coordinates $\alpha_0< \dots < \alpha_{n-1}$.
Then there are $\mu_i \in A_{\alpha_i}^p$ for $i<n$ such that $q=p+ \langle \mu_0, \dots, \mu_{n-1} \rangle$.
As a consequence, the order of the objects we use to extend $p$ to $q$ can be commuted modulo squishing.

\end{enumerate}

\end{lemma}

\begin{lemma}

$\mathbb{P}$ is $\lambda^{++}$-c.c.

\end{lemma}

If $p \in \mathbb{P}$ and $\alpha \in \supp(p)$, then $\mathbb{P} / p$ factors into two posets $\mathbb{P}_0=((\mathbb{P}/p) \restriction \alpha)_{\langle e_{\beta,\alpha}(f_\alpha^p)): \beta<\alpha \rangle}$ and $\mathbb{P}_1=(\mathbb{P}/p) \setminus \alpha$.

\begin{lemma}
\label{factorization}
$\mathbb{P}_0$ is $\lambda_\alpha(f_\alpha^p)^{++}$-c.c.
If $ \alpha+1 \not \in \supp(p)$,  $(\mathbb{P}_1,\leq^*)$ is $\kappa_{\alpha+1}$-closed.
If $\alpha+1 \in \supp(p)$, $(\mathbb{P}_1,\leq^*)$ is $\lambda_{\alpha+1}(f_{\alpha+1}^p)^+$-closed.

\end{lemma}

\begin{lemma}

Let $p \in \mathbb{P}$ and $\alpha \not \in \supp(p)$.
Let $f^\prime \leq f_\alpha^p$ with $d^\prime=\dom(f^\prime)$. 
Assume that if $\alpha^*=\min(\supp(p) \setminus (\alpha+1))$ exists, $A^\prime \in e_{\alpha,\alpha^*}(f_{\alpha^*})(d^\prime)$, and if $\alpha^*$ does not exist, $A^\prime \in E_\alpha(d^\prime)$.
Fix $f_\beta$ for $\beta>\alpha$ such that $f_\beta \leq f_\beta^p$ and $d_\beta:=\dom(f_\beta) \supseteq d^\prime$, and if $\alpha<\beta<\beta^\prime$, $\dom(f_\beta) \subseteq \dom(f_{\beta^\prime})$.
Suppose that for each $\mu \in A^\prime$, there is a condition $t(\mu) \leq^* (p+ (\mu \restriction d_\alpha^p)) \restriction \alpha$, and there is $\vec{r}(\mu)=\langle f_\beta,A_\beta(\mu): \beta>\alpha\rangle$ such that $A_\beta(\mu)$ is of measure-one with respect to $\dom(f_\beta)$. 
Then there is a condition $q \leq^* p$ such that 

\begin{enumerate}

\item $f_\alpha^q \leq f^\prime$ and $A_\alpha^q$ projects down to $A^\prime$.

\item for $\tau \in A_\alpha^q$ with $\mu= \tau \restriction A^\prime$, we have that 

\begin{itemize}

\item $(q+\tau) \restriction \alpha = t(\mu)$.

\item $(q+\tau) \setminus (\alpha+1) \leq^* \vec{r}(\mu)$.

\end{itemize}

\end{enumerate}

\end{lemma}

\begin{thm}

$\mathbb{P}$ has the Prikry property.

\end{thm}

\begin{thm}

$\mathbb{P}$ has the strong Prikry property. Namely, for each dense open set $D \subseteq \mathbb{P}$ and $p \in \mathbb{P}$, there is a condition $q \leq^*p$ and a finite subset $a$ of $\eta$ ($a$ can be empty) such that

\begin{enumerate}

\item $a \cap \supp(p) =\emptyset$.

\item every $|a|$-step extension of $p$ using objects from $\{A_\alpha^q:\alpha \in a\}$ lies in $D$.

\end{enumerate}

\end{thm}

\section{cardinal preservation}
\label{cardinalpreserve}
Let $\eta$ be a limit ordinal and $\mathbb{P}=\mathbb{P}_{\langle E_\alpha: \alpha<\eta \rangle}$.
In this section we determine the cardinals which are preserved and collapsed.
Note that $\mathbb{P}$ is $\lambda^{++}$-c.c., so every cardinal above and including $\lambda^{++}$ is preserved.

\begin{thm}

When forcing with $\mathbb{P}$, cardinals below and including $\bar{\kappa}_\eta$ which are not in the intervals $(\bar{\kappa}_\alpha,\kappa_\alpha)$ for $\alpha<\eta$ limit are preserved.

\end{thm}

\begin{proof}

First, note that $(\mathbb{P}, \leq^*)$ is $\kappa_0$-closed, so every cardinal below and including $\kappa_0$ is preserved.
We now consider the cardinals in the interval $(\kappa_\alpha,\kappa_{\alpha+1}]$.
Let $p \in \mathbb{P}$ be such that $\alpha,\alpha+1 \in \supp(p)$.
Then $\mathbb{P}/p$ factors into 3 posets, namely $\mathbb{P}_0=(\mathbb{P}/p) \restriction \alpha=\mathbb{P}_{\langle e_{\beta,\alpha}(f_\alpha^p) : \beta<\alpha \rangle}$, $\mathbb{P}_1= (\mathbb{P}/p)(\alpha)=\mathbb{B}^{e_{\alpha,\alpha+1}(f_{\alpha+1}^p)}(\lambda_{\alpha+1}(f_{\alpha+1}^p)$, $\lambda_{\alpha,\alpha+1}^j(f_{\alpha+1}^p))$, and $\mathbb{P}_2=\mathbb{P}_{\langle E_\beta: \beta \geq \alpha+1 \rangle}$.
$\mathbb{P}_0$ is $\lambda_\alpha(f_\alpha^p)^{++}$-c.c. and $\lambda_\alpha(f_\alpha^p)^{++}<\kappa_\alpha$.
$\mathbb{P}_1$ is $\lambda_{\alpha+1}(f_{\alpha+1}^p)$-closed, and $\lambda_{\alpha+1}(f_{\alpha+1}^p)^+$-c.c., so $\mathbb{P}_1$ preserves all cardinals, and
$(\mathbb{P}_2,\leq^*)$ is $\kappa_{\alpha+1}$-closed.
Hence, all cardinals in the interval $(\kappa_\alpha,\kappa_{\alpha+1}]$ are preserved.
For $\alpha<\eta$ limit, $\bar{\kappa}_\alpha=\sup_{\beta<\alpha} \kappa_\beta$, so $\bar{\kappa}_\alpha$ is preserved.

\end{proof}

To understand the cardinal behaviors in the interval $(\bar{\kappa}_\alpha,\kappa_\alpha)$ for $\alpha<\eta$ limit, as well as the interval $(\bar{\kappa}_\eta,\lambda^{++})$, let's simplify by considering the case where $\eta=\omega$ first, since this will be used to describe general cases.

Let $\eta=\omega$, namely $\mathbb{P}=\mathbb{P}_{\langle E_n: n<\omega \rangle}$.
We analyze the $V$-cardinals in the interval $(\bar{\kappa}_\omega,\lambda^+)$ in a generic extension.
Fix a cardinal $\gamma \in (\bar{\kappa}_\omega,\lambda^+)$ with $\cf(\gamma) > \bar{\kappa}_\omega$.
Then for all $n$, in $M_n$, $\cf(\sup j_{E_n}[\gamma])=\cf(\gamma)\leq \lambda<j_{E_n}(\kappa_n)<\cf(j_{E_n}(\gamma))$, which implies $\sup j_{E_n}[\gamma] < j_{E_n}(\gamma)$.
Set $\gamma_n=\sup(j_{E_n}[\gamma])$ and $\bar{\gamma}_n=\gamma^{j_{E_n}}$ (which is $j_{E_n}(\gamma))$).
In $V[G]$, define $H_\gamma$ whose domain is $\omega$ as follows:
if there is a condition $p \in G$ such that $\max(\supp(p))=n$ and $\gamma_n \in \dom(f_n^p)$, define $H_\gamma(n)=f_n^p(\gamma_n)$.
Otherwise, let $H_\gamma(n)=0$.
By genericity, $H_\gamma$ is well-defined.

\begin{lemma}
\label{addcofseq}
When forcing with $\mathbb{P}_{\langle E_n:n<\omega \rangle}$, $H_\gamma$ is an $\omega$-cofinal sequence in $\gamma$.

\end{lemma}

\begin{proof}

Assume $\gamma<\lambda$ (the case $\gamma=\lambda$ is similar).
Let $\nu<\gamma$.
Define $\nu_n=\nu^{j_{E_n}}$.
Define $D_\nu$ as the collection of $p \in \mathbb{P}$ such that

\begin{enumerate}

\item there is $n_0<\omega$ such that $\supp(p) \subseteq n_0$.

\item for $n \geq n_0$,

\begin{itemize}

\item $\nu_n,\gamma_n,\bar{\gamma}_n \in d_n^p$.

\item for $\mu \in A_n^p$, $\nu_n,\gamma_n,\bar{\gamma}_n \in \dom(\mu)$, $\mu(\nu_n)=\nu$ and $\mu(\bar{\gamma}_n)=\gamma$.

\end{itemize}

\end{enumerate}

Clearly $D_\nu$ is dense.
Let $p \in D_\nu \cap G$.
For $q \leq p$ with $\max(\supp(q))=n \geq n_0$, $q \in G$, we have that $q \leq p+ \mu$ for some $\mu \in A_n^p$.
Hence, $\nu_n,\gamma_n,\bar{\gamma}_n \in d_n^q$, $\nu=f_n^q(\nu_n)=\mu(\nu_n)<\mu(\gamma_n)=f_n^q(\gamma_n)$ and $f_n^q(\gamma_n)=\mu(\gamma_n)<\mu(\bar{\gamma}_n)=\gamma$.
Hence, $H_\gamma(n) \in (\nu,\gamma)$.
Since $\nu<\gamma$ is arbitrary, we are done.

\end{proof}

We see that for every $V$-cardinal $\gamma$ in the interval $(\bar{\kappa}_\omega,\lambda^+)$ with $\cf(\gamma)>\bar{\kappa}_\omega$, $V[G]$ adds an $\omega$-cofinal sequence of $\gamma$.
Recall that $\bar{\kappa}_\omega$ is preserved. 
We now explain that if $\gamma \in (\bar{\kappa}_\omega,\lambda^+)$, $\gamma$ is collapsed to have cardinal $\bar{\kappa}_\omega$.
Otherwise, let $\gamma$ be the least in the interval which is preserved.
Then $\gamma=(\bar{\kappa}_\omega^+)^{V[G]}$.
If $\gamma$ is a cardinal with $\cf^V(\gamma)>\bar{\kappa}_\omega$, then $\cf^{V[G]}(\gamma)=\omega$, which is a contradiction.
If $\cf(\gamma)<\bar{\kappa}_\omega$, then $\cf^{V[G]}(\gamma)<\bar{\kappa}_\omega<(\bar{\kappa}_\omega^+)^{V[G]}$,   contradicting the fact that $\gamma=(\bar{\kappa}_\omega^+)^{V[G]}$.

\begin{coll}
\label{collapse1}
When forcing with $\mathbb{P}_{\langle E_n:n<\omega \rangle}$, all cardinals in the interval $(\bar{\kappa}_\omega,\lambda^+)$ are collapsed.

\end{coll}

\begin{lemma}
\label{strongpreserve1}
When forcing with $\mathbb{P}_{\langle E_n:n<\omega\rangle}$, $\lambda^+$ is preserved.

\end{lemma}

\begin{proof}

Suppose not.
Since all cardinals in the interval $(\bar{\kappa}_\omega,\lambda^+)$ are collapsed,  $(\lambda^+)^V$ is collapsed to have size $\bar{\kappa}_\omega$.
Then in $V[G]$, let $\xi=\cf{\lambda^+}<\overline{\kappa}_\omega$. Choose $n<\omega$ such that $\xi<\kappa_n$. 
Extend $p$ so that $n \in \supp(p)$. Break $p$ into $p \restriction n$ and $p \setminus n$. Since $p \restriction n$ lies in $\mathbb{P}_{\langle e_{m,n}(f_n^p): m<n \rangle}$ which is $\lambda_n(f_n^p)$-.c.c., so $\lambda^+$ is collapsed in the forcing in which $p \setminus n$ lives (which is $(\mathbb{P}/p) \setminus n$).
Note that $(\mathbb{P}/p) \setminus n ,\leq^*)$ is $\kappa_n$-closed.

In $V$, let $\{\dot{\gamma}_i:i<\xi\}$ be a sequence of names, forced by $ p \setminus n$ to be a cofinal sequence in $\lambda^+$. Build a sequence of conditions
$\{p_i : i<\xi \}$ such that $p_0=p \setminus n$, $\{p_i : i<\xi \}$ is $\leq^*$-decreasing, and $p_{i+1}$ satisfies Lemma \ref{strongprikry1} for $D_i=\{q \in \mathbb{P} \setminus n: q$ decides the value of $\dot{\gamma}_i \}$. 

Set $r$ to be a $\leq^*$-lower bound of $\{p_i :i<\xi\}$ in $\mathbb{P} \setminus n$. 
Since each measure-one set has size at most $\lambda$, for each $i<\xi$, $A_i=\{\beta : \exists r^\prime \leq r, r^\prime \Vdash \dot{\gamma}_i=\check{\beta}\}$ has size at most $\lambda$.
Set $\beta_i=\sup A_i$, and $\beta=\sup\limits_{i<\xi} \beta_i$. Then $r \Vdash \sup\{ \dot{\gamma}_i:i<\xi\} \leq \check{\beta}$ and $\beta<(\lambda^+)^V$, which is a contradiction.

\end{proof}

We now consider the case $\eta>\omega$ is limit.
For $\alpha<\eta$ limit and $p \in G$ with $\alpha \in \supp(p)$ and $\lambda_\alpha=\lambda_\alpha(f_\alpha^p)$, the forcing $\mathbb{P}_\alpha=(\mathbb{P}/p) \restriction \alpha$  is equivalent to the forcing defined from a sequence of extenders with shorter length $\mathbb{P}_{\langle e_{\beta,\alpha}(f_\alpha^p): \beta<\alpha \rangle}$.
We can generalize the proofs of Lemma \ref{addcofseq}, Corollary \ref{collapse1}, and Lemma \ref{strongpreserve1} in a similar fashion to show the following:

\begin{thm}
\label{cardinalstructure}
Assume $\mathbb{P}=\mathbb{P}_{\langle E_\alpha:\alpha<\eta \rangle}$ and $\eta>\omega$ is regular.
For $\alpha<\eta$ limit, let $\lambda_\alpha=\lambda_\alpha(f_\alpha^p)$ for some $p \in G$.
Then in the interval $(\bar{\kappa}_\alpha,\kappa_\alpha)$, all cardinals in $(\bar{\kappa}_\alpha,\lambda_\alpha^+)$ are collapsed, cardinals in the interval $[\lambda_\alpha^+,\kappa_\alpha)$ are preserved.
In particular, $(\lambda_\alpha^+)^V=(\bar{\kappa}_\alpha^+)^{V[G]}$.
For $\gamma \in (\bar{\kappa}_\alpha,\lambda_\alpha^+)$ with  $\cf^V(\gamma)>\bar{\kappa}_\alpha$, we have that $cf^{V[G]}(\gamma)=\cf(\alpha)$.
In addition, all cardinals in the intervals $(\bar{\kappa}_\eta,\lambda^+)$ are collapsed, and $\lambda^+$ is preserved, hence, $(\lambda^+)^V=(\bar{\kappa}_\eta^+)^{V[G]}$.
For $\gamma \in (\bar{\kappa}_\eta,\lambda^+)$ with $\cf^V(\gamma)>\bar{\kappa}_\eta$, we have that $\cf^{V[G]}(\gamma)=\eta$.
Other $V$-cardinals which were not mentioned are preserved.

\end{thm}

\section{blowing up power sets}
\label{blowinguppowersets}
We now show that forcing with $\mathbb{P}$ will add $|\bar{\lambda}_\eta^j|$ new subsets of $\bar{\kappa}_\eta$ (recall that $\bar{\lambda}_\eta^j=\sup_{\alpha<\eta}j_{E_\alpha}(\lambda)$ and $\bar{\kappa}_\eta =\sup_{\alpha<\eta} \kappa_\alpha)$.
Lemma \ref{scaledef} will be used to obtain a condition in the generic extension that can be used to define scales (the formal definition of scales is in Section \ref{scaleanalysis}).
The technical conditions in the lemma ascertain that the scales will be well-defined.

\begin{lemma}\label{scaledef}
  Let $\beta \leq \eta$ be a limit ordinal.
   Let $p^\beta$ be a condition such that 
   
   \begin{enumerate}
   
\item if $\beta<\eta$, then we require that $\beta \in \supp(p^\beta)$, we let $\lambda_\beta=\lambda_\beta(f_\beta^{p^\beta})$,  $\lambda_{\alpha,\beta}^j=t_\beta^\alpha(f_\beta^{p^\beta}(\kappa_\beta))$ and $\theta_\beta:=u_\beta^\beta(f_\beta^{p^\beta}(\kappa_\beta))$ (recall that $j_{E_{\beta}}(t_\beta^\alpha)(\kappa_\beta)=\lambda_\alpha^j$, and $j_{E_{\beta}}(u_\beta^\beta)(\kappa_\beta)=\bar{\lambda}_\beta^j=\sup_{\alpha<\beta} \lambda_\alpha^j$).
	
	\item if $\beta=\eta$, then let $\lambda_\beta=\lambda$, $\lambda_{\alpha,\beta}^j=\lambda_\alpha^j$, and $\theta_\beta=\bar{\lambda}_\eta^j$.   
   
   \end{enumerate}
   
  Fix $\gamma \in [\overline{\kappa}_\beta,\theta_\beta)$.
 We also assume that $p^\beta$ satisfies the following: there is $\bar{\alpha}<\beta$ with $\supp(p) \cap \beta \subseteq \bar{\alpha}$, $\gamma \in d_{\bar{\alpha}}^{p^\beta}$ and for $\alpha^\prime \geq \bar{\alpha}+1$ and $\mu \in A_{\alpha^\prime}^{p^\beta}$, we have $\gamma \in \dom(\mu)$ and $\mu(\gamma)<j_{e_\iota,e_{\alpha^\prime}(\mu)}(\kappa_\iota)$ for all $\iota \in [\alpha+1,\alpha^\prime]$ (here $j_{e_\alpha^\prime,e_\alpha^\prime(\mu)}(\kappa_\alpha^\prime):=\kappa_\alpha^\prime)$.
 Fix $\alpha \in [\bar{\alpha},\beta)$.
   Let $D$ be the collection of $q \leq p^\beta$ such that 
\begin{enumerate}
\item $\alpha \in \supp(q)$.
\item If we enumerate $(\supp(q) \cap \beta) \setminus \alpha$ in decreasing order as
    $\alpha_0> \dots >\alpha_{k-1}$, then 
    \begin{enumerate}
    
	\item $k>1$ and $\alpha_{k-1}=\alpha$ and $\alpha_{k-2}=\alpha+1$.
	
	\item \label{gamma_i} $\gamma \in \dom(f^q_{\alpha_0})$.  
	Furthermore, let $\langle \gamma_i \rangle_i$ be the sequence of ordinals defined inductively by setting $\gamma_0 = \gamma$ and $\gamma_{i+1} = f^q_{\alpha_i}(\gamma_i)$ for as long as $\gamma_i \in \dom(f^q_{\alpha_i})$, then $\langle \gamma_i \rangle_i$
    reaches a stage where $\gamma_{k-1}$ is defined, and $\gamma_{k-1} \in \dom(f^q_\alpha)$. 
    
  \item \label{technicalcase2} let $\langle \gamma_i \rangle_i$ be as in (\ref{gamma_i}), then for $i<k-2$, we have that $\gamma_{i+1}<j_{e_{\alpha_{i+1},\alpha_i}(f_{\alpha_i}^q)}(\kappa_{\alpha_{i+1}})$.
  If $f_{\alpha_i}^q$ is satisfies the condition in this item, we say that $f_{\alpha_i}^q$ is {\em sensible}.
    
    \item \label{satisfyobject} let $\xi \in (\beta \setminus (\alpha+2))$.
     If $\xi>\alpha_0$, then for $\mu \in A_\xi^q$, $\gamma_0,\gamma_1 \in \dom(\mu)$ and $\mu(\gamma_0)<j_{e_{\alpha_1,\alpha_0}(\mu)}(\kappa_{\alpha_1})$
     If $i \in [0,k-2)$ is the least such that $\xi<\alpha_i$, then for $\mu \in A_\xi^q$, $\gamma_{i+1},\gamma_{i+2} \in \dom(\mu)$ and $\mu(\gamma_{i+1})<j_{e_{\alpha_{i+1},\alpha_i}(\mu)}(\kappa_{\alpha_{i+1}})$.
       If $f_{\alpha_i}^q$ is satisfies the condition in this item, we say that $f_{\alpha_i}^q$ is {\em sensible}.
	
    \end{enumerate}
    
    \end{enumerate}

Then $D$ is open dense below $p^\beta$.

\end{lemma}

\begin{proof}

We prove for the case $\beta=\eta$.
The case $\beta<\eta$ is similar.
We will show something stronger: if $q\leq p$ and $\alpha,\alpha+1 \in \supp(q)$, then $q \in D$.
Let $q \leq p$ be such that $\alpha,\alpha+1 \in \supp(q)$.
Enumerate $\supp(q) \setminus \alpha$ as $\alpha_0>\alpha_1> \dots >\alpha_{k-2}>\alpha_{k-1}$, where $k>1$, $\alpha_{k-1}=\alpha$ and $\alpha_{k-2}=\alpha+1$.
Since $q \leq p$, let $\mu_i \in A_{\alpha_i}^{p^\beta}$ be such that $q \leq^* p+ \langle \mu_{k-1},\mu_{k-2}, \cdots, \mu_1,\mu_0 \rangle$ (note that the order of the objects is important).
By induction on $i<k-1$, we show that

\begin{enumerate}

\item for $\gamma_i \in \dom(f_{\alpha_i}^q)$ and $f_{\alpha_i}^q$ is sensible.

\item if $\xi \in (\alpha_i,\alpha_{i-1})$ and $\mu \in A_\xi^q$, then $\mu$ is sensible, where $\alpha_{-1}:=\eta$.

\item $\gamma_{i+1}=\mu_i(\gamma)$.

\end{enumerate}

$i=0$: Note that $f_{\alpha_0}^q=f_{\alpha_0}^p \oplus \mu_0$, so $\gamma_0 \in \dom(f_{\alpha_0}^q)$, and $f_{\alpha_0}^q(\gamma_0)=\mu_0(\gamma)<j_{e_{\alpha_1,\alpha_0}(\mu_0)}(\kappa_{\alpha_1})$.
For $\xi \in (\alpha_0,\eta)$ and $\mu \in A_\xi^q$, $\dom(\mu_0) \cup \rge(\mu_0)$, so $\gamma_0,\gamma_1 \in \dom(\mu)$.
Clearly $\mu$ is sensible.
Note that $\gamma_1=\mu_0(\gamma)$.

$i=j+1$: By commuting the objects, we can see that $f_{\alpha_i}^q=\mu_j \circ (f_{\alpha_i}^p \oplus \mu_i) \circ \mu_j^{-1}$.
Note that $\gamma_i=\mu_j(\gamma)$, so $ (f_{\alpha_i}^p \oplus \mu_i) \circ \mu_j^{-1}(\gamma_i)=\mu_i(\gamma)$.
Since $i<k-1$, we have $\alpha_i>\alpha$, so $\gamma< j_{E_{\alpha_i}}(\kappa_{\alpha_i})$.
By this clause, $\mu_i(\gamma)<\kappa_{\alpha_i}$.
Hence $\mu_i(\gamma) \in \dom(\mu_j)$ and $\mu_j$ fixes $\mu_i(\gamma)$.
As a consequence, $\gamma_{i+1}=f_{\alpha_i}^q(\gamma_i)=\mu_i(\gamma)$, which is sensible.
Now for $\xi \in  (\alpha_i,\alpha_j)$, $\mu \in A_\xi^q$, we have that $\dom(f_{\alpha_i}^q) \cup \rge(f_{\alpha_i}^q) \subseteq \dom(\mu)$, hence, $\gamma_i,\gamma_{i+1} \in \dom(\mu)$.
Finally, since $\mu \in A_\xi^q$, $\mu=\mu_j \circ \tau \circ \mu_j^{-1}$ for some $\tau \in A_\xi^p$.
Hence, $\mu(\gamma_i)=\mu_j \circ \tau(\gamma)=\tau(\gamma)$ (again, $\mu_j$ fixes $\tau(\gamma)$).
Hence $\mu$ is sensible.

Now we see that $\gamma_{k-1}$ is defined.
Note that $f_\alpha^q=\mu_{k-2} \circ(f_\alpha^p \oplus \mu_{k-1}) \circ \mu_{k-2} ^{-1}$.
Since $\gamma_{k-1}=\mu_{k-2}(\gamma)$, $(f_\alpha^p \oplus \mu_{k-1}) \circ \mu_{k-2}^{-1}(\gamma_{k-1})=\mu_{k-1}(\gamma)$.
Finally, since $\dom(\mu_{k-1}) \cup \rge(\mu_{k-1}) \subseteq \dom(\mu_{k-2})$, we have that $\gamma_{k-1} \in \dom(f_\alpha^q)$ as required.

\end{proof}

Let $G$ be a generic object for $\mathbb{P}$.
Fix a limit ordinal $\beta \leq \eta$. 
Note that the collection of $p^\beta$ satisfying the initial condition in Lemma \ref{scaledef} is open dense, we may assume that $p^\beta \in G$.
From now on, in $V[G]$, fix $\lambda_\beta$ and $\theta_\beta$ as in Lemma \ref{scaledef}, which are defined from $p^\beta$.
Note that by genericity
of $G$, $\lambda_\beta$ and $\theta_\beta$ is well-defined.

For $\gamma \in [\overline{\kappa}_\beta,\theta_\beta)$, define a function $F^\beta_\gamma:\beta \rightarrow \bar{\kappa}_\beta$ as follows: if $\alpha \leq \bar{\alpha}$, let $F^\beta_\gamma(\alpha)=0$.
 Assume $\gamma<\lambda_{\alpha,\beta}^j$, find $p \in G$ with $p$ lying in the dense open set below $p^\beta$ from Lemma \ref{scaledef}.
  Enumerate $\supp(p) \cap (\beta+1) \setminus \alpha$ in decreasing order as $\alpha_0> \dots >\alpha_{k-1}=\alpha$.
Define $\gamma_0, \dots ,\gamma_{k-1}$ as in Lemma \ref{scaledef}, and define $F^\beta_\gamma(\alpha)=f^p_\alpha(\gamma_{k-1})$.
If $\beta=\eta$, we remove the superscript $\beta$ and just write $F_\gamma$ instead of $F_\gamma^\beta$.
Let us visualize the definition of $F_\gamma(\alpha)$ in Figure \ref{scalepic}, where $\beta=\eta$.

 \begin{figure}[H]
 \centering
 \captionsetup{justification=centering} 
 \includegraphics[scale = 0.65]{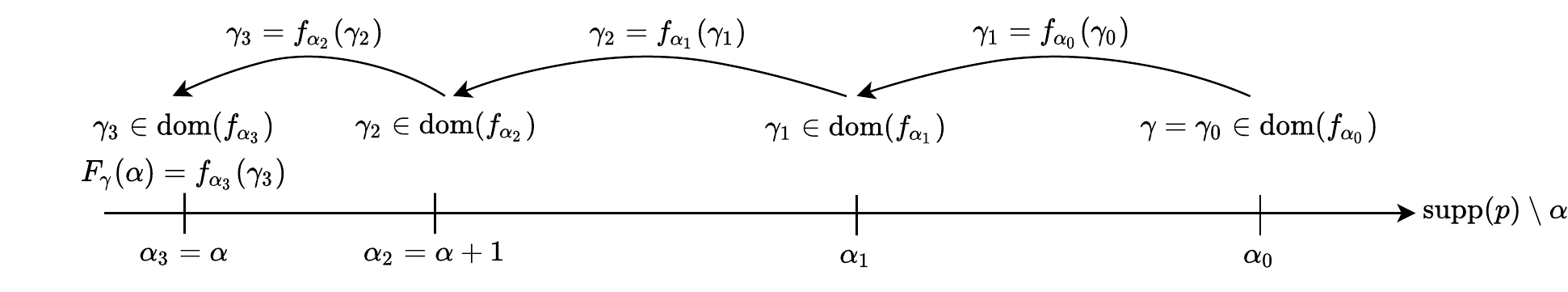}
  \caption{The definition of $F_\gamma(\alpha)$ when $| \supp(p) \setminus \alpha |=4$}
  \label{scalepic}
\end{figure}

In Figure \ref{scalepic}, we assume that in Lemma \ref{scaledef}, $\beta=\eta$ and $p \in D \cap G$.
Each $p_\alpha$ is either $\langle f_\alpha \rangle$ or $\langle f_\alpha,A_\alpha \rangle$.
We also assume $\gamma< \lambda_\alpha^j$ and assume that a decreasing enumeration of $\supp(p) \setminus \alpha$ is $\alpha_0>\alpha_1>\alpha_2>\alpha_3=\alpha$ (where $\alpha_2=\alpha+1$).

To check that $F^\beta_\gamma$ is well-defined, we consider only the case $\eta=\beta$.
Suppose $p,q \in G$ satisfy the conditions in Lemma \ref{scaledef}. 
Find $r \in G$ with $r \leq p,q$. 
Hence, $((\supp(p) \cup \supp(q)) \cap \beta) \setminus \alpha \subseteq (\supp(r) \cap \beta) \setminus \alpha$.
Assume $r \leq^* p+ \langle \mu_0, \dots, \mu_{l-1} \rangle$ and $r \leq^* q+ \langle \tau_0, \dots, \tau_{l-1} \rangle$. 
For simplicity, assume $\mu_i$ is a $\beta_i$-object, $\tau_j$ is a $\zeta_j$-object, $\alpha<\beta_0< \dots < \beta_{l-1}$ and $\alpha<\zeta_0< \dots <\zeta_{m-1}$. 
We will show that $p$ and $r$ compute the same $F_\gamma(\alpha)$-value. A similar argument will show that $q$ computes the same $F_\gamma(\alpha)$ as $r$. We simplify further that $l=1$, $\mu=\mu_0$.

We see that the value $F_\gamma(\alpha)$ computed by $p$ is $f_{\alpha_{k-1}}^p \circ \dots f_{\alpha_n}^p \circ f_{\alpha_{n-1}}^q \circ \dots f_{\alpha_0}^p(\gamma)$.

\textbf{\underline{CASE I}} $\xi>\alpha_0$:
As in the proof of Lemma \ref{scaledef}, $\mu$ fixes $f_{\alpha_0}^p(\gamma)$ and hence,

\begin{align*}
f_{\alpha_{k-1}}^r \circ \dots \circ f_{\alpha_0}^r \circ f_\xi^r(\gamma) &= f_{\alpha_{k-1}}^p \circ \dots \circ f_{\alpha_1}^p \circ \mu \circ f_{\alpha_0}^p \circ \mu^{-1} \circ \mu(\gamma) \\
& = f_{\alpha_{k-1}}^p \circ \dots \circ f_{\alpha_1}^p \circ \mu \circ f_{\alpha_0}^p(\gamma) \\
&=  f_{\alpha_{k-1}}^p \circ \dots \circ f_{\alpha_1}^p  \circ f_{\alpha_0}^p(\gamma).
\end{align*}

\textbf{\underline{CASE II}} $\xi <\alpha_0$:
Let $n$ be the least such that $\xi<\alpha_n$.
As in the proof of Lemma \ref{scaledef}, $\mu$ fixes $f_{\alpha_{n+1}}^p(\gamma_{n+1})$, where $\gamma_{n+1} =f_{\alpha_n}^p \circ f_{\alpha_{n-1}}^p \circ \dots \circ f_{\alpha_0}^p(\gamma)$, so 

\begin{align*}
f_{\alpha_{k-1}}^r \circ \dots f_{\alpha_{n+1}}^r \circ f_\xi^r \circ f_{\alpha_n}^r \circ \dots \circ f_{\alpha_0}^r(\gamma)
&= f_{\alpha_{k-1}}^p \circ \dots \circ \mu \circ f_{\alpha_{n+1}}^p \circ \mu^{-1} \circ \mu \circ f_{\alpha_n}^p \circ \dots \circ f_{\alpha_0}^p(\gamma) \\
&=f_{\alpha_{k-1}}^p \circ \dots \circ \mu \circ  f_{\alpha_{n+1}}^p  \circ f_{\alpha_n}^p \circ \dots \circ  f_{\alpha_0}^p(\gamma) \\
& =f_{\alpha_{k-1}}^p \circ \dots f_{\alpha_{n+1}}^p \circ f_{\alpha_n}^p \circ \dots \circ f_{\alpha_0}^p(\gamma).
\end{align*}

Thus, $p$ and $r$ compute the same $F_\gamma(\alpha)$.

\begin{rmk}

If $\gamma<j_{E_\alpha}(\kappa_{\alpha+1})$, then the requirements for a condition $p \in G$ that is used to compute $F_\gamma(\alpha)$ can be weakened a bit: the requirement that $\alpha+1 \in \supp(p)$ is  not necessary.
In practice, if $\gamma<j_{E_0}(\kappa_1)$, then a scale analysis will be slightly simpler.
Here is a reason.
Suppose we start with a condition $p$ with $\max(\supp(p))<\alpha$, $\gamma \in d_\alpha^p$.
Let $\mu_\alpha \in A_\alpha^p$ and $\mu_{\alpha+1}^p$ be such that $\gamma \in \dom(\mu_\alpha)$, $\dom(\mu_\alpha) \cup \rge(\mu_\alpha) \subseteq \dom(\mu_{\alpha+1})$.
Then $\mu_\alpha(\gamma)<\mu_\alpha(j_{E_\alpha}(\kappa_{\alpha+1}))=\kappa_{\alpha+1}$.
Hence $\mu_{\alpha+1} \circ \mu_\alpha \circ \mu_{\alpha+1}^{-1} \circ \mu_{\alpha+1} (\gamma)=\mu_\alpha(\gamma)$.
This concludes that $p+\mu_\alpha$ and $p+\langle \mu_\alpha,\mu_{\alpha+1} \rangle$ compute the same $F_\gamma(\alpha)$-value.

\end{rmk}

\begin{prop}\label{unbdd} 
In $V[G]$, $\langle F^\beta_\gamma : \gamma \in[\bar{\kappa}_\beta,\theta_\beta)\rangle $ is $<_{bd}$- increasing, where for each pair of ordinal functions $t,t^\prime$ with domains limit ordinal $\theta$,  $t<_{bd} t^\prime$ means there is $\alpha<\theta$ such that for all $\alpha^\prime>\alpha$, $t(\alpha^\prime)<t^\prime(\alpha^\prime)$.
\end{prop}

\begin{proof}
  We prove the case $\beta=\eta$. The case $\beta<\eta$ is similar. Let $\gamma,\gamma^\prime \in [\overline{\kappa}_\eta,\bar{\lambda}_\eta^j)$, $\gamma<\gamma^\prime$.
  Let $\xi<\eta$ be such that $\gamma^\prime < \lambda_\xi^j$.
    We will show that $F_\gamma <_{bd} F_{\gamma^\prime}$ by a density argument. Let $p\in \mathbb{P}$. We can find $p^\prime \leq p$ and $\xi_0 \geq \xi$ such that $\max(\supp(p^\prime)) \leq \xi_0$, $\gamma$ and $\gamma^\prime$ are in the domains of the $\alpha$th Cohen parts and measure-one parts of $p$ for all $\alpha >\xi_0$.

We can also assume that for $\alpha>\alpha_0$, the domain of each object in
$A_\alpha^{p^\prime}$ contains $\gamma$, and $\gamma^\prime$. We will show
\begin{center}
    $p^\prime \Vdash \forall \alpha>\xi_0(\dot{F}_\gamma(\alpha)<\dot{F}_{\gamma^\prime}(\alpha))$.
\end{center}
This is true because for each $\alpha>\xi_0$, if $q\leq p^\prime$ can be used to define $F_\gamma(\alpha)$ and $F_{\gamma^\prime}(\alpha)$, the functions used to compute their values computed (from $q$) are just compositions of objects.  Every object is order-preserving, and by a density argument, we are done.
\end{proof}

By Proposition \ref{unbdd}, we conclude that in $V[G]$, $2^{\bar{\kappa}_\eta} \geq |\bar{\lambda}_\eta^j|$.
Recall that $|\lambda_\alpha^j|^V$ is regular.
If $\langle |\lambda_\alpha^j|^V:\alpha<\eta \rangle$ is not constant on the tail, then by a classical König's result on cardinal arithmetic, $2^{\bar{\kappa}_\eta} >| \bar{\lambda}_\eta^j|$.
With GCH, $|\mathbb{P}|=|\lambda_\eta^j|$.
By the nice name theorem, $(2^{\bar{\kappa}_\eta})^{V[G]} \leq ||\mathbb{P}|^{\lambda^+ \otimes \bar{\kappa}_\eta}|^V$.
We now have two cases: the case where $\langle |\lambda_\alpha^j|^V: \alpha<\eta \rangle$ is eventually constant, and the case where $\langle |\lambda_\alpha^j|^V: \alpha<\eta \rangle$ is not eventually constant.
For the first case, $\langle |\lambda_\alpha^j|^V: \alpha<\eta \rangle$ is eventually constant, say the constant value is $\theta$.
Then $||\mathbb{P}|^{\lambda^+ \otimes \bar{\kappa}_\eta}|^V=\theta=\sup_{\alpha<\eta} |\lambda_\alpha^j|$.
Hence $2^{\bar{\kappa}_\eta}=\sup_{\alpha<\eta} |\lambda_\alpha^j|$.
For the second case, we have that  $||\mathbb{P}|^{\lambda^+ \otimes \bar{\kappa}_\eta}|^V=|\bar{\lambda}_\eta^j|^+$.
Hence $2^{\bar{\kappa}_\eta}=(\sup_{\alpha<\eta} |\lambda_\alpha^j|)^+$.
We now conclude the cardinal arithmetic of $2^{\bar{\kappa}_\eta}$.

\begin{thm}

By forcing with $\mathbb{P}_{\langle E_\alpha: \alpha<\eta \rangle}$, where $\eta$ is regular, we have that

\begin{enumerate}

\item if $\langle |\lambda_\alpha^j|^V: \alpha<\eta \rangle$ is eventually constant,  then in $V[G]$, $2^{\bar{\kappa}_\eta}=|\bar{\lambda}_\eta^j|$.

\item if $\langle |\lambda_\alpha^j|^V: \alpha<\eta \rangle$ is not eventually constant, then in $V[G]$, $2^{\bar{\kappa}_\eta}=|\bar{\lambda}_\eta^j|^+$.

\end{enumerate}

\end{thm}

\section{scale analysis}
\label{scaleanalysis}

First we define a notion of scales.

\begin{defn}
\label{defscale}
Let $\rho$ be a singular cardinal and $\xi>\rho$.
A {\em scale} on $\rho$ of length $\xi$ is a family of functions $\vec{f}=\langle f_\alpha: \alpha<\xi \rangle$, together with $\vec{\rho}=\langle \rho_\beta: \beta<\cf(\rho) \rangle$ such that 

\begin{enumerate}

\item $\vec{\rho}$ is an increasing sequence of regular cardinals, cofinal in $\rho$, and $\rho_0>\cf(\rho)$.

\item $f_\alpha \in \prod\limits_{\beta<\cf(\rho)} \rho_\beta$.

\item for $\alpha_0<\alpha_1<\xi$, $f_{\alpha_0} <_{bd} f_{\alpha_1}$, (recall that this means there exists $\beta_0<\cf(\rho)$ such that for $\beta>\beta_0$, $f_{\alpha_0}(\beta)<f_{\alpha_1}(\beta)$).
A $\vec{f}$ satisfying this property is said to be {\em $<_{bd}$-increasing}.

\item $\vec{f}$ is cofinal, namely for any $h \in \prod\limits_{\beta<\cf(\rho)} \rho_\beta$, there is an $\alpha<\xi$ such that $h<_{bd} f_\alpha$.

\end{enumerate}

\end{defn}

In Definition \ref{defscale}, it is oftenly enough to define $f_\alpha(\beta)$ for all sufficiently large $\beta$, since the scale properties only refer to how functions are on their tails.

\begin{defn}
Let $\vec{f}=\langle f_\alpha: \alpha<\xi \rangle$ be a scale.
Then $f$ is an {\em exact upper bound (eub)} for $\vec{f}$ if

\begin{enumerate}

\item $f$ is an {\em upper bound} for $\vec{f}$, i.e. for $\alpha<\xi$, $f_\alpha<_{bd} f$.

\item If $g <_{bd} f$ then there is an $\alpha$ such that $g<_{bd} f_\alpha$.

\end{enumerate}

\end{defn}

Note that if $f$ is an eub of $\vec{f}$, $f(\alpha)$ is regular for sufficiently large $\alpha$, and $f$ is increasing on the tail, then $\vec{f}$ is a scale on $\prod\limits_{\alpha<\eta} f(\alpha)$.

\begin{defn}

For a scale $\langle f_\alpha: \alpha<\rho^+ \rangle$ on $\prod\limits_{\beta<\cf(\rho)}\rho_\beta$ where $\rho=\sup_{\beta<\cf(\rho)} \rho_\beta$, we say that an ordinal $\alpha<\rho^+$ with $\cf(\alpha)>\cf(\rho)$ is {\em very good} if there are a club $C \subseteq \alpha$ of order-type $\cf(\alpha)$ and an ordinal $\beta_0<\cf(\rho)$ such that for $\alpha_0<\alpha_1<\xi$ and $\beta>\beta_0$, $f_{\alpha_0}(\beta)<f_{\alpha_1}(\beta)$.
A scale for a singular cardinal $\rho$ is very good if every $\alpha$ with $\cf(\alpha)>\cf(\rho)$ is very good.

\end{defn}

We follow our definitions from Section \ref{blowinguppowersets}.
Namely, we have the ordinal functions $F_\gamma^\beta$ as well as $F_\gamma$.
We now make some scale analysis.
Note that if $\beta \leq \eta$ is limit, we have that $(\lambda_\beta^+)^V=(\bar{\kappa}_\beta^+)^{V[G]}$.
In particular, $(\lambda^+)^V=(\bar{\kappa}_\eta^+)^{V[G]}$.

\begin{prop}
\label{scale2}
In $V[G]$, for $\beta \leq \eta$ limit.
Let $\xi<\theta_\beta$ be such that $\cf^{V[G]}(\xi)>\bar{\kappa}_\beta$.
Then $F^\beta_\xi$ is an eub for $\langle F^\beta_\gamma: \gamma \in [\bar{\kappa}_\beta,\xi) \rangle$.

\end{prop}

\begin{proof}

We have already checked in Proposition \ref{unbdd} that the sequence is $<_{bd}$-increasing.
It remains to show that the sequence is cofinal.
We will only show the case $\beta=\eta$.
Assume for simplicity that $\xi<\lambda_0^j$.

Clearly $F_\xi$ is an upper bound of $\langle F_{\xi^\prime}:\xi^\prime<\xi \rangle$.
Let $\dot{h}$ be a $\mathbb{P}$-name such that $p \Vdash \dot{h} \in \prod_{\alpha<\eta} \dot{F}_\xi(\alpha)$.
We are going to build a $\leq^*$-decreasing sequence $\langle q^\alpha: \alpha<\eta$ or $\alpha=-1 \rangle$.
Let $q^{-1}=p$. 
Suppose $q^\beta$ is built for all $\beta \in\alpha \cup \{-1\}$.
Let $(q^\alpha)^*$ be a $\leq^*$-lower bound of $\langle q^\beta: \beta \in \alpha \cup \{-1\} \rangle$.
Fix $\alpha<\eta$.
Let $D_\alpha=\{q: q$ decides $\dot{h}(\alpha)\}$.
By the strong Prikry property, find $q \leq^* (q^\alpha)^*$ witnessing the strong Prikry property of $D_\alpha$ with the corresponding finite set $I_\alpha$.
We assume that $\alpha,\alpha+1 \in I_\alpha$.
For $\vec{\mu} \in \prod\limits_{\gamma \in I_\alpha} A_\gamma^q$, write $\vec{\mu}=\langle \mu_\gamma \rangle_{\gamma \in I_\alpha}$. Define $Y_\alpha(\vec{\mu})$ as the value that $q+\vec{\mu}$ decides for $\dot{h}(\alpha)$.
It is easy to check that the decided value is less than the value $F_\xi(\alpha)$ computed by $q+\vec{\mu}$, which is $\mu_{\alpha+1}(\mu_\alpha(\xi))$.
For $\vec{\mu} \in \prod_{\gamma \in I_\alpha \setminus \alpha} A_\gamma^q$, define $Z_\alpha(\vec{\mu})=\sup_{\vec{\tau}}(Y_\alpha(\vec{\tau}^\frown \vec{\mu}))+1$.
Note that the number of possible such $\vec{\tau}$ is $\lambda_\alpha(\mu_\alpha)=s_\alpha(\mu_\alpha(\kappa_\alpha))$.
By our assumption on cofinality of $\xi$, we may assume that $\lambda_\alpha(\mu_\alpha)<\cf(\mu_{\alpha+1}(\mu_\alpha(\xi))$ (this can be done on a measure-one sets), so $Z_\alpha(\vec{\mu})<\mu_{\alpha+1}(\mu_\alpha(\xi))$ and for $\vec{\mu} \in \prod\limits_{\gamma \in I_\alpha \setminus \alpha} A_\gamma^q$, $q+\vec{\mu} \Vdash \dot{h}(\alpha)<Z_\alpha(\vec{\mu})$.
For $\gamma>\alpha+1$, let $A_\gamma^\prime=A_\gamma^{j_{E_{\alpha+1}}(q)+\mc_{\alpha+1}(d_{\alpha+1}^q)}$.
Then for $\vec{\mu} \in \prod\limits_{\gamma \in I_\alpha \setminus (\alpha+2)} A_\gamma^\prime$ and $\psi \in A_\alpha^q$ (recall that $A_\alpha^q=\mc_{\alpha+1}(d_{\alpha+1}^q) \circ (j_{E_{\alpha+1}}(A_\alpha^q))_{\mc_{\alpha+1}} \circ \mc_{\alpha+1}^{-1}(d_{\alpha+1}^q)$ as shown in Lemma \ref{aux}),

\begin{align*}
j_{E_{\alpha+1}}(q)+\langle \mc_{\alpha+1}(d_{\alpha+1}^q)^\frown \psi {}^\frown \vec{\mu} \rangle \Vdash j_{E_{\alpha+1}}(\dot{h}(\alpha))<\mc_{\alpha+1}(d_{\alpha+1}^q)(j_{E_\alpha+1}(\psi)j_{E_\alpha+1}(\xi)).
\end{align*}

Notice that the order of the objects is not increasing: $\mc_{\alpha+1}$ is an $\alpha+1$-object, while $\psi$ is an $\alpha$-object.
Also, $j_{E_{\alpha+1}}(q)+\langle \mc_{\alpha+1}(d_{\alpha+1}^q)^\frown \psi \rangle=j_{E_{\alpha+1}}(q)+\langle j_{E_{\alpha}}(\psi)^\frown \mc_{\alpha+1}(d_{\alpha+1}^q)\rangle$, which is why the value at the rightmost of the forcing relation is $\mc_{\alpha+1}(d_{\alpha+1}^q)(j_{E_\alpha+1}(\psi)j_{E_\alpha+1}(\xi))$.
Note that the last value in the forcing relation, which is $\mc_{\alpha+1}(d_{\alpha+1}^q)(j_{E_{\alpha+1}}(\psi)j_{E_{\alpha+1}}(\xi))$, is equal to $\psi(\xi)$.
Note that $\psi(\xi)<\lambda$.
Fix $\psi \in A_\alpha^q$.
For $\gamma \in I_\alpha \setminus (\alpha+2)$, $j_{E_{\alpha+1}}(A_\gamma^q)$ comes from a measure which is $j_{E_{\alpha+1}}(\kappa_\gamma)$-complete, and $j_{E_{\alpha+1}}(\kappa_\gamma)>\kappa_\alpha^j>\lambda$, we can inductively shrink $A_\gamma^*$ for $\gamma\in I_\alpha \setminus(\alpha+1)$ to $B_{\gamma,\psi}$ so that the following holds: there is $\gamma_{\alpha,\psi}<\psi(\xi)$ such that for all $\vec{\mu} \in \prod\limits_{\gamma \in I_\alpha \setminus (\alpha+2)} B_{\gamma, \psi}$,

\begin{align*}
j_{E_{\alpha+1}}(q)+\langle \mc_{\alpha+1}^\frown \psi{}^\frown \vec{\mu} \rangle \Vdash j_{E_{\alpha+1}}(\dot{h}(\alpha))=\gamma_{\alpha,\psi},
\end{align*}

which is equivalent to saying that for $\psi \in j_{E_{\alpha+1}}[A_\alpha^q]$, and $\vec{\mu} \in \prod\limits_{\gamma \in I_\alpha \setminus (\alpha+2)} B_{\gamma, \psi}$, 

\begin{align*}
j_{E_{\alpha+1}}(q)+\langle \psi^\frown \mc_{\alpha+1}{}^\frown \vec{\mu} \rangle \Vdash j_{E_{\alpha+1}}(\dot{h}(\alpha))=\gamma_{\alpha,j_{E_{\alpha+1}}^{-1}(\psi)}
\end{align*}
Let's replace the notation $\gamma_{\alpha,j_{E_{\alpha+1}}^{-1}(\psi)}$ by $\gamma_{\alpha,\psi}$.
Fix $\psi \in A_\alpha^q$.
By elementarity, there are measure-one sets $A_{\gamma,\psi} \subseteq A_\gamma^{q+\psi}$ for $\gamma \in I_\alpha \setminus (\alpha+1)$ such that for $\vec{\mu} \in \prod\limits_{\gamma \in I_\alpha \setminus (\alpha+1)} A_{\gamma,\psi}$, 

\begin{align*}
q+\langle \psi^\frown \vec{\mu} \rangle \Vdash \dot{h}(\alpha)=\mu_{\alpha+1}(\gamma_{\alpha,\psi}).
\end{align*}

For $\gamma \in I_\alpha \setminus (\alpha+1)$, let $A_\gamma^*=\triangle_{\psi \in A_\alpha^q} A_{\gamma,\psi}$.
Shrink all measure-one sets $A_\gamma^q$ for such $\gamma$ in $q$ to $A_\gamma^*$ and call the new condition $q^\prime$.
Let $\gamma_\alpha=j_{E_\alpha}(\psi \mapsto \gamma_{\alpha,\psi})(\mc_\alpha(d_\alpha^q))$.
Then $\gamma_\alpha<\lambda_0^j$.
Extend $q^\prime$ to $q^*$ so that $\gamma_\alpha \in d_\alpha^{q^*}$.
Then note that

\begin{align*}
j_{E_\alpha}(\tau \mapsto \tau(\gamma_\alpha))(\mc_\alpha(d_\alpha^{q^*}))&=\gamma_\alpha\\
&=j_{E_\alpha}(\psi \mapsto \gamma_{\alpha,\psi})(\mc_\alpha(d_\alpha^q))\\
&= j_{E_\alpha}(\tau \mapsto \gamma_{\alpha,\tau \restriction d_\alpha^q})(\mc_\alpha(d_\alpha^{q^*})).
\end{align*}

Then there is a measure-one set $A_\alpha^* \in E_\alpha(d_\alpha^{q^*})$ such that for $\tau \in A_\alpha^*$, $\tau(\gamma_\alpha)=\gamma_{\alpha,\tau \restriction d_\alpha^q}$.
Shrink the $\alpha$-th measure-one set to in $q^*$ to $A_\alpha^*$ and call the final condition $q^\alpha$.
Replace the notation $\gamma_{\alpha,\tau \restriction d_\alpha^q}$ by $\gamma_{\alpha,\tau}$.
We see that for $\tau \in A_\alpha^{q^\alpha}$, $\mu \in A_{\alpha+1}^{q^\alpha}$, and $\vec{\sigma} \in \prod_{\gamma \in I_\alpha \setminus (\alpha+1)} A_\gamma^{q^\alpha}$, we have

\begin{align*}
q^\alpha+\langle \tau^\frown \mu {}^\frown \vec{\sigma} \rangle \Vdash \dot{h}(\alpha)=\mu(\tau(\gamma_\alpha))=\dot{F}_{\gamma_\alpha}(\alpha).  
\end{align*}

Let $r$ be a $\leq^*$-lower bound of $\langle q^\alpha: \alpha<\eta \rangle$, and $\gamma^*=\sup_\alpha \gamma_\alpha$.
This is less than $\xi$ since $\cf(\xi)>\eta$.

\begin{claim}
\label{geteub}
There is $s \leq^*r$ such that $s \Vdash \forall \alpha \forall \beta<\eta (\dot{F}_{\gamma_\beta}(\alpha)<\dot{F}_{\gamma^*}(\alpha))$.

\end{claim}

\begin{claimproof}{(Claim \ref{geteub})}

Note that we assume $\xi<\lambda_0^j$.
Let $s \leq^*r$ be such that $\gamma_\alpha$ for all $\alpha<\eta$,  and $\gamma^*$, are in $d_0^s$ and for $\beta<\eta$, $\mu \in A_\beta^s$, we assume that $\{\gamma_\alpha: \alpha<\eta\} \cup \{\gamma^*\} \subseteq \dom(\mu)$.
Since objects are order-preserving, we have that $s \Vdash \forall \alpha<\eta \forall \beta<\eta (\dot{F}_{\gamma_\beta}(\alpha)<\dot{F}_{\gamma^*}(\alpha))$.

\end{claimproof}{(Claim \ref{geteub})}

We now show that $s \Vdash \dot{h}<_{bd} \dot{F}_{\gamma^*}$.
Fix an $\alpha<\eta$.
Let $s^\prime \leq s$.
By the density, assume that $I_\alpha \subseteq \supp(s^\prime)$, so $s^\prime$ decide $\dot{h}(\alpha)$.
Note that $s^\prime \leq q^\alpha +\vec{\mu}$ for some $\vec{\mu} \in \prod\limits_{\gamma \in I_\alpha} A_\gamma^{q^\alpha}$, so $s^\prime \Vdash \dot{h}(\alpha)=\dot{F}_{\gamma_\alpha}(\alpha)<\dot{F}_\gamma(\alpha)$.

\end{proof}

\begin{coll}
\label{beingscale}
In $V[G]$, for $\beta \leq \eta$ limit.
Let $\xi \in (\lambda_\beta,\theta_\beta)$ be regular.
Then $\langle F^\beta_\gamma: \gamma \in [\bar{\kappa}_\beta,\xi) \rangle$ is a scale in $\prod_{\alpha<\eta} F_\xi^\beta(\alpha)$.

\end{coll}

In Collorary \ref{beingscale}, we need that $\xi>\lambda_\beta$ because all cardinals in the interval $(\bar{\kappa}_\beta,\lambda_\beta]$ are collapsed in $V[G]$ (as in Theorem \ref{cardinalstructure}).
We may read off $F_\xi^\beta$ for $\xi<\kappa_0^j$ from a generic object in a simpler way if we assume that for $\alpha<\eta$, there is $v_\alpha: \kappa_\alpha \to \kappa_\alpha$ such that $j_{E_\alpha}(v_\alpha)(\kappa_\alpha)=\xi$.
Then one can verify that in the extension, $F_\xi(\alpha)=v_\xi(f_\alpha^p(\kappa_\alpha))$ for some $p \in G$ with $\alpha \in \supp(p)$.
In particular, if $\xi \in (\lambda,\kappa_0^j)$ is regular, then $\langle F_\gamma: \gamma \in [\bar{\kappa}_\beta,\xi) \rangle$ is a scale in $\prod\limits_{\alpha<\eta} v_\xi(f^p_\alpha(\kappa_\alpha))$ for $p \in G$ with $\alpha \in \supp(p)$.
Note that since $j_{E_\alpha}(s_\alpha)(\kappa_\alpha)=\lambda$ and ${}^{\lambda} \Ult(V,E_\alpha) \subseteq \Ult(V,E_\alpha)$, the function $s_\alpha^\prime:\kappa_\alpha \to \kappa_\alpha$ defined by $s_\alpha^\prime(\gamma)=s_\alpha(\gamma)^+$ represents $\lambda^+$ in $\Ult(V,E_\alpha)$.
Hence $\langle F_\gamma \in [\bar{\kappa}_\eta,(\lambda^+)^V) \rangle$ is a scale in $\prod\limits_{\alpha<\eta} s_\alpha(f_\alpha(\kappa_\alpha))^+$ when $f_\alpha=f_\alpha^p$ for some $p \in G$ with $\alpha \in \supp(p)$.
A similar argument explains a situation for limit $\beta<\eta$.
Thus we have the following.

\begin{prop}
\label{verygoodpt}
In $V[G]$, for $\beta \leq \eta$ limit, $\langle F^\beta_\gamma: \gamma \in [\bar{\kappa}_\beta,(\bar{\kappa}_\beta^+)^{V[G]})\rangle$ is a scale in $\prod\limits_{\alpha<\eta} s_\alpha(f_\alpha(\kappa_\alpha))^+$.
Every $\gamma \in (\bar{\kappa}_\beta,(\bar{\kappa}_\beta^+)^{V[G])} \rangle$ with $\beta<\cf^V(\gamma)<\bar{\kappa}_\beta$ is very good.
As a consequence, the scale is very good.

\end{prop}

\begin{proof}

We prove only the case $\beta=\eta$.
Recall that $(\lambda^+)^V=(\bar{\kappa}_\eta^+)^{V[G]}$.
In $V$, let $\gamma \in [\bar{\kappa}_\eta,\lambda^+)$ with $\eta<\cf^{V[G]}(\gamma)<\bar{\kappa}_\eta$.
By Theorem \ref{cardinalstructure}, we have that $\cf^V(\gamma)<\bar{\kappa}_\eta$.
Assume $\cf(\gamma)<\kappa_{\alpha_0}$.
Let $C \subseteq \gamma \setminus \bar{\kappa}_\eta$ be a club of order-type $\cf(\gamma)$.
Extend $p$ to $q$ such that $C \subseteq d_{\alpha_0}^q$, and for $\alpha \geq \alpha_0$ and $\mu \in A_\alpha^q$, $C \subseteq \dom(\mu)$.
It is now easy to see that $q \Vdash \forall \beta_0,\beta \in C(\beta_0<\beta_1 \implies \forall \alpha \geq \alpha_0(F_{\beta_0}(\alpha)<F_{\beta_1}(\alpha)))$.

\end{proof}

\section{conclusion}
\label{conclusion}
We prove the following theorem:

\begin{thm}
\label{conclude}

Assume GCH.
Assume the result of Theorem \ref{deriveext}, where the length of the sequence of extenders is a regular cardinal $\eta$ (including $\omega$).
Assume that in the context of Theorem \ref{deriveext}, $|j_{E_\alpha}(\lambda)|^V$ is regular for all $\alpha<\eta$.
There is a $\lambda^{++}$-c.c. forcing notion $\mathbb{P}=\mathbb{P}_{\langle E_\alpha: \alpha<\eta \rangle }$ such that

\begin{enumerate}

\item $\mathbb{P}$ satisfies the Prikry property and the strong Prikry property.

\item let $G$ be $\mathbb{P}$-generic.
For limit $\alpha<\eta$, let $\lambda_\alpha=s_\alpha(f_\alpha^p(\kappa_\alpha))$ where $j_{E_\alpha}(s_\alpha)(\kappa_\alpha)=\lambda$.
All cardinals in the intervals $(\bar{\kappa}_\alpha,\lambda_\alpha^+)$, for $\alpha<\eta$ limit, and $(\bar{\kappa}_\eta,\lambda^+)$ are collapsed.
Other cardinals are preserved.

\item If $\langle |j_{E_\alpha}(\lambda)|^V: \alpha<\eta \rangle$ is eventually constant, then $2^{\bar{\kappa}_\eta}=\sup_{\alpha<\eta}|j_{E_\alpha}(\lambda)|$, otherwise $2^{\bar{\kappa}_\eta}=(\sup_{\alpha<\eta}|j_{E_\alpha}(\lambda)|)^+$.

\item Fix a limit ordinal $\beta<\eta$, $p \in G$ and $\beta \in \supp(p)$.
For $\alpha<\beta$, let $\lambda_{\alpha,\beta}^j=t_\beta^\alpha(f_\beta^p(\kappa_\beta))$.
Then if $\langle |\lambda_{\alpha,\beta}^j|^V: \alpha<\beta \rangle$ is eventually constant, then $2^{\bar{\kappa}_\beta}=\sup_{\alpha<\beta} |\lambda_{\alpha,\beta}^j|$.
Otherwise, $2^{\bar{\kappa}_\beta}=(\sup_{\alpha<\beta} |\lambda_{\alpha,\beta}^j|)^+$.

\end{enumerate}

\end{thm}

We point out an interesting phenomenon, thanks to a discussion with Gitik.
If $\kappa$ is a singular cardinal in a ground model, then it is possible to preserve $\kappa$, collapse $\kappa^+$, and preserve all cardinals above $\kappa^+$.
Here is our explanation.
When we set $\lambda=(\sup_{\alpha<\eta}\kappa_\alpha)^+$ in Theorem \ref{conclusion}, then the only cardinal above $\sup_{\alpha<\eta} \kappa_\alpha$ which is collapsed is $(\sup_{\alpha<\eta} \kappa_\alpha)^+$.
All the cardinals above are preserved in an extension.

\begin{rmk}

\begin{enumerate}

\item The assumption for Theorem \ref{conclude} that $|\lambda_\alpha^j|^V$ is regular can be relaxed.
We actually just need $\cf(|\lambda_\alpha^j|^V) > \lambda^+$ to obtain the same cardinal arithmetics.
It is still possible to weaken the GCH assumption to build the forcing, and the value of $2^{\bar{\kappa}_\eta}$ will depend on the cardinal arithmetic assumption assumed in the ground model.

\item The functions $u_\alpha^\beta$ for $\beta \leq \alpha$ are not necessary to define the forcing at all,
In fact, the functions $s_\alpha,t_\alpha^\beta,u_\alpha^\beta$ defined at the beginning of Section \ref{analysisofextenders} are not necessary.
Here is the reason: for each $\alpha$-object $\mu$, we have $s_\alpha(\mu(\kappa_\alpha))=\mu(\lambda)$, $t_\alpha^\beta(\mu(\kappa_\alpha))=\mu(\lambda_\beta^j)$ and $u_\alpha^\beta(\mu(\kappa_\alpha))=\mu(\bar{\lambda}_\beta^j)$.
Instead of imposing the requirements for $\gamma$ to be $\alpha$-reflected in Definition \ref{reflect}, we can add more corresponding requirements to $\mu$ being an $\alpha$-object in Definition \ref{alphaobject}, for example,  since $\lambda=j_{E_\alpha}(\mu \mapsto \mu(\lambda))(\mc_\alpha(d_\alpha))$, we can say that every $\alpha$-object $\mu$ is such that $\mu(\lambda)$ is regular.
Hence, Lemma \ref{woodinscc1} is enough to build a desired sequence of extenders that is enough to build the forcing.
However, definability of important cardinals helps keeping track of the reflection phenomenon in a slightly easier way.

\end{enumerate}

\end{rmk}

\bibliographystyle{ieeetr}
\bibliography{supercompactsquishedforcingreference}

\textsc{School of Mathematical Sciences, Tel Aviv University, Tel Aviv-Yafo,  Israel, 6997801}
\newline
\textit{E-mail address}:
\texttt{jir.sittinon@gmail.com}

\end{document}